\documentclass[11pt]{article}

\usepackage{amsfonts}
\usepackage{amscd}
\usepackage{amssymb}
\usepackage{amsthm}
\usepackage{amsmath, xspace}
\usepackage{blkarray}
\usepackage{cancel}
\usepackage{color}
\usepackage{graphics}
\usepackage{graphicx}
\usepackage{enumitem}
\usepackage{fancyhdr}
\usepackage{mathdots}
\usepackage{mathrsfs}
\usepackage{multicol}
\usepackage{stmaryrd}
\usepackage{ytableau}
\usepackage[all]{xy}

\usepackage[plainpages,backref]{hyperref}

\usepackage{lscape}

\theoremstyle{plain}
\newtheorem{thm}{Theorem}[section]
\newtheorem{lemma}[thm]{Lemma}
\newtheorem{prop}[thm]{Proposition}
\newtheorem{cor}[thm]{Corollary}
\theoremstyle{definition}
	
\newtheorem{remark}[thm]{Remark}
\theoremstyle{example}

\theoremstyle{remark}

\numberwithin{equation}{section}

\providecommand{\keywords}[1]{\textbf{\textit{Key words---}} #1}

\def\cB{\mathcal{B}}

\def\CC{\mathbb{C}}

\def\ZZ{\mathbb{Z}}

\def\fa{\mathfrak{a}}

\def\ev{\mathrm{ev}}

\def\id{\mathrm{id}}

\def\wt{\mathrm{wt}}

\usepackage{tikz}
	\usepgflibrary[patterns] 
	\usetikzlibrary{patterns} 
	\usepgflibrary{shapes.geometric}

\newcommand{\TikZ}[1]{
\begin{matrix}\begin{tikzpicture}#1\end{tikzpicture}\end{matrix}
}

\tikzstyle{V}=[draw, fill =black, circle, inner sep=0pt, minimum size=1.5pt]
\tikzstyle{wV}=[draw, fill =white, circle, inner sep=0pt, minimum size=4.5pt]
\tikzstyle{bV}=[draw, fill =black, circle, inner sep=0pt, minimum size=4.5pt]
\tikzstyle{over}=[draw=white,double=black,line width=2pt, double distance=.5pt]

\def\Over[#1,#2][#3,#4]{ 
	\draw[style=over]   (#2,#1) .. controls ++(#4*.5-#2*.5,0) and ++(-#4*.5+#2*.5,0) .. (#4,#3);}
\def\Under[#1,#2][#3,#4]{ 
	\draw  (#2,#1) .. controls ++(#4*.5-#2*.5,0) and ++(-#4*.5+#2*.5,0) .. (#4,#3);}
\def\Cross[#1,#2][#3,#4]{
	\Under[#3,#2][#1,#4]\Over[#1,#2][#3,#4]}

\def\Tops[#1][#2][#3]{
	\foreach\x in {#1}{
		\draw (#2,\x+.15) -- (#2+.1, \x+.15) (#2, \x-.15) -- (#2+.1, \x-.15) ;
		\draw (#2+.1,\x) arc (0:360:.75mm and 1.5mm);}
	\foreach \x in {1,...,#3} {\draw (#2,\x)  to (#2+.05,\x); \node[V] at (#2+.05,\x){};}
	}
\def\Bottoms[#1][#2][#3]{
	\foreach\x in {#1}{
		\draw (#2, \x+.15) -- (#2-.1, \x+.15) (#2, \x-.15) -- (#2-.1, \x-.15) ;
		\draw (#2-.1, \x+.15) arc (90:270:.75mm and 1.5mm);}
	\foreach \x in {1,...,#3} {\draw (#2, \x)  to (#2-.05, \x); \node[V] at (#2-.05, \x){};}
	}
\def\Caps[#1][#2,#3][#4]{
	\Tops[#1][#3][#4]
	\Bottoms[#1][#2][#4]
	}
\def\Pole[#1][#2,#3]{
	\shade[left color=white,right color=white] (#2,#1+.15) rectangle (#3,#1-.15);
	\draw[over] (#2,#1+.15) to (#3,#1+.15) (#2,#1-.15) to (#3,#1-.15) ;}
\def\Label[#1,#2][#3][#4]{
	\node[right] at (#2+.1,#3) {#4};
	\node[left] at (#1-.1,#3) {#4};		}
\def\Nodes[#1][#2]{
	 \foreach \x in {1,...,#2} {\node[V] at (#1,\x){};	}
	}
\def\PoleCaps[#1][#2,#3]{
	\foreach\x in {#1}{
		\draw (#2,\x+.15) -- (#2-.1,\x+.15) (#2,\x-.15) -- (#2-.1,\x-.15) ;
		\draw (#2-.1,\x+.15) arc (0:-180:1.5mm and .75mm);}
	\foreach\x in {#1}{
		\draw (#3,\x+.15) -- (#3+.1,\x+.15) (#3,\x-.15) -- (#3+.1,\x-.15) ;
		\draw (#3+.1,\x+.15) arc (0:360:1.5mm and .75mm);}
	}
\def\PoleTwist[#1,#2]{
	\foreach \x/\y in {-1/1L, -.7/1R, 0/2L, .3/2R}{\coordinate(T\y) at (#2,\x); \coordinate(B\y) at (#1,\x);}
	\draw[thin] (B1R) .. controls ++(#2*.5-#1*.5-.1,0) and ++(-#2*.5+#1*.5-.1,0) ..  (T2R)
			(B1L)   .. controls ++(#2*.5-#1*.5+.1,0) and ++(-#2*.5+#1*.5+.1,0) ..    (T2L) ;
	\draw[line width=2pt, white]
			(#1,.15)  .. controls +(#2*.5-#1*.5,0) and +(-#2*.5+#1*.5,0) ..   (#2,-.85) ;
	\draw[thin,over] 
		(B2R) .. controls ++(#2*.5-#1*.5+.1,0) and ++(-#2*.5+#1*.5+.1,0) ..  (T1R) 
			(B2L)  .. controls +(#2*.5-#1*.5-.1,0) and +(-#2*.5+#1*.5-.1,0) ..   (T1L) ;
			}

\def\SymPolesCaps[#1,#2][#3]{
	\draw (#1,.3) -- (#1-.1,.3) (#1,.15) -- (#1-.1, .15) ;
	\draw (#1-.1, .3) arc (0:-180:2pt and 1.5pt);
	\draw (#1,#3+.7) -- (#1-.1,#3+.7) (#1,#3+.85) -- (#1-.1,#3+.85) ;
	\draw (#1-.1,#3+.85)  arc (0:-180:2pt and 1.5pt);
	\draw (#2,.3) -- (#2+.1, .3) (#2, .15) -- (#2+.1, .15) ;
	\draw (#2+.1, .3) arc (0:360:2pt and 1.5pt);
	\draw (#2, #3+.7) -- (#2+.1, #3+.7) (#2, #3+.85) -- (#2+.1, #3+.85) ;
	\draw (#2+.1, #3+.85) arc (0:360:2pt and 1.5pt);}

\newcommand{\posleq}[1]{
	\hspace{0.1cm}
	\begin{tikzpicture}
	\draw (-0.8ex, -0.5ex) -- (0.8ex, -0.5ex);
	\draw (-0.8ex, 0.4ex) -- (0.7ex, -0.2ex);
	\draw (-0.8ex, 0.4ex) -- (0.7ex, 1ex);
	\draw (0.4ex,0.4ex) --(1.1ex, 0.4ex);
	\draw (0.75ex,0.75ex) --(0.75ex, 0.05ex);
	\end{tikzpicture}
	\hspace{0.1cm}
	}
\newcommand{\negleq}[1]{
	\hspace{0.1cm}
	\begin{tikzpicture}
	\draw (-0.8ex, -0.5ex) -- (0.8ex, -0.5ex);
	\draw (-0.8ex, 0.4ex) -- (0.7ex, -0.2ex);
	\draw (-0.8ex, 0.4ex) -- (0.7ex, 1ex);
	\draw (0.4ex,0.4ex) --(1.1ex, 0.4ex);
	\end{tikzpicture}
	\hspace{0.1cm}
	}
	
\newcommand{\zeroleq}[1]{
	\hspace{0.1cm}
	\begin{tikzpicture}
	\draw (-0.8ex, -0.5ex) -- (0.8ex, -0.5ex);
	\draw (-0.8ex, 0.4ex) -- (0.7ex, -0.2ex);
	\draw (-0.8ex, 0.4ex) -- (0.7ex, 1ex);
	\draw  (0.75ex,0.4ex) ellipse (0.2ex and 0.35ex);
	\end{tikzpicture}
	\hspace{0.1cm}
	}
	
\newcommand{\posgeq}[1]{
	\hspace{0.1cm}
	\begin{tikzpicture}
	\draw (-0.8ex, -0.5ex) -- (0.8ex, -0.5ex);
	\draw (0.8ex, 0.4ex) -- (-0.7ex, -0.2ex);
	\draw (0.8ex, 0.4ex) -- (-0.7ex, 1ex);
	\draw (-0.4ex,0.4ex) --(-1.1ex, 0.4ex);
	\draw (-0.75ex,0.75ex) --(-0.75ex, 0.05ex);
	\end{tikzpicture}
	\hspace{0.1cm}
	}
\newcommand{\neggeq}[1]{
	\hspace{0.1cm}
	\begin{tikzpicture}
	\draw (-0.8ex, -0.5ex) -- (0.8ex, -0.5ex);
	\draw (0.8ex, 0.4ex) -- (-0.7ex, -0.2ex);
	\draw (0.8ex, 0.4ex) -- (-0.7ex, 1ex);
	\draw (-0.4ex,0.4ex) --(-1.1ex, 0.4ex);
	\end{tikzpicture}
	\hspace{0.1cm}
	}
	
\newcommand{\zerogeq}[1]{
	\hspace{0.1cm}
	\begin{tikzpicture}
	\draw (-0.8ex, -0.5ex) -- (0.8ex, -0.5ex);
	\draw (0.8ex, 0.4ex) -- (-0.7ex, -0.2ex);
	\draw (0.8ex, 0.4ex) -- (-0.7ex, 1ex);
	\draw  (-0.75ex,0.4ex) ellipse (0.2ex and 0.35ex);
	\end{tikzpicture}
	\hspace{0.1cm}
	}

\newcommand{\posl}[1]{
	\hspace{0.1cm}
	\begin{tikzpicture}
	\draw (-0.8ex, 0.4ex) -- (0.7ex, -0.2ex);
	\draw (-0.8ex, 0.4ex) -- (0.7ex, 1ex);
	\draw (0.4ex,0.4ex) --(1.1ex, 0.4ex);
	\draw (0.75ex,0.75ex) --(0.75ex, 0.05ex);
	\end{tikzpicture}
	\hspace{0.1cm}
	}
\newcommand{\negl}[1]{
	\hspace{0.1cm}
	\begin{tikzpicture}
	\draw (-0.8ex, 0.4ex) -- (0.7ex, -0.2ex);
	\draw (-0.8ex, 0.4ex) -- (0.7ex, 1ex);
	\draw (0.4ex,0.4ex) --(1.1ex, 0.4ex);
	\end{tikzpicture}
	\hspace{0.1cm}
	}
	
\newcommand{\zerol}[1]{
	\hspace{0.1cm}
	\begin{tikzpicture}
	\draw (-0.8ex, 0.4ex) -- (0.7ex, -0.2ex);
	\draw (-0.8ex, 0.4ex) -- (0.7ex, 1ex);
	\draw  (0.75ex,0.4ex) ellipse (0.2ex and 0.35ex);
	\end{tikzpicture}
	\hspace{0.1cm}
	}
	
\newcommand{\posg}[1]{
	\hspace{0.1cm}
	\begin{tikzpicture}
	\draw (0.8ex, 0.4ex) -- (-0.7ex, 1ex);
	\draw (0.8ex, 0.4ex) -- (-0.7ex, -0.2ex);
	\draw (-0.4ex,0.4ex) --(-1.1ex, 0.4ex);
	\draw (-0.75ex,0.75ex) --(-0.75ex, 0.05ex);
	\end{tikzpicture}
	\hspace{0.1cm}
	}
\newcommand{\negg}[1]{
	\hspace{0.1cm}
	\begin{tikzpicture}
	\draw (0.8ex, 0.4ex) -- (-0.7ex, -0.2ex);
	\draw (0.8ex, 0.4ex) -- (-0.7ex, 1ex);
	\draw (-0.4ex,0.4ex) --(-1.1ex, 0.4ex);
	\end{tikzpicture}
	\hspace{0.1cm}
	}
	
\newcommand{\zerog}[1]{
	\hspace{0.1cm}
	\begin{tikzpicture}
	\draw (0.8ex, 0.4ex) -- (-0.7ex, -0.2ex);
	\draw (0.8ex, 0.4ex) -- (-0.7ex, 1ex);
	\draw  (-0.75ex,0.4ex) ellipse (0.2ex and 0.35ex);
	\end{tikzpicture}
	\hspace{0.1cm}
	}

\makeatletter
\renewcommand{\@makefnmark}{\mbox{\textsuperscript{}}}
\makeatother

\title{Comparing formulas for type $GL_n$ Macdonald polynomials}
\author{
Weiying Guo\quad\ \ email:\ guwg@student.unimelb.edu.au \\
Arun Ram\quad\ \ email:\ aram@unimelb.edu.au \\
\\
}
\date{}

\usetikzlibrary{arrows.meta}

\begin{document}

\maketitle

\vspace{-3em}
\begin{center}
{\sl Dedicated to H\'el\`ene Barcelo}
\end{center}


\begin{abstract}
\noindent
The paper compares (and reproves) the alcove walk and the nonattacking fillings formulas
for type $GL_n$ Macdonald polynomials which were given in  \cite{HHL06},
\cite{Al16} and \cite{RY08}.  The ``compression'' relating the two formulas 
in this paper is the same as that of Lenart \cite{Len08}. 
We have reformulated it so that it holds without conditions and so that the proofs of the
alcove walk formula and the nonattacking fillings formula are parallel.  This reformulation 
highlights the role of the double affine Hecke algebra and Cherednik's intertwiners.
An exposition of the type $GL_n$ double affine braid
group, double affine Hecke algebra, and all definitions and proofs regarding Macdonald
polynomials are provided to make this paper self contained.
\end{abstract}


\setcounter{section}{-1}
\tableofcontents

\section{Introduction}


The Macdonald polynomials are an incredible family of orthogonal
polynomials which simultaneously 
generalize Schur functions, Weyl characters, Demazure characters,
Askey-Wilson polynomials, Koornwinder polynomials, Hall-Littlewood polynomials,
Jack polynomials
and spherical functions on $p$-adic groups.  They are eigenfunctions of a 
family of difference operators which generalize the classical Laplacian and, in this
sense, the Macdonald polynomials $E_\mu$ are generalizations of spherical harmonics.

This paper is a study of the relationship between combinatorial formulas
for $GL_n$-type Macdonald polynomials:
\begin{itemize}[itemsep=0em]
\item[(a)] The nonattacking fillings formulas from 
\cite[Theorem 3.5.1]{HHL06} and \cite{Al16}, and 
\item[(b)] The alcove walk formula from  \cite[Theorem 3.1]{RY08}.
\end{itemize}
Except for Section 4, which contains the recursions and the calculations for the proofs, 
we have made an effort to try to make the different sections of this paper readable 
independent of each other.
The reader should not hesitate to go directly to Section 5 for an introduction to the
double affine Hecke algebra, to Section 2 for an entr\'ee to $n$-periodic permutations
and the affine Weyl group, and to Section 3 for the basics of Macdonald polynomials and
some explicit examples of them.

The first half of Section 3 defines the various kinds of Macdonald polynomials, the $E_\mu$, the $P_\lambda$ and the $E^z_\mu$; the second half of Section 3 computes some examples.  
In \cite{Al16},
the \emph{relative Macdonald polynomials} $E^z_\mu = (const)T_z E_\mu$ of this paper 
are called ``permuted basement Macdonald polynomials''.
These ``$T_z$ shifted Macdonald polynomials'' are useful for all root systems and have an 
alcove walk formula \cite[Theorem 2.2]{RY08}. In the general root system setting, 
the notion of a ``basement'' has a different flavor (the $z$ in the $T_z$ is the ``basement'') 
and so we propose the term \emph{relative Macdonald polynomials} for the $E^z_\mu$.

As explained in Macdonald's book \cite{Mac03},
the $n$-periodic permutation $u_\mu$ defined in Section 2
is a critical ingredient for the understanding of the combinatorics
of Macdonald polynomials and their construction by intertwiners $\tau^\vee_i$.
Proposition \ref{KSredwd}
provides a favorite reduced word for $u_\mu$ and determines its inversions.
The inversions of $u_\mu$ provide the ``arms'' and ``legs'' that appear in \cite{HHL06}
(denoted $\mathrm{Narm}_\mu$ and $\mathrm{Nleg}_\mu$ in this paper),
and this observation connects those statistics 
with the roots of the affine root system for type $GL_n$.  Proposition \ref{boxbybox} derives
a box-by-box recursion for computing Macdonald polynomials and Remark \ref{covidremark} 
shows that the statistic that falls out of this derivation (in terms of a comparison of lengths of
permutations) counts  the coinversion triples that are used in \cite{HHL06}.
This observation completes the interpretation of the statistics in the nonattacking fillings formula
in terms of the Weyl group and the root system.  Let us highlight that using the box-by-box recursion
to compute Macdonald polynomials is equivalent to using a special reduced word for the 
$n$-periodic permutation $u_\mu$, the \emph{box greedy reduced word for $u_\mu$}.

The proof of the alcove walk formula is obtained by iterating the \emph{step-by-step recursion}
for the relative Macdonald polynomials $E^z_\mu$.
The proof of the nonattacking fillings formula is obtained by iterating the \emph{box-by-box recursion}
for the relative Macdonald polynomials $E^z_\mu$.  Except for the effort to normalize
the $E^z_\mu$ so that $x^{z\mu}$ has coefficient 1, the proof of the step-by-step recursion
does not differ from the proof of \cite[Theorem 2.2]{RY08}.  The proof of the box-by-box
recursion is, at its core, the same as \cite[Proposition 4.1]{Len08}.  Our reformulation
and proof highlights the role of the intertwiners and the connection to the affine root system
and pinpoints exactly which intertwiners get ``compressed''.

Section 5 provides a Type $GL_n$ specific exposition, 
from scratch, of the double affine Hecke algebra and
its use for defining and studying Macdonald polynomials.  \textbf{In \cite{GR21}, 
a supplement to this paper,
we provide examples and further observations.}

A small warning: Even though they all have a Type A root system,
type $SL_n$ Macdonald polynomials, type $PGL_n$ Macdonald polynomals
and   type $GL_n$ Macdonald
polynomials are all \emph{different} 
(though the relationship is well known and not difficult).  We should
stress that this paper is specific to the $GL_n$-case and some results of this paper
do not hold for Type $SL_n$ or type $PGL_n$ unless properly modified.

We thank L.\ Williams and M.\ Wheeler for bringing our attention to \cite{CMW18}
and \cite{BW19}, both of which were important stimuli during our work.  We are also very grateful 
for the encouragement, questions, and discussions from A.\ Hicks, S.\ Mason, O.\ Mandelshtam,
Z.\ Daugherty,  Y.\ Naqvi, S.\ Assaf, and especially A.\ Garsia and S.\ Corteel,
which helped so much in getting going and keeping up the energy.  We thank
S.\ Billey, Z.\ Daugherty, C.\ Lenart and J.\ Saied 
for very useful specific comments for improving 
the exposition.  A.\ Ram extends a very special and heartfelt
thank you to P.\ Diaconis who has provided 
unfailing support and advice and honesty and encouragement.


\section{Boxes, alcove walks and nonattacking fillings}\label{combsection}

The goal of this section is to state the main results: the alcove walks formula and 
the nonattacking fillings formula, and the compression map $\psi$ which relates them.  We begin
by setting up the combinatorics of boxes, diagrams, alcove walks and nonattacking fillings.
Then, after specifying the weights attached to alcove walks and to nonattacking fillings
we state the alcove walks formula and the nonattacking fillings formula for Macdonald
polynomials as weighted sums of alcove walks and nonattacking fillings, respectively.

\subsection{Boxes}

Fix $n\in \ZZ_{>0}$.
A \emph{box} is an element of $\{1,\ldots, n\}\times \ZZ_{\ge 0}$ so that
$$\{\hbox{boxes}\} = \{ (i,j)\ |\ i\in \{1, \ldots, n\},\ j\in \ZZ_{\ge 0}\}.$$
To conform to \cite[p.2]{Mac}, we draw the box $(i,j)$ 
as a square in row $i$ and column $j$ using the same coordinates
as are usually used for matrices.
\begin{equation}
\hbox{The \emph{cylindrical coordinate} of the box $(i,j)$ is the number $i+nj$.}
\label{cylcoorddefn}
\end{equation}
The \emph{basement} is the set $\{ (i,0)\ |\ i\in \{1,\ldots, n\}\}$, so that the basement is
the collection of boxes in the 0th column.
Pictorially, 
$$
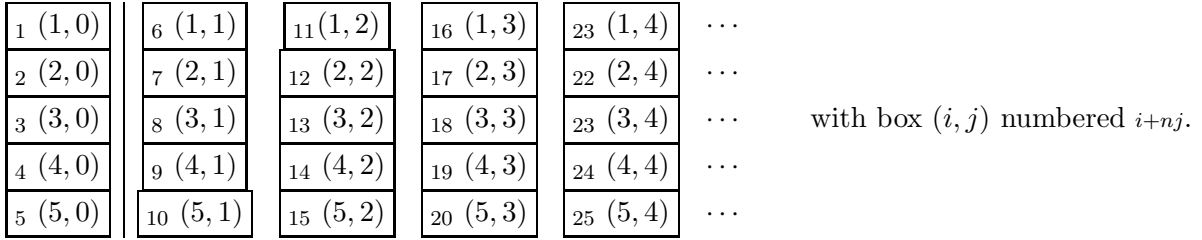
\begin{array}{c|cccccc}
\boxed{{}_1\ (1,0)} 
&\boxed{{}_6\ (1,1)} 
&\boxed{{}_{11} (1,2)} 
&\boxed{{}_{16}\ (1,3)} 
&\boxed{{}_{23}\ (1,4)} 
&\cdots 
\\
\boxed{{}_{2}\ (2,0)} 
&\boxed{{}_{7}\ (2,1)} 
&\boxed{{}_{12}\ (2,2)} 
&\boxed{{}_{17}\ (2,3)} 
&\boxed{{}_{22}\ (2,4)} 
&\cdots 
\\
\boxed{{}_{3}\ (3,0)} 
&\boxed{{}_{8}\ (3,1)} 
&\boxed{{}_{13}\ (3,2)} 
&\boxed{{}_{18}\ (3,3)} 
&\boxed{{}_{23}\ (3,4)} 
&\cdots 
\\
\boxed{{}_{4}\ (4,0)} 
&\boxed{{}_{9}\ (4,1)} 
&\boxed{{}_{14}\ (4,2)} 
&\boxed{{}_{19}\ (4,3)} 
&\boxed{{}_{24}\ (4,4)} 
&\cdots 
\\
\boxed{{}_{5}\ (5,0)} 
&\boxed{{}_{10}\ (5,1)} 
&\boxed{{}_{15}\ (5,2)} 
&\boxed{{}_{20}\ (5,3)} 
&\boxed{{}_{25}\ (5,4)} 
&\cdots 
\end{array}
\qquad\hbox{with box $(i,j)$ numbered $\scriptstyle{i+nj}$.}
$$

Let $\mu=(\mu_1, \ldots, \mu_n)\in \ZZ_{\ge 0}^n$ an $n$-tuple of nonnegative integers.  
The \emph{diagram of $\mu$} is 
the set $dg(\mu)$ of boxes with $\mu_i$ boxes in row $i$
and the \emph{diagram of $\mu$ with basement} $\widehat{dg}(\mu)$
includes the extra boxes $(i,0)$ for 
$i\in \{1, \ldots,n\}$:
\begin{align*}
dg(\mu) &= \{ (i,j)\ |\ \hbox{$i\in \{1, \ldots, n\}$ and $j\in \{1, \ldots, \mu_i\}$} \} \quad\hbox{and}  \\
\widehat{dg}(\mu) &= \{ (i,j)\ |\ \hbox{$i\in \{1, \ldots, n\}$ and $j\in \{0, 1, \ldots, \mu_i\}$} \} 
\end{align*}
It is often convenient to abuse notation and identify $\mu$, $dg(\mu)$ and $\widehat{dg}(\mu)$
(because these are just different ways of viewing the sequence $(\mu_1, \ldots, \mu_n)$).
For example, if $\mu = (0,4,1,5,4)$ then
$$
dg(\mu) = \begin{array}{|cccccc}
\\
\boxed{\phantom{T}} 
&\boxed{\phantom{T}} 
&\boxed{\phantom{T}} 
&\boxed{\phantom{T}} 
\\
\boxed{\phantom{T}} 
\\
\boxed{\phantom{T}} 
&\boxed{\phantom{T}} 
&\boxed{\phantom{T}} 
&\boxed{\phantom{T}} 
&\boxed{\phantom{T}} 
\\
\boxed{\phantom{T}} 
&\boxed{\phantom{T}} 
&\boxed{\phantom{T}} 
&\boxed{\phantom{T}} 
\end{array}
\quad\hbox{and}\quad
\widehat{dg}(\mu) = \begin{array}{c|cccccc}
\boxed{\phantom{T}} 
\\
\boxed{\phantom{T}} 
&\boxed{\phantom{T}} 
&\boxed{\phantom{T}} 
&\boxed{\phantom{T}} 
&\boxed{\phantom{T}} 
\\
\boxed{\phantom{T}} 
&\boxed{\phantom{T}} 
\\
\boxed{\phantom{T}} 
&\boxed{\phantom{T}} 
&\boxed{\phantom{T}} 
&\boxed{\phantom{T}} 
&\boxed{\phantom{T}} 
&\boxed{\phantom{T}} 
\\
\boxed{\phantom{T}} 
&\boxed{\phantom{T}} 
&\boxed{\phantom{T}} 
&\boxed{\phantom{T}} 
&\boxed{\phantom{T}} 
\end{array}
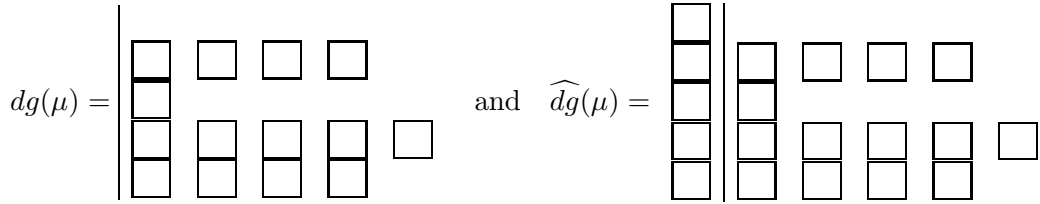
$$


\subsection{Alcove walks and nonattacking fillings}

Let $\mu\in \ZZ_{\ge 0}^n$.
Using cylindrical coordinates for boxes as specified \eqref{cylcoorddefn}, define, 
for a box $b\in dg(\mu)$,
\begin{align}
\mathrm{attack}_\mu(b) &= \{b-1, \ldots, b-n+1\}\cap \widehat{dg}(\mu), 
\label{attackdefn}
\\
\mathrm{Nleg}_\mu(b) &= (b+n\ZZ_{>0})\cap dg(\mu)
\quad \hbox{and} \\
\mathrm{Narm}_\mu(b) &= \{ a\in \mathrm{attack}_\mu(b)
 \ |\ \#\mathrm{Nleg}_\mu(a)\le \#\mathrm{Nleg}_\mu(b)\}.
\label{Harmdef}
\end{align}
where $\#\mathrm{Nleg}_\mu(a)$ denotes the number of elements of $\mathrm{Nleg}_\mu(a)$.
For example, with $\mu = (3,0,5,1,4,3,4)$ and $b=(5,2)$, which has cylindrical coordinate
$b=5+7\cdot 2 = 19$ the sets $\mathrm{attack}_\mu(b)$, $\mathrm{Narm}_\mu(b)$
and $\mathrm{Nleg}_\mu(b)$ are pictured as
$$
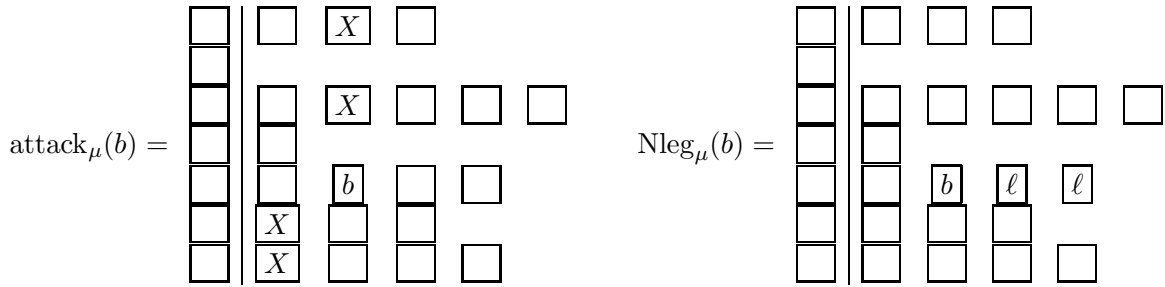
\mathrm{attack}_\mu(b) =
\begin{array}{c|ccccccc}
\boxed{\phantom{T} }
&\boxed{\phantom{T} }
&\boxed{X}
&\boxed{\phantom{T} }
\\
\boxed{\phantom{T} }
\\
\boxed{\phantom{T} }
&\boxed{\phantom{T}} 
&\boxed{X}
&\boxed{\phantom{T} }
&\boxed{\phantom{T} }
&\boxed{\phantom{T} }
\\
\boxed{\phantom{T} }
&\boxed{\phantom{T} }
\\
\boxed{\phantom{T} }
&\boxed{\phantom{T} }
&\boxed{b}
&\boxed{\phantom{T} }
&\boxed{\phantom{T} }
\\
\boxed{\phantom{T} }
&\boxed{X}
&\boxed{\phantom{T} }
&\boxed{\phantom{T} }
\\
\boxed{\phantom{T}}
&\boxed{X}
&\boxed{\phantom{T} }
&\boxed{\phantom{T} }
&\boxed{\phantom{T} }
\end{array}
\qquad 
\mathrm{Nleg}_\mu(b) = \begin{array}{c|ccccccc}
\boxed{\phantom{T} }
&\boxed{\phantom{T} }
&\boxed{\phantom{T} }
&\boxed{\phantom{T} }
\\
\boxed{\phantom{T} }
\\
\boxed{\phantom{T} }
&\boxed{\phantom{T}} 
&\boxed{\phantom{T} }
&\boxed{\phantom{T} }
&\boxed{\phantom{T} }
&\boxed{\phantom{T} }
\\
\boxed{\phantom{T} }
&\boxed{\phantom{T} }
\\
\boxed{\phantom{T} }
&\boxed{\phantom{T} }
&\boxed{b}
&\boxed{\ell }
&\boxed{\ell }
\\
\boxed{\phantom{T} }
&\boxed{\phantom{T} }
&\boxed{\phantom{T} }
&\boxed{\phantom{T} }
\\
\boxed{\phantom{T} }
&\boxed{\phantom{T} }
&\boxed{\phantom{T} }
&\boxed{\phantom{T} }
&\boxed{\phantom{T} }
\end{array}
$$
$$
\mathrm{Narm}_\mu(b) = \begin{array}{c|ccccccc}
\boxed{\phantom{T} }
&\boxed{\phantom{T} }
&\boxed{a}
&\boxed{\phantom{T} }
\\
\boxed{\phantom{T} }
\\
\boxed{\phantom{T} }
&\boxed{\phantom{T}} 
&\boxed{\phantom{T} }
&\boxed{\phantom{T} }
&\boxed{\phantom{T} }
&\boxed{\phantom{T} }
\\
\boxed{\phantom{T} }
&\boxed{\phantom{T} }
\\
\boxed{\phantom{T} }
&\boxed{\phantom{T} }
&\boxed{b}
&\boxed{\phantom{T} }
&\boxed{\phantom{T} }
\\
\boxed{\phantom{T} }
&\boxed{a}
&\boxed{\phantom{T} }
&\boxed{\phantom{T} }
\\
\boxed{\phantom{T} }
&\boxed{\phantom{T} }
&\boxed{\phantom{T} }
&\boxed{\phantom{T} }
&\boxed{\phantom{T} }
\end{array}
$$

Let $\mu = (\mu_1, \ldots, \mu_n)\in \ZZ^n_{\ge 0}$
and let $u_\mu$ be the $n$-periodic permutation defined in \eqref{umudefn}.
The \emph{box-greedy reduced word for $u_\mu$} is
\begin{equation}
u^\square_\mu = \prod_{\mathrm{boxes}\ (i,j)\ \mathrm{in}\  dg(\mu)} (s_{u_\mu(i,j)}\cdots s_2s_1\pi)
\qquad\hbox{where}\quad
u_\mu(i,j) = n-1-\#\mathrm{attack}_\mu(i,j).
\label{bgredwddefn}
\end{equation}
For the purposes of this section it is only necessary to recognize $u^\square_\mu$ as an
abstract word in symbols $s_1, \ldots, s_{n-1}, \pi$.
For an example, if $\mu = (0,4,1,5,4)$ then the box-greedy reduced word for $u_\mu$ is
\begin{equation}
u^\square_\mu = (s_1\pi)^5(s_2s_1\pi)^8(s_3s_2s_1\pi) = 
\begin{array}{|ccccc}
\phantom{ \boxed{ \begin{matrix} \phantom{T} \\  \end{matrix} }  }
\\
\boxed{ \begin{matrix} s_1\pi  \end{matrix} } 
&\boxed{ \begin{matrix}  s_1\pi \end{matrix} }
&\boxed{ \begin{matrix} s_2s_1\pi \end{matrix} }
&\boxed{ \begin{matrix} s_2s_1\pi  \end{matrix} }
\\
\boxed{ \begin{matrix} s_1\pi  \end{matrix} } \\
\boxed{\begin{matrix} s_1\pi \end{matrix} } 
&\boxed{ \begin{matrix}   s_2s_1\pi  \end{matrix} }
&\boxed{ \begin{matrix} s_2s_1\pi \end{matrix} }
&\boxed{ \begin{matrix} s_2s_1\pi \end{matrix} }
&\boxed{\begin{matrix} s_3s_2s_1\pi \end{matrix} } 
\\
\boxed{ \begin{matrix} s_1\pi \end{matrix} }
&\boxed{ \begin{matrix} s_2s_1\pi \end{matrix} }
&\boxed{ \begin{matrix} s_2s_1\pi  \end{matrix} }
&\boxed{ \begin{matrix} s_2s_1\pi \end{matrix} }
\end{array}
\label{bgredwdexample}
\end{equation}
(the reduced word is a product of the boxes read in increasing order by cylindrical coordinate).

Let $\mu = (\mu_1, \ldots, \mu_n)\in \ZZ^n_{\ge 0}$ and $z\in S_n$.
Let $\vec u_\mu=w_1w_2\cdots w_\ell$ be a reduced word for $u_\mu$
so that $w_1, \ldots, w_\ell$ are the factors of $\vec u_\mu$ 
(a good choice is to let $\vec u_\mu = u^\square_\mu$).
An \emph{alcove walk} of type $(z,\vec u_\mu)$
is a sequence $p=(p_0, p_1, \ldots, p_r)$ of elements of $W$ (see \eqref{nperiodicdefn},
but, as for $u^\square_\mu$, it is sensible just to view the $p_k$ as words in the 
symbols $s_1, \ldots, s_{n-1},\pi$)
such that
$$\hbox{$p_0 = z$,}\qquad
\hbox{$p_k = p_{k-1}\pi$ if $w_k = \pi$},
\qquad\hbox{and}\qquad
\hbox{$p_k\in \{p_{k-1}, p_{k-1}w_k\}$ if $w_k\ne \pi$.}
$$
In other words, an alcove walk of type $(z,\mu)$ is equivalent to choosing a 
subset of the $s_i$ factors in $\vec u_\mu$ to cross out.  For example,
\begin{equation}
P = \quad \begin{array}{|ccccc}
\phantom{ \boxed{ \begin{matrix}  \phantom{T} \\ \end{matrix} }  }
\\
\boxed{ \begin{matrix} \cancel{s_1}\pi  \end{matrix} } 
&\boxed{ \begin{matrix}  s_1\pi \end{matrix} }
&\boxed{ \begin{matrix} s_2\cancel{s_1}\pi \end{matrix} }
&\boxed{ \begin{matrix} \cancel{s_2}s_1\pi  \end{matrix} }
\\
\boxed{ \begin{matrix}  s_1\pi  \end{matrix} } 
\\
\boxed{\begin{matrix} \cancel{s_1}\pi \end{matrix} } 
&\boxed{ \begin{matrix}   s_2s_1\pi  \end{matrix} }
&\boxed{ \begin{matrix} \cancel{s_2}\cancel{s_1}\pi \end{matrix} }
&\boxed{ \begin{matrix} s_2s_1\pi \end{matrix} }
&\boxed{\begin{matrix} \cancel{s_3}s_2\cancel{s_1}\pi \end{matrix} } 
\\
\boxed{ \begin{matrix} s_1\pi \end{matrix} }
&\boxed{ \begin{matrix} s_2s_1\pi \end{matrix} }
&\boxed{ \begin{matrix} \cancel{s_2}\cancel{s_1}\pi  \end{matrix} }
&\boxed{ \begin{matrix} s_2s_1\pi \end{matrix} }
\end{array}
\qquad\hbox{is equivalent to the alcove walk} 
\label{alcwalkexample}
\end{equation}
$$p=(p_0, p_1, \ldots, p_{37}) = (z, z, z\pi, z\pi s_1, z\pi s_1\pi, z\pi s_1\pi, z\pi s_1\pi^2, 
z\pi s_1 \pi^2 s_1, z\pi s_1 \pi^2 s_1\pi, \ldots)$$
(there is a repeat entry in $p$ each time there is an $s_i$ crossed out in $P$).
In this example, there are $5+2\cdot 8+3=24$ factors of the form $s_i$ in $u^\square_\mu$ and 
so there are a total of $2^{24}$ alcove walks of type $(z,u^\square_\mu)$ 
(for any fixed permutation $z\in S_n$).

Let $\mu = (\mu_1, \ldots, \mu_n)\in \ZZ^n_{\ge 0}$ and $z\in S_n$.
A \emph{nonattacking filling} for $(z,\mu)$ is
$T\colon \widehat{dg}(\mu) \to \{1, \ldots, n\}$ such that
\begin{enumerate}
\item[(a)] $T(i,0) = z(i)$ for $i\in \{1, \ldots, n\}$ and
\item[(b)] if $b\in dg(\mu)$ and $a\in \mathrm{attack}_\mu(b)$ then $T(a)\ne T(b)$.
\end{enumerate}
For example,
\begin{equation}
T = \quad \begin{array}{c|ccccc}
1 \\
2 &1 &1 &1 &2 \\
3 &3 \\
4 &4 &4 &5 &4 &4 \\
5 &5 &2 &3 &3
\end{array}
\qquad\qquad
\begin{array}{l}
\hbox{is a nonattacking filling for $(z,\mu)$} \\
\\
\quad\hbox{with $z=\id\in S_5$ and $\mu = (0,4,1,5,4)$.}
\end{array}
\label{nafexample}
\end{equation}
Let $b$ be a box in $\mu$.  Starting at $b$ read, in succession, in reverse order by 
cylindrical coordinate, the entries from $T$ in (earlier) boxes, 
skipping values that have already been encountered.  
This process produces,
for each box $b\in dg(\mu)$, a permutation in $S_n$.  For example, with
$T$ as in \eqref{nafexample},
box $(4,3)$ in $T$ (row 4, column 3) produces the permutation
$$(34215)\ \hbox{formed from the circled numbers in}\qquad
\quad \begin{array}{c|ccccc}
1 \\
2 &1 &1 &\textcircled{1} &2 \\
3 &\textcircled{3} \\
4 &4 &\textcircled{4} &\textcircled{5} &4 &4 \\
5 &5 &\textcircled{2} &3 &3
\end{array}
$$
and doing this for all boxes in $T$ produces
\begin{equation}
z_T = \qquad \begin{array}{|ccccc}
\phantom{ \boxed{ \begin{matrix} \phantom{T} \\  \end{matrix} }  }
\\
\boxed{ \begin{matrix} (23451)  \end{matrix} } 
&\boxed{ \begin{matrix}  (23451) \end{matrix} }
&\boxed{ \begin{matrix} (35421) \end{matrix} }
&\boxed{ \begin{matrix} (41532)  \end{matrix} }
\\
\boxed{ \begin{matrix} (24513)  \end{matrix} } \\
\boxed{\begin{matrix} (25134) \end{matrix} } 
&\boxed{ \begin{matrix}   (23514)  \end{matrix} }
&\boxed{ \begin{matrix} (34215) \end{matrix} }
&\boxed{ \begin{matrix} (15324) \end{matrix} }
&\boxed{\begin{matrix} (15234) \end{matrix} } 
\\
\boxed{ \begin{matrix} (21345) \end{matrix} }
&\boxed{ \begin{matrix} (35142) \end{matrix} }
&\boxed{ \begin{matrix} (42153)  \end{matrix} }
&\boxed{ \begin{matrix} (15243) \end{matrix} }
\end{array}
\label{permseqexample}
\end{equation}
%
The sequence
\begin{equation}
z_T = ( z_T(b) \ |\ b\in dg(\mu))
\quad\hbox{is the \emph{permutation sequence of $T$}.}
\label{permseqsemidefn}
\end{equation}
Let $c_n = s_1\cdots s_{n-1}$, an $n$-cycle in $S_n$.
If $b=(i,j)$ is a box in $dg(\mu)$, the permutation $z_T(b')$ in the next box of $z_T$ 
(by cylindrical coordinate) is
\begin{equation}
z_T(b') = z_T(b) s_r \cdots s_2s_1c_n,
\qquad \hbox{where $r\in \{0,1, \ldots, u_\mu(i,j)\}$},
\label{suffixsource}
\end{equation}
and $s_r\cdots s_2s_1\pi$ is the entry in box $b$ of  the alcove walk 
$\varphi(T)$ corresponding to the nonattacking filling $T$. (If this construction of the permutation
sequence feels ad hoc, the sentence before Lemma \ref{compressionlemma}
may help to provide some insight into its source.)

For example, for $z_T$ as in \eqref{permseqexample}, and with $z=(12345)$ the permutation
in the basement of $T$, then
$$
\begin{array}{ll}
(23451) = z_T(2,1) = zc_n, 
&(23451) = z_T(2,2) = z_T(5,1)s_1c_n, 
\\
(24513) = z_T(3,1) = z_T(2,1)s_1c_n, 
\\
(25134) = z_T(4,1) = z_T(3,1)s_1c_n, 
&(23514) = z_T(4,2) = z_T(2,2) s_2s_1c_n, 
\\
(21345) = z_T(5,1) = z_T(4,1)s_1c_n, 
\quad 
&(35142) = z_T(5,2) = z_T(4,2)c_n, 
\\
\end{array}
$$
and so forth all the way to the last box of $\mu$.  Keeping track only of the factor which is the difference between successive boxes produces the alcove walk
$$\varphi(T) = 
\begin{array}{|ccccc}
\phantom{ \boxed{ \begin{matrix} \phantom{T} \\  \end{matrix} }  }
\\
\boxed{ \begin{matrix} \cancel{s_1}\pi  \end{matrix} } 
&\boxed{ \begin{matrix}  s_1\pi \end{matrix} }
&\boxed{ \begin{matrix} s_2s_1\pi \end{matrix} }
&\boxed{ \begin{matrix} \cancel{s_2}s_1\pi  \end{matrix} }
\\
\boxed{ \begin{matrix} s_1\pi  \end{matrix} } \\
\boxed{\begin{matrix} s_1\pi \end{matrix} } 
&\boxed{ \begin{matrix}   s_2s_1\pi  \end{matrix} }
&\boxed{ \begin{matrix} \cancel{s_2}s_1\pi \end{matrix} }
&\boxed{ \begin{matrix} \cancel{s_2}\cancel{s_1}\pi \end{matrix} }
&\boxed{\begin{matrix} s_3s_2s_1\pi \end{matrix} } 
\\
\boxed{ \begin{matrix} s_1\pi \end{matrix} }
&\boxed{ \begin{matrix} \cancel{s_2}\cancel{s_1}\pi \end{matrix} }
&\boxed{ \begin{matrix} \cancel{s_2}\cancel{s_1}\pi  \end{matrix} }
&\boxed{ \begin{matrix} s_2s_1\pi \end{matrix} }
\end{array}
$$
In summary, letting
$$\mathrm{AW}^z_\mu = \{ \hbox{alcove walks of type $(z,\vec u_\mu)$}\}
\qquad\hbox{and}\qquad
\mathrm{NAF}^z_\mu = \{ \hbox{nonattacking fillings for $(z,\mu)$}\},
$$
we have produced an injective map
$$\varphi\colon \mathrm{NAF}^z_\mu \to \mathrm{AW}^z_\mu.$$
By \eqref{suffixsource}, the image of $\varphi$ consists exactly of the alcove walks such that,
each box $b=(i,j)\in dg(\mu)$ contains a suffix $s_r\cdots s_1\pi$ of the entry
$s_{u_\mu(i,j)}\cdots s_1\pi$ in box $b$ in $u^\square_\mu$.  The \emph{compression
map} is the function
\begin{equation}
\psi\colon \mathrm{AW}^z_\mu \to \mathrm{NAF}^z_\mu
\label{compressionfcn}
\end{equation}
which, in each box, forces every $s_i$ factor before the last crossed out factor in that box
also to be crossed out.  For example, if $P$ is the alcove walk in \eqref{alcwalkexample}
then
\begin{equation*}
\psi(P) = \quad \begin{array}{|ccccc}
\phantom{ \boxed{ \begin{matrix}  \phantom{T} \\ \end{matrix} }  }
\\
\boxed{ \begin{matrix} \cancel{s_1}\pi  \end{matrix} } 
&\boxed{ \begin{matrix}  s_1\pi \end{matrix} }
&\boxed{ \begin{matrix} {\color{red}\cancel{s_2}}\cancel{s_1}\pi \end{matrix} }
&\boxed{ \begin{matrix} \cancel{s_2}s_1\pi  \end{matrix} }
\\
\boxed{ \begin{matrix}  s_1\pi  \end{matrix} } 
\\
\boxed{\begin{matrix} \cancel{s_1}\pi \end{matrix} } 
&\boxed{ \begin{matrix}   s_2s_1\pi  \end{matrix} }
&\boxed{ \begin{matrix} \cancel{s_2}\cancel{s_1}\pi \end{matrix} }
&\boxed{ \begin{matrix} s_2s_1\pi \end{matrix} }
&\boxed{\begin{matrix} \cancel{s_3}{\color{red}\cancel{s_2}}\cancel{s_1}\pi \end{matrix} } 
\\
\boxed{ \begin{matrix} s_1\pi \end{matrix} }
&\boxed{ \begin{matrix} s_2s_1\pi \end{matrix} }
&\boxed{ \begin{matrix} \cancel{s_2}\cancel{s_1}\pi  \end{matrix} }
&\boxed{ \begin{matrix} s_2s_1\pi \end{matrix} }
\end{array}
\end{equation*}
Identifying $\mathrm{NAF}^z_\mu$ with $\mathrm{im}\,\varphi$, then
$\psi\circ \varphi\colon \mathrm{NAF}^z_\mu \to \mathrm{NAF}^z_\mu$ is the identity
map on nonattacking fillings.

\subsection{Formulas for the relative Macdonald polynomials $E^z_\mu$}

In this subsection we state the alcove walks formula and the nonattacking fillings formula
for $E^z_\mu$.  The proofs are by the step by step recursion (Proposition \ref{stepbystep})
and the box by box recursion (Proposition \ref{boxbybox}), respectively. The
statistics $\mathrm{sh}(-\beta^\vee_k)$, $\mathrm{ht}(-\beta^\vee_k)$, 
$\mathrm{norm}(p_k)$ on alcove walks which are introduced below are read off of the 
step by step recursion,
Proposition \ref{stepbystep}.  Similarly,
the statistics $\#\mathrm{Nleg}_\mu(b)+1$, $\#\mathrm{Narm}_\mu(b)+1$,
and $\#\mathrm{bwn}_T(b)$ on nonattacking fillings which are introduced below are 
read off the box by box recursion, Proposition \ref{boxbybox}, and Remark \ref{covidremark}.

Equations \eqref{alcwalkdefn}-\eqref{alcovewalkwt} use the notations of 
Section 
\ref{affcorootssection}
so that $W$ is the group of $n$-periodic permutations defined in \eqref{nperiodicdefn}
the root sequence for $\vec u_\mu$ corresponds to the inversions of $u_\mu$ as in 
\eqref{invsetascoroots}
and the shift and height of an affine coroot are as given in \eqref{GLnshhtdefn}.

Let $\mu = (\mu_1, \ldots, \mu_n)\in \ZZ^n_{\ge 0}$ and $z\in S_n$.
Let $s_\pi = \pi$ and let $\vec u_\mu=s_{i_1}\cdots s_{i_r}$ be a reduced word for $u_\mu$
(a good choice is to let $\vec u_\mu = u^\square_\mu$).
An \emph{alcove walk} of type $(z,\vec u_\mu)$
is 
\begin{equation}
\hbox{a sequence\quad $p=(p_0, p_1, \ldots, p_r)$\quad of elements of $W$\quad such that}
\label{alcwalkdefn}
\end{equation}
$p_0 = z$;
if $s_{i_k} = \pi$
then $p_k = p_{k-1}\pi$; and 
if $s_{i_k}\ne \pi$ then
$p_k\in \{p_{k-1}, p_{k-1}s_{i_k}\}$.  The \emph{permutation sequence of $p$} is the 
sequence of elements of $S_n$,
\begin{equation}
\vec z_p=(z_0, \ldots, z_r),
\qquad\hbox{given by}\qquad z_j = \overline{p_j},
\label{alcwalkpermseq}
\end{equation}
where $\overline{\phantom{T}}\colon W\to S_n$ 
is the homomorphism given by $\overline{t_\nu v} = v$ (see \eqref{directionhom}).
The \emph{root sequence} for $\vec u_\mu$ is 
$$\hbox{the sequence\quad $(\beta^\vee_k\ |\ i_k\ne \pi)$\quad
given by\quad}
\beta^\vee_j = s^{-1}_{i_r} s^{-1}_{i_{r-1}} \cdots s^{-1}_{i_{k+1}}\alpha^\vee_{i_k}.
$$
Define
$$
\mathrm{ht}(\varepsilon^\vee_i - \varepsilon^\vee_j - \ell K) = j-i, 
\qquad\hbox{and} \qquad
\mathrm{sh}(\varepsilon^\vee_i - \varepsilon^\vee_j - \ell K) = \ell.
$$
For $k\in \{1, \ldots, r\}$ with $p_{k-1}=p_k$ define
$$\mathrm{norm}(p_k) 
= \hbox{$\frac12$}\big(\ell(z_{k-1}s_{i_k}v^{-1}_{z_{k-1}\mu})-\ell(z_{k-1}v^{-1}_{z_{k-1}\mu}) - \ell(s_{i_k}) \big),
$$
where $\ell(s_{i_k})=1$.
For a step $k\in \{1, \ldots, \ell\}$ of the alcove walk $p= (p_0, \ldots, p_\ell)$ define
the \emph{weight of $p_k$} by
$$\wt_p(k) = \begin{cases}
\displaystyle{
\Big(\frac{1-t}{1-q^{\mathrm{sh}(-\beta^\vee_k)}t^{\mathrm{ht}(-\beta^\vee_k)}}\Big)
t^{\mathrm{norm}(p_k)}, }
&\hbox{if $p_k=p_{k-1}$ and $p_{k-1}s_{i_k}<p_{k-1}$,} \\
\displaystyle{
\Big(\frac{(1-t)q^{\mathrm{sh}(-\beta^\vee_k)}t^{\mathrm{ht}(-\beta^\vee_k)}
}{1-q^{\mathrm{sh}(-\beta^\vee_k)}t^{\mathrm{ht}(-\beta^\vee_k)}}\Big)
t^{\mathrm{norm}(p_k)},
}
&\hbox{if $p_k=p_{k-1}$ and $p_{k-1}s_{i_k}>p_{k-1}$,} \\
1, &\hbox{if $p_k=p_{k-1}s_{i_k}$,} \\
x_{z_{k-1}(1)} &\hbox{if $p_k=p_{k-1}\pi$,} 
\end{cases}
$$
and define the \emph{weight of $p$} by
\begin{equation}
\wt(p) = \prod_{k=1}^\ell \wt_p(k),
\qquad\hbox{a product over the steps of $p$.}
\label{alcovewalkwt}
\end{equation}

Let $\mu\in \ZZ^n_{\ge 0}$ and $z\in S_n$ and let $T$ be a nonattacking filling of shape $(z,\mu)$.
For $b\in dg(\mu$) let 
\begin{equation}
\mathrm{bwn}_T(b) = \{ a\in \mathrm{arm}_\mu(b)\ |\ \hbox{$T(b-n)>T(a)>T(b)$} \}.
\label{bwndefn}
\end{equation}
The \emph{weight of $b$ in $T$} is 
\begin{equation}
\wt_T(b) = \begin{cases}
\displaystyle{
\Big(\frac{1-t}{1-q^{\#\mathrm{Nleg}_\mu(b)+1}t^{\#\mathrm{Narm}_\mu(b)+1}} \Big)
t^{\#\mathrm{bwn}_T(b)}x_{T(b)},
}
&\hbox{ if $T(b-n)>T(b)$,} \\
\displaystyle{
\Big( \frac{(1-t)q^{\#\mathrm{Nleg}_\mu(b)+1}t^{\#\mathrm{Narm}_\mu(b)+1}}
{1-q^{\#\mathrm{Nleg}_\mu(b)+1}t^{\#\mathrm{Narm}_\mu(b)+1}}\Big)
t^{
-1-\#\mathrm{bwn}_T(b)}x_{T(b)},
}
&\hbox{ if $T(b-n)<T(b)$,} \\
x_{T(b)}, &\hbox{if $T(b-n)=T(b)$,}
\end{cases}
\label{boxwt}
\end{equation}
and the \emph{weight of $T$} is
\begin{equation}
\wt(T) = \prod_{b\in dg(\mu)} \wt_T(b),
\qquad\hbox{a product over the boxes of $T$.}
\label{nafwt}
\end{equation}

The following theorem summarizes (and slightly generalizes)
\cite[Theorem 3.1]{RY08}, \cite[Def.\ 5 and Prop.\ 6]{Al16} and
\cite[Theorem 3.5.1]{HHL06}. 

\begin{thm}
Let $\mu\in \ZZ_{\ge 0}^n$ and $z\in S_n$.  Let $E^z_\mu$ be the permuted basement 
nonsymmetric Macdonald polynomial defined in \eqref{Ezmudefn}.
Let $\vec u_\mu$ be a reduced word for $u_\mu$ and let
$$\mathrm{AW}^z_\mu = \{ \hbox{alcove walks of type $(z,\vec u_\mu)$}\}
\qquad\hbox{and} \qquad
\mathrm{NAF}^z_\mu = \{ \hbox{nonattacking fillings for $(z,\mu)$} \}
$$
\item[(a)] Alcove walk formula:
$\displaystyle{  E^z_\mu = \sum_{p\in AW^z_\mu}  \wt(p) }.$
\item[(b)] Nonattacking fillings formula:
$\displaystyle{
E^z_\mu  = \sum_{T\in \mathrm{NAF}^z_\mu} \wt(T)
}$.
\end{thm}
\begin{proof}
(a) is obtained by successive
applications of the step-by-step recursion (Proposition \ref{stepbystep}),
and (b) is obtained by successive
applications of the box-by-box recursion (Proposition \ref{boxbybox}).
The weight of each box $\wt_T(b)$ comes from the coefficient of the correpsonding
term in Proposition \ref{boxbybox} and, by Remark \ref{covidremark} and Remark \ref{rootstolegarm},
these weights can be stated in the form \eqref{bwndefn} and \eqref{boxwt}.
\end{proof}

The following example illustrates the proof of the nonattacking fillings formula
by iterating the box by box recursion to produce the nonattacking fillings expansion
of the Macdonald polynomial 
$E_{(2,2,1,1,0,0)}=E^{(123456)}_{(2,2,1,1,0,0)}$.
The first four applications of Proposition \ref{boxbybox} give
$$E^{(123456)}_{(2,2,1,1,0,0)}
= x_1 E^{(234561)}_{(2,1,1,0,0,1)}
= x_1x_2 E^{(345612)}_{(1,1,0,0,1,1)}
= x_1x_2 x_3 E^{(456123)}_{(1,0,0,1,1,0)}
= x_1x_2 x_3 x_4 E^{(561234)}_{(0,0,1,1,0,0)}, 
$$
The fifth box is produced by applying Proposition \ref{boxbybox} 
to $E^{(123456)}_{(2,2,1,1,0,0)}$ to obtain
$$E^{(561234)}_{(0,0,1,1,0,0)}
= x_1 E^{(562341)}_{(0,0,1,0,0,0)}
+\Big( \frac{1-t}{1-qt^{5-2}}\Big) \big(
qx_6 E^{(512346)}_{(0,0,1,0,0,0)}
+qx_5 E^{(612345)}_{(0,0,1,0,0,0)}\big).
$$
The last box is obtained by  applying Proposition \ref{boxbybox}  to each of the 
terms $E^{(562341)}_{(0,0,1,0,0,0)}$, 
$E^{(512346)}_{(0,0,1,0,0,0)}$ and 
$E^{(612345)}_{(0,0,1,0,0,0)}$ which have been generated in the previous step:
\begin{align*}
E^{(562341)}_{(0,0,1,0,0,0)}
&= x_2 + \Big(\frac{1-t}{1-qt^{6-2}}\Big)\big( qtx_6+qtx_5\big), \\
E^{(512346)}_{(0,0,1,0,0,0)}
&= x_2 + \Big(\frac{1-t}{1-qt^{6-2}}\Big)\big( x_1 + qtx_5\big), \\
E^{(612345)}_{(0,0,1,0,0,0)}
&= x_2 + \Big(\frac{1-t}{1-qt^{6-2}}\Big)\big( x_1 + q x_6\big),
\qquad \hbox{since $E^z_{(0,0,0,0,0,0)} = 1$ for $z\in S_n$.}
\end{align*}
Compilng these produces an expansion of $E_{(2,2,1,1,0,0)}$ with 9 terms
\begin{align*}
E^{(123456)}_{(2,2,1,1,0,0)}
&= x_1x_2 x_3 x_4 E^{(561234)}_{(0,0,1,1,0,0)}, \\
&= x_1x_2x_3x_4\left(x_1 E^{(562341)}_{(0,0,1,0,0,0)}
+\Big( \frac{1-t}{1-qt^{5-2}}\Big) \big(
qx_6 E^{(512346)}_{(0,0,1,0,0,0)}
+qx_5 E^{(612345)}_{(0,0,1,0,0,0)}\big) \right) \\
&= x_1x_2x_3x_4 \left(
\begin{array}{l}
x_1 \Big( x_2 + \Big(\frac{1-t}{1-qt^{6-2}}\Big)\big( qtx_6+qtx_5\big) \Big) \\
+\Big( \frac{1-t}{1-qt^{5-2}}\Big) qx_6 
\Big(x_2 + \Big(\frac{1-t}{1-qt^{6-2}}\Big)\big( x_1 + qtx_5\big) \Big) \\
+\Big( \frac{1-t}{1-qt^{5-2}}\Big) qx_5 
\Big(x_2 + \Big(\frac{1-t}{1-qt^{6-2}}\Big)\big( x_1 + q x_6\big) \Big) 
\end{array}
\right) 
\end{align*}
These 9 terms are exactly the 9 nonattacking fillings of $\mu = (2,2,1,1,0,0)$ as follows
$$
\begin{array}{ccc}
\begin{array}{c|cc}
1 &1 &1 \\
2 &2 &2 \\
3 &3 \\
4 &4 \\
5 \\
6
\end{array}
&
\begin{array}{c|cc}
1 &1 &1 \\
2 &2 &6 \\
3 &3 \\
4 &4 \\
5 \\
6
\end{array}
&\begin{array}{c|cc}
1 &1 &1 \\
2 &2 &5 \\
3 &3 \\
4 &4 \\
5 \\
6
\end{array}
\\
x_1x_2x_3x_4x_1x_2
&x_1x_2x_3x_4x_1\Big(\frac{1-t}{1-qt^{6-2}}\Big) qt x_6
&x_1x_2x_3x_4 x_1 \Big(\frac{1-t}{1-qt^{6-2}}\Big)qt x_5

\\
\\
\begin{array}{c|cc}
1 &1 &6 \\
2 &2 &2 \\
3 &3 \\
4 &4 \\
5 \\
6
\end{array}
&
\begin{array}{c|cc}
1 &1 &6 \\
2 &2 &1 \\
3 &3 \\
4 &4 \\
5 \\
6
\end{array}
&\begin{array}{c|cc}
1 &1 &6 \\
2 &2 &5 \\
3 &3 \\
4 &4 \\
5 \\
6
\end{array}
\\
x_1x_2x_3x_4\Big( \frac{1-t}{1-qt^{5-2}}\Big) q x_6x_2
&x_1x_2x_3x_4\Big( \frac{1-t}{1-qt^{5-2}}\Big) q  x_6\Big(\frac{1-t}{1-qt^{6-2}}\Big) x_1
&x_1x_2x_3x_4\Big( \frac{1-t}{1-qt^{5-2}}\Big) q  x_6\Big(\frac{1-t}{1-qt^{6-2}}\Big) qt x_5 \\
\\
\begin{array}{c|cc}
1 &1 &5 \\
2 &2 &2 \\
3 &3 \\
4 &4 \\
5 \\
6
\end{array}
&
\begin{array}{c|cc}
1 &1 &5 \\
2 &2 &1 \\
3 &3 \\
4 &4 \\
5 \\
6
\end{array}
&\begin{array}{c|cc}
1 &1 &5 \\
2 &2 &6 \\
3 &3 \\
4 &4 \\
5 \\
6
\end{array}
\\
x_1x_2x_3x_4\Big( \frac{1-t}{1-qt^{5-2}}\Big) q x_5x_2
&x_1x_2x_3x_4\Big( \frac{1-t}{1-qt^{5-2}}\Big) q x_5 \Big(\frac{1-t}{1-qt^{6-2}}\Big)x_1
&x_1x_2x_3x_4\Big( \frac{1-t}{1-qt^{5-2}}\Big) q x_5 \Big(\frac{1-t}{1-qt^{6-2}}\Big) qx_6
\end{array}
$$
In this table, the weight $\wt(T)$ of the nonattacking filling is shown directly below the filling.
These are exactly the weights produced by iterating the box by box recursion.

The following is a corollary of Lemma \ref{compressionlemma}
(specifically, the step in line \eqref{compressionlineA}).
Lemma \ref{compressionlemma} is a version of \cite[Proposition 4.1]{Len08}, which
forms the core of the proof of the box-by-box recursion Proposition \ref{boxbybox}.

\begin{cor} Let $\mu\in \ZZ_{\ge 0}^n$ and $z\in S_n$.  Let 
$\psi\colon \mathrm{AW}^z_\mu \to \mathrm{NAF}^z_\mu$ be the compression function
defined in \eqref{compressionfcn} and let $T\in \mathrm{NAF}^z_\mu$.
Then
$$\wt(T) = \sum_{p\in \psi^{-1}(T)} \wt(p).$$
\end{cor}


\section{The affine Weyl group and the element $u_\mu$}

The underlying permutation combinatorics that controls Macdonald polynomials is
that of $n$-periodic permutations.  In this section we define the group of 
$n$-periodic permutations (the affine Weyl group), and establish notations and
facts about inversions and lengths of $n$-periodic permutations.  At the end of this
section we introduce the special $n$-periodic permutation $u_\mu$,
which is used for the construction of the Macdonald polynomial $E_\mu$.
Proposition  \ref{KSredwd}
provides a favorite reduced word for $u_\mu$ (the box-greedy reduced word)
and determines the inversions of $u_\mu$.

\subsection{The finite Weyl group $W_{\mathrm{fin}}=S_n$ and
the affine Weyl group $W=W_{GL_n}$}\label{affWsection}

The \emph{finite Weyl group} is 
$$\hbox{$W_{\mathrm{fin}}= S_n$,\quad the symmetric group of bijections
$v\colon \{1, \ldots, n\} \to \{1, \ldots, n\}$}
$$
with operation of composition of functions.
The \emph{type $GL_n$ affine Weyl group}
$W$ is the group of \emph{$n$-periodic permutations} $w\colon \ZZ\to \ZZ$ i.e.,
\begin{equation}
\hbox{bijective functions $w\colon \ZZ\to \ZZ$ such that $w(i+n) = w(i)+n$.}
\label{nperiodicdefn}
\end{equation}
Any $n$-periodic permutation $w$ is determined by its values $w(1), \ldots, w(n)$.
Using $w(i+n) = w(i)+n$, any permutation $v\colon \{1, \ldots, n\} \to \{1, \ldots, n\}$ in $S_n$
extends to an $n$-periodic permutation in $W$, 
and so $S_n \subseteq W$.

Define $\pi\in W$ by
\begin{equation}
\pi(i) = i+1, \quad\hbox{for $i\in \ZZ$.}
\label{pidefn}
\end{equation}
Define $s_0, s_1, \ldots, s_{n-1}\in W$ by
\begin{equation}
\begin{array}{l}
s_i(i) = i+1, \\
s_i(i+1) = i,
\end{array}\qquad\hbox{and}\quad s_i(j) = j\ \ \hbox{for
$j\in \{0, 1, \ldots, i-1, i+2, \ldots, n-1\}.$}
\label{sidefn}
\end{equation}
The finite Weyl group $S_n$ is the subgroup of $W$ generated by $s_1, \ldots, s_{n-1}$.

For $\mu = (\mu_1, \ldots, \mu_n)\in \ZZ^n$ define $t_\mu\in W$ by
\begin{equation}
t_\mu(1) = 1+n\mu_1,\quad t_\mu(2) = 2+n\mu_2,\quad \ldots, \quad t_\mu(n) = n+n\mu_n.
\label{tmudefn}
\end{equation}
Then
\begin{equation}
W = \{ t_\mu v\ |\ \mu\in \ZZ^n, v\in S_n\}
\qquad\hbox{with}\qquad v t_\mu = t_{v\mu} v\ \ \hbox{for $v\in S_n$ and $\mu\in \ZZ^n$.}
\label{translpresW}
\end{equation}
The map
\begin{equation}
\overline{\phantom{T}}\colon W \to S_n
\quad\hbox{given by}\qquad
\overline{t_\mu v} = v,\quad\hbox{for $\mu\in \ZZ^n$ and $v\in S_n$,}
\label{directionhom}
\end{equation}
is a surjective group homomorphism.
If $v\in S_n$ and $\mu = (\mu_1, \ldots, \mu_n)\in \ZZ^n$ then $(t_\mu v)(i) = v(i)+n\mu_i$
for $i\in \{1, \ldots, n\}$.  The \emph{two-line notation} for $w=t_\mu v$ is 
\begin{equation}
t_\mu v = \begin{pmatrix} 1 &2 &\cdots &n \\
v(1) + n\mu_{v(1)} &v(2) + n\mu_{v(2)} &\cdots &v(n) + n\mu_{v(n)} \end{pmatrix}.
\label{twolinenotation}
\end{equation}
Another useful notation for $n$-periodic permutations is an extended \emph{one-line-notation}: 
If $\mu = (\mu_1, \ldots, \mu_n)\in \ZZ^n$ and $v\in S_n$ write
\begin{equation}
t_\mu v = ((\mu_1)_{v^{-1}(1)}, (\mu_2)_{v^{-1}(2)}, \ldots, (\mu_n)_{v^{-1}(n)}).
\label{onelinenotation}
\end{equation}
For example, if $\mu = (0,4,5,1,4)$ with $n=5$
and
$ v = \begin{pmatrix} 1 & 2 &3 &4 &5 \\ 1 &4 &2 &5 &3 \end{pmatrix}$ then
$$t_\mu v = (0_1,4_3, 5_5, 1_2, 4_4) = \begin{pmatrix}
1 &2 &3 &4 &5 \\ 1 &4+n &2+4n &5+4n &3+5n \end{pmatrix}
= \begin{pmatrix} 1 &2 &3 &4 &5 \\ 1 &9 &22 &25 &28 \end{pmatrix}.
$$

\subsection{Inversions of $n$-periodic permutations}

Let $w\in W$ be an $n$-periodic permutation.  
An \emph{inversion of $w$} is 
$$\hbox{$(j,k)$\quad with\quad $j<k$ and $w(j)>w(k)$.}
$$
If $(j,k)$ is an
inversion of $w$ then $(j+\ell n, k+\ell n)$ is an inversion of $w$ for $\ell\in \ZZ$ and
so it is sensible to assume $j\in \{1,\ldots, n\}$ and define
$$
\mathrm{Inv}(w) = \{ (j,k)\ |\  \hbox{$j\in \{1, \ldots, n\}$, $k\in \ZZ$, $j<k$  and $w(j)>w(k)$} \}.
$$
The number of elements of $\mathrm{Inv}(w)$,
$$
\ell(w) = \#\mathrm{Inv}(w), \quad \hbox{is the \emph{length of $w$}.}
$$
For notational convenience when working with reduced words, let 
$s_\pi = \pi$.  Then
$$\ell(s_\pi) = \ell(\pi) = 0
\quad\hbox{and}\qquad \ell(s_i) = 1\ \ \hbox{for $i\in \{1, \ldots, n-1\}$.}
$$
Let $w\in W$.  A \emph{reduced word for $w$} is an expression of $w$ as a product
of $s_1, \ldots, s_{n-1}$ and $s_\pi$,
$$w= s_{i_1}\ldots s_{i_\ell}
\quad\hbox{with $i_1, \ldots, i_\ell\in \{1, \ldots, n-1, \pi\}$}
\quad\hbox{such that}\qquad \ell(w) = \ell(s_{i_1})+\cdots+\ell(s_{i_\ell}).$$

\subsection{Affine coroots and the root sequence of a reduced word}\label{affcorootssection}

Let $\fa_\ZZ$ be the set of $\ZZ$-linear combinations of symbols $\varepsilon_1^\vee,\ldots, \varepsilon_n^\vee, K$.  The \emph{affine coroots} are 
$$\varepsilon_i^\vee- \varepsilon_j^\vee + \ell K
\qquad\hbox{with $i,j\in \{1, \ldots, n\}$ and $i\ne j$ and $\ell\in \ZZ$}
$$
(in the context of the corresponding affine Lie algebra the symbol $K$ is the central element).
The \emph{shift} and \emph{height} of an affine coroot are given by
\begin{equation}
\mathrm{sh}(\varepsilon_i^\vee - \varepsilon_j^\vee+\ell K) = -\ell
\qquad\hbox{and}\qquad
\mathrm{ht}(\varepsilon_i^\vee - \varepsilon_j^\vee + \ell K) = j-i.
\label{GLnshhtdefn}
\end{equation}
The affine coroot corresponding to an inversion
\begin{equation}
(i,k) = (i, j+\ell n)\quad \hbox{with $i,j\in \{1, \ldots, n\}$ and $\ell\in \ZZ$,\qquad is}\quad
\beta^\vee = \varepsilon_i^\vee - \varepsilon_j^\vee + \ell K.
\label{invtoaffineroots}
\end{equation}
Define a $\ZZ$-linear action of the affine Weyl group $W$ on $\fa_\ZZ$ by
\begin{equation}
\pi^{-1}\varepsilon_1^\vee = \varepsilon^\vee_n+K,
\quad
\pi^{-1}\varepsilon_i^\vee = \varepsilon^\vee_{i-1}\ \ \hbox{for $i\in \{2, \ldots, n\}$},
\label{WXonaZ}
\end{equation}
$$s_i\varepsilon_i^\vee = \varepsilon^\vee_{i+1}, 
\quad
s_i\varepsilon_{i+1}^\vee = \varepsilon^\vee_i,
\quad
s_i\varepsilon_j = \varepsilon^\vee_j
\ \ \hbox{if $j\in \{1,\ldots, n\}$ and $j\not\in\{i,i+1\}$.}
$$
If $\mu = (\mu_1, \ldots, \mu_n)\in \ZZ^n$ then
$t_\mu \varepsilon^\vee_i = \varepsilon^\vee_i-\mu_i K$.
This action comes from the double affine Hecke algebra results in Proposition \ref{intertwiners}
and equation \eqref{Ypasttau}.

Let
$$
\alpha_0^\vee = \varepsilon_n^\vee-\varepsilon_1^\vee+K,
\qquad\hbox{and}\qquad
\alpha_i^\vee = \varepsilon_i^\vee - \varepsilon_{i+1}^\vee\ \ \hbox{for $i\in \{1, \ldots, n-1\}$.}
$$
Let $w\in W$ and let $w=s_{i_1}\cdots s_{i_\ell}$ be a reduced word for $w$.
The \emph{root sequence} of the reduced word $w = s_{i_1}\cdots s_{i_\ell}$
(recall that $s_\pi = \pi$) is  
\begin{equation}
\hbox{the sequence $(\beta^\vee_k\ |\ \hbox{$k\in \{1, \ldots, \ell\}$ and $i_k\ne \pi$} \})$ 
given by}\qquad
\beta^\vee_k = s_{i_\ell}^{-1}\cdots s^{-1}_{i_{k+1}}\alpha^\vee_{i_k}.
\end{equation}
Then,
identifying inversions with affine coroots as in \eqref{invtoaffineroots},
\begin{equation}
\mathrm{Inv}(w)= \{ \beta^\vee_k\ |\ \hbox{$k\in \{1, \ldots, \ell\}$ and $k\ne \pi$} \}
\label{invsetascoroots}
\end{equation}
(see \cite[(2.2.9)]{Mac03} or \cite[Ch.\ VI \S 1 no.\ 6 Cor.\ 2]{Bou}).

\subsection{The element $u_\mu$ in the affine Weyl group}

Define an action of $W$ on $\ZZ^n$ by
\begin{align}
\pi(\mu_1, \ldots, \mu_n) &= (\mu_n+1, \mu_1, \ldots, \mu_{n-1})
\qquad\hbox{and} 
\label{WXonZn} \\
s_i(\mu_1, \ldots, \mu_n) 
&= (\mu_1, \ldots, \mu_{i-1}, \mu_{i+1}, \mu_i, \mu_{i+2}, \ldots, \mu_n),
\quad\hbox{for $i\in \{1, \ldots, n\}$.}
\nonumber
\end{align}
Let $u_\mu$ be the minimal length element of $W$ such that
\begin{equation}
u_\mu(0, 0, \ldots, 0) = (\mu_1, \ldots, \mu_n)
\qquad\hbox{and define\quad $v_\mu\in S_n$ by}
\quad u_\mu = t_\mu v_\mu^{-1},
\label{umudefn}
\end{equation}
where $t_\mu\in W$ is as defined in \eqref{tmudefn}.
As noted in  \cite[(2.4.3)]{Mac03}, 
$u_\mu$ is the minimal length element of the coset $t_\mu S_n$ in $W$
and the choice of the notation $u_\mu$ and $v_\mu$ for these elements follows 
that lead.  
Let $\lambda= ( \lambda_1\ge \cdots \ge \lambda_n)$ be the decreasing rearrangement of $\mu$
and let
\begin{equation}
\hbox{$z_\mu\in S_n$ \quad be minimal length such that $\mu = z_\mu\lambda$.}
\label{zmudefn}
\end{equation}


The following result is the translation of \cite[(2.4.1)-(2.4.5) and (2.4.14)(i) and (2.4.12)]{Mac03} 
to our current setting.

\begin{prop} \label{wmuexplicit}
Let $\mu= (\mu_1, \ldots, \mu_n) \in \ZZ_{\ge 0}^n$.\ 
Let $u_\mu$, $v_\mu$, $\lambda$ and $z_\mu$  be as defined in \eqref{umudefn}
and \eqref{zmudefn}.
\item[(a)] $v_\mu$ is the minimal length element of $S_n$ such that
$v_\mu \mu$ is (weakly) increasing.
\item[(b)] The permutation $v_\mu\colon \{1, \ldots, n\} \to \{1, \ldots, n\}$ is given by
\begin{equation*}
v_\mu(i) = 1+
\#\{i'\in \{1, \ldots, i-1\} \ |\ \mu_{i'}\le \mu_i \}  
+ \#\{i' \in \{i+1, \ldots, n\} \ |\ \mu_{i'}< \mu_i \},
\end{equation*}
\item[(c)]  The $n$-periodic permutations $u_\mu\colon \ZZ\to \ZZ$ and
$u^{-1}_\mu\colon \ZZ\to \ZZ$ 
are given by
$$u_\mu(i) = v^{-1}_\mu(i) + n\mu_i
\qquad\hbox{and}\qquad
u_\mu^{-1}(i)  = v_\mu(i)-n\mu_{v_\mu(i)}
\qquad\hbox{for $i\in \{1, \ldots, n\}$.}
$$
\item[(d)]  Let $\lambda$ 
be the decreasing rearrangement of $\mu$.
The lengths of $t_\mu$, $u_\mu$ and $v_\mu$ are given by
$$
\ell(t_\mu) = \ell(t_\lambda) = \sum_{i,j\in \{1, \ldots, n\}\atop i<j} \lambda_i-\lambda_j, 
\quad
\ell(v_\mu) = \#\{ i<j\ |\ \mu_i>\mu_j\}
\quad\hbox{and}\quad
\ell(u_\mu) = \ell(t_\mu)-\ell(v_\mu).
$$
\item[(e)]  
Let $i\in \{1, \ldots, n-1\}$. 
If $\mu_i\ne \mu_{i+1}$ so that $s_i\mu \ne \mu$ then
$u_{s_i\mu}=s_iu_\mu$ and
$v_{s_i\mu}=v_\mu s_i$ and 
\begin{align*}
\ell(u_{s_i\mu}) = \begin{cases}
\ell(u_\mu)+1, &\hbox{if $\mu_i > \mu_{i+1}$,} \\
\ell(u_\mu)-1, &\hbox{if $\mu_i < \mu_{i+1}$,} \\
\end{cases}
\quad\hbox{and}\quad
\ell(v_{s_i\mu}) = \begin{cases}
\ell(v_\mu)-1, &\hbox{if $\mu_i > \mu_{i+1}$,} \\
\ell(v_\mu)+1, &\hbox{if $\mu_i < \mu_{i+1}$,} \\
\end{cases}
\end{align*}
\item[(f)]  
With $\pi$ as in \eqref{pidefn}, then
$u_{\pi \mu} = \pi u_\mu$ and
$\ell(u_{\pi \mu}) = \ell(u_\mu)$ and
$$\ell(v_\mu)-\ell(v_{\pi\mu}) =(n-1)-2(v_\mu(n)-1).
$$
\end{prop}
\begin{proof}  
(c)  The first formula follows from $u_\mu = t_\mu v^{-1}_\mu$ and \eqref{tmudefn}.  
To verify the second formula:
$$u_\mu^{-1}u_\mu(i) = u_\mu^{-1}(v^{-1}_\mu(i)+n\mu_i)
=u_\mu^{-1}(v^{-1}_\mu(i))+ n\mu_i 
= v_\mu(v^{-1}_\mu(i))-n\mu_{v_\mu v^{-1}_\mu(i)}+n\mu_i = i.
$$
(d)  From the definition of $t_\mu$ and $\mathrm{Inv}(w)$,
$$
\mathrm{Inv}(t_\mu) = 
\Big(\bigcup_{i<j\atop \mu_i\ge \mu_j} \{ (i,j), (i,j+n), \ldots, (i,j+n(\mu_j-\mu_i-1))\}\Big)
\cup
\Big(\bigcup_{i<j\atop \mu_j<\mu_i} \{ (j,i+n), \ldots, (j,i+n(\mu_i-\mu_j))\}\Big)
$$
and so
$\displaystyle{\ell(t_\mu) = \#\mathrm{Inv}(t_\mu) = \sum_{i<j} \vert \mu_i-\mu_j\vert}$, 
which gives the first statement.  More generally,
\begin{align*}
\mathrm{Inv}(t_\mu v) &= 
\Big(\bigcup_{i<j, v(i)<v(j)\atop \mu_{v(i)}\ge \mu_{v(j)}} 
\{ (i,j), (i,j+n), \ldots, (i,j+n(\mu_j-\mu_i-1))\}\Big) \\
&\qquad
\cup
\Big(\bigcup_{i<j, v(i)>v(j)\atop \mu_{v(i)} \ge \mu_{v(j)} } 
\{ (i,j), (i,j+n), \ldots, (i,j+n(\mu_j-\mu_i))\}\Big) \\
&\qquad
\cup
\Big(\bigcup_{i<j, v(i)<v(j) \atop \mu_{v(i)}<\mu_{v(j)} } \{ ((j,i+n), \ldots, (j,i+n(\mu_i-\mu_j))\}\Big) \\
&\qquad
\cup
\Big(\bigcup_{i<j, v(i)>v(j) \atop \mu_{v(i)} < \mu_{v(j)} } \{ ((j,i+n), \ldots, (j,i+n(\mu_i-\mu_j-1))\}\Big)
\end{align*}
The length of $t_\mu v$ is $\ell(t_\mu v) = \#\mathrm{Inv}(t_\mu v)$.  
Thus the minimal length element of the coset $t_\mu S_n$ is the element
$t_\mu v^{-1}_\mu$ where, if $i<j$ then $v^{-1}_\mu(i)>v^{-1}_\mu(j)$ 
if $\mu_{v^{-1}_\mu(i)} < \mu_{v^{-1}_\mu(j)}$ and
$v^{-1}_\mu(i)<v^{-1}_\mu(j)$ if $\mu_{v^{-1}_\mu(i)} \ge \mu_{v^{-1}_\mu(j)}$.
Thus $v_\mu \mu = v_\mu (\mu_1, \ldots, \mu_n) = (\mu_{v^{-1}_\mu(1)}, \ldots, \mu_{v^{-1}_\mu(n)})$
is in weakly increasing order and $\ell(t_\mu) = \ell(u_\mu) + \ell(v_\mu)$.

\smallskip\noindent
(a) and (e): These now follow from the last line of the proof of (d).

\smallskip\noindent
(b) In order for $v_\mu$ to rearrange $\mu$ into increasng order $v_\mu$ must move
the $i$th part of $\mu$  to the position
just to the right of the number of parts of $\mu$ which are less than $\mu_i$,
or equal to $\mu_i$ and to the left of $\mu_i$.  

\smallskip\noindent
(f)
Write $\gamma = (\mu_1, \ldots, \mu_{n-1})$.  Then
\begin{align*}
\ell(v_\mu) &= \ell(v_\gamma) + \#\{ i\in \{1, \ldots, n-1\}\ |\  \mu_i>\mu_n\}
\qquad\hbox{and} \\
\ell(v_{\pi \mu}) &= \ell(v_\gamma) + \#\{ i\in \{1, \ldots, n-1\}\ |\  \mu_i<\mu_n+1\},
\qquad\hbox{giving}
\end{align*}
\begin{align*}
\ell(v_\mu)-\ell(v_{\pi\mu}) 
&= \#\{ i\in \{1, \ldots, n-1\}\ |\  \mu_i>\mu_n\} - \#\{ i\in \{1, \ldots, n-1\}\ |\  \mu_i<\mu_n+1\} \\
&= (n-1)- \#\{ i\in \{1, \ldots, n-1\}\ |\  \mu_i \le \mu_n\} - \#\{ i\in \{1, \ldots, n-1\}\ |\  \mu_i\le \mu_n\}  \\
&=(n-1)-2(v_\mu(n)-1),
\end{align*}
where the third equality follows from the description of $v_\mu(n)$ in 
(b).
%
\end{proof}

\subsection{The box-greedy reduced word for $u_\mu$}

Let $\mu = (\mu_1, \ldots, \mu_n)\in \ZZ_{\ge 0}^n$ and let $u_\mu$ be as defined in 
\eqref{umudefn}.
The \emph{box-greedy reduced word} for the element $u_\mu$ is 
the sequence $u^{ {}_\square}_\mu$ 
defined inductively by the conditions $u^\square_{(0,\ldots, 0)} = 1$ and
\begin{equation}
u^\square_{(0,\ldots, 0, 0, \mu_k,  \mu_{k+1}, \ldots, \mu_n)}
=s_{k-1}\cdots s_2s_1\pi u^\square_{(0,\ldots, 0, 0, \mu_{k+1}, \ldots, \mu_n, \mu_k-1)} 
\label{KSdefn}
\end{equation}
This is the reduced word for $u_\mu$ that is used implicitly in \cite{KS96}, \cite{Kn96}, \cite{Sa96}.
Under the action in \eqref{WXonZn}, the factor $s_{k-1}\cdots s_2s_1\pi$ which appears in  
\eqref{KSdefn} is an element of $W$ of minimal length which moves 
$(0,\ldots, 0, 0, \mu_k,  \mu_{k+1}, \ldots, \mu_n)$
to a composition with one less box.

\begin{prop} \label{KSredwd}
For a box $(i,j)\in dg(\mu)$ (i.e. $i\in \{1, \ldots, n\}$ and $j\in \{1, \ldots, \mu_i\}$)
define
\begin{align}
u_\mu(i,j) &= \#\{ i'\in \{1, \ldots, i-1\}\ |\ \mu_{i'}<j\le \mu_i \} 
+\#\{ i'\in \{i+1, \ldots, n\}\ |\ \mu_{i'} < j-1 < \mu_i \} )
\label{boxlength}
\\
R_\mu(i,j) &= \{ \varepsilon^\vee_{v_\mu(i)} - \varepsilon^\vee_1+(\mu_i-j+1)K,
\ldots, \varepsilon^\vee_{v_\mu(i)} - \varepsilon^\vee_{u_\mu(i,j)}+(\mu_i-j+1)K  \}
\label{boxinversions}
\end{align}
\item[(a)]
The box greedy reduced word for $u_\mu$ is 
$$u_\mu^\square = \prod_{(i,j)\in dg(\mu)} (s_{u_\mu(i,j)} \cdots s_1\pi)
$$
where the product is over the boxes of $\mu$ \emph{in increasing cylindrical wrapping order}.
\item[(b)]  The inversion set of $u_\mu$ is 
$$\mathrm{Inv}(u_\mu) = \bigcup_{(i,j)\in dg(\mu)} R_\mu(i,j).
\qquad\hbox{and}\qquad
\ell(u_\mu) = \sum_{(i,j)\in \mu} u_\mu(i,j).$$

\end{prop}
\begin{proof} 
Let $\mu = (0, \ldots, 0, \mu_k, \ldots, \mu_n)$ and let $\nu = \pi^{-1} s_1 s_2\cdots s_{k-1}\mu
= (0, \ldots, 0, \mu_{k+1}, \ldots, \mu_n, \mu_k-1)$.  
From the definition of $u_\mu(i,j)$ in \eqref{boxlength}, $u_\mu(k,1) = k-1$,
\begin{align*}
u_\mu(i,j) &= u_\nu(i-1,j) \ \hbox{for $i\in \{k+1, \ldots, n\}$,} 
\quad\hbox{and}\quad
u_\mu(k,j) = u_\nu(n,j-1), \ \hbox{if $j\in \{2, \ldots, \mu_k\}$.}
\end{align*}
which already establishes (a).
Then, using \eqref{wmuexplicit} gives $v_\mu(i) = i$ for $i\in \{1, \ldots, k-1\}$,
\begin{align*}
v_\mu(i) = v_\nu(i-1)\ \hbox{for $i\in \{k+1, \ldots, n\}$}
\qquad\hbox{and}\qquad 
v_\mu(k) = v_\nu(n).
\end{align*}
These expressions for $u_\mu(i,j)$ and $v_\mu(i)$ in terms of $u_\nu(i,j)$ and $v_\nu(i)$
establish that
\begin{align*}
R_\mu(i,j) &= R_\nu(i-1,j), \qquad\hbox{if $i\ne k$, and} \\
R_\mu(k,j) &= R_\nu(n,j-1), \qquad\hbox{if $j\in \{2, \ldots, \mu_k\}$.}
\end{align*}
It remains to compute $R_\mu(k,1)$.
Since $u_\nu^{-1}\varepsilon_i^\vee = v_\nu t^{-1}_\nu \varepsilon_i^\vee
= \varepsilon^\vee_{v_\nu(i)-n\nu_i} = \varepsilon^\vee_{v_\nu(i)} + \nu_i K$ then
\begin{align*}
R_\mu(k,1) 
&=\{u_\nu^{-1}\pi^{-1}  \alpha_1^\vee, \ldots, 
u_\nu^{-1}\pi^{-1} s_1s_2\cdots s_{k-2}\alpha_{k-1}^\vee\} \\
&=\{u_\nu^{-1}\pi^{-1} (\varepsilon^\vee_1-\varepsilon^\vee_2), \ldots, 
u_\nu^{-1}\pi^{-1} s_1s_2\cdots s_{k-2}(\varepsilon^\vee_{k-1} - \varepsilon^\vee_k) \} \\
&=\{ u_\nu^{-1} ((\varepsilon^\vee_n+K)-\varepsilon^\vee_1), \ldots, 
u_\nu^{-1}((\varepsilon^\vee_n+K)-\varepsilon^\vee_{k-1}) \} \\
&=\{ (\varepsilon^\vee_{v_\nu(n)}+\nu_nK + K) - (\varepsilon^\vee_1+\nu_1 K), \ldots, 
(\varepsilon^\vee_{v_\nu(n)} + \nu_n K + K) - (\varepsilon^\vee_{k-1}+\nu_{k-1} K)\} \\
&=\{ \varepsilon^\vee_{v_\mu(k)} - \varepsilon^\vee_1 + (\mu_k-1)K + K, \ldots, 
\varepsilon^\vee_{v_\mu(k)} - \varepsilon^\vee_{k-1}+(\mu_k-1) K + K \},
\end{align*}
where the next to last equality uses $\nu_1=\cdots=\nu_{k-1}=0$ and $\nu_n = \mu_k-1$.
\end{proof}


\begin{remark} \label{rootstolegarm}
\textbf{Relating affine roots to $\#\mathrm{Nleg}_\mu(b)$ and 
$\#\mathrm{Narm}_\mu(b)$.}
In the derivation of box-by-box recursion for relative Macdonald polynomials (Proposition \ref{boxbybox}), 
the last root in each box in the expression of $\mathrm{Inv}(u_\mu)$ in Proposition \ref{KSredwd}(b)
gets picked out (this is the  $d_{-\beta^\vee_{j-1}}$ and $f_{-\beta^\vee_{j-1}}$ in the proof of
Lemma \ref{compressionlemma}).
More precisely, 
for $(i,j)\in dg(\mu)$, let $\beta^\vee_\mu(i,j)$ be the last element of $R_\mu(i,j)$ in \eqref{boxinversions}:
$$\beta^\vee_\mu(i,j) = \varepsilon^\vee_{v_\mu(i)} - \varepsilon^\vee_{u_\mu(i,j)}  + (\mu_i-j+1)K.$$
With the \emph{shift and height of an affine coroot} as defined in
\eqref{GLnshhtdefn},
then
\begin{align*}
\mathrm{sh}(-\beta^\vee_\mu(i,j)) 
= \#\mathrm{Nleg}_\mu(i,j)+1, \quad\hbox{and} \quad
\mathrm{ht}(-\beta^\vee_\mu(i,j)) 
=\#\mathrm{Narm}_\mu(i,j)+1,
\end{align*}
since, by \eqref{boxlength} and Proposition \ref{wmuexplicit}(b),
\begin{align*}
v_\mu(i) - u_\mu(i,j) 
&= 1+\#\{ i'\in \{1, \ldots, i-1\} \ |\ j \le \mu_{i'}\le  \mu_i \} 
+ \#\{ i'\in \{i+1, \ldots, n\}\ |\ j-1 \le \mu_{i'} < \mu_i\}  \\
&=\#\mathrm{Narm}_\mu(i,j)+1.
\end{align*}
\end{remark}


\section{Type $GL_n$ Macdonald polynomials}\label{GLnMacsection}

In this section we define the Macdonald polynomials $E_\mu$ and provide
explicit fomulas for all $E_\mu$ for $\mu$ with less than 3 boxes.  These
examples are helpful for getting a feel for what Macdonald polynomials actually look like.
Although we have hidden the double affine Hecke algebra (DAHA) from our exposition in this section,
Section 5 derives, from scratch, all the formulas for the operators $T_i$ and $Y_i$ and 
the Macdonald polynomials $E_\mu$ which are efficiently pulled out of a hat in this section.

\subsection{The polynomial representation and Cherednik-Dunkl operators}

For $\mu = (\mu_1, \ldots, \mu_n)\in \ZZ^n$ let
\begin{equation}
x^\mu = x_1^{\mu_1}\cdots x_n^{\mu_n}.
\label{xmonomnotation}
\end{equation}
The Laurent polynomial ring
$
\CC[x_1^{\pm1}, \ldots, x_n^{\pm1}]$
has basis $\{ x^\mu\ |\ \mu\in \ZZ^n\}$
and the polynomial ring
$
\CC[x_1, \ldots, x_n]$
has basis $\{ x^\mu\ |\ \mu\in \ZZ_{\ge 0}^n\}$, indexed by the set $\ZZ_{\ge 0}^n$
of compositions.
The symmetric group
$$\hbox{$S_n$ acts on $\CC[x_1^{\pm1}, \ldots, x_n^{\pm1}]$ and $\CC[x_1, \ldots, x_n]$ 
by permuting the variables $x_1, \ldots, x_n$.}
$$
The symmetric group $S_n$ acts on $\ZZ^n$ by permuting the positions of the entries
so that $wx^\mu = x^{w\mu}$ for $w\in S_n$ and $\mu\in \ZZ^n$.

Let $q, t^{\frac12}\in \CC^\times$.  Following the notation of \cite[Ch.\ VI (3.1)]{Mac},
let $T_{q^{-1},x_n}$ be the operator on $\CC[x_1^{\pm1}, \ldots, x_n^{\pm1}]$ given by
$$(T_{q^{-1},x_n} h)(x_1, \ldots, x_n) = h(x_1, \ldots, x_{n-1}, q^{-1}x_n).$$
For $i\in \{1, \ldots, n-1\}$ let
$s_i$ be the transposition which switches $i$ and $i+1$.
Define operators $T_1, \ldots, T_{n-1}$, $g$ and $g^\vee$ on 
$\CC[x_1^{\pm1}, \ldots, x_n^{\pm1}]$ by 
\begin{equation}
T_i =t^{-\frac12} \big( t -\frac{tx_i - x_{i+1}}{x_i-x_{i+1}}(1-s_i) \big),
\qquad
g = s_1s_2\cdots s_{n-1} T_{q^{-1}, x_n}, \qquad
g^\vee = x_1T_1\cdots T_{n-1}.
\label{DAHAonCX}
\end{equation}
In \S \ref{GLnpolyderiv} we give the derivation of these operators from the
type $GL_n$ double affine Hecke algebra (DAHA). Except for the
factor of $t^{-\frac12}$ this is the operator in \cite[(2.3)]{BW19}, which appears in the
form \eqref{origpolyop} in \cite[(7)]{HHL06}).
The \emph{Cherednik-Dunkl operators} are 
\begin{equation}
Y_1 = g T_{n-1}\cdots T_1,
\quad Y_2 = T_1^{-1}Y_1T_1^{-1}, \quad Y_3 = T_2^{-1}Y_2T_2^{-1}, \quad \ldots,\quad
Y_n= T_{n-1}^{-1}Y_{n-1}T_n^{-1}.
\label{GLnCDopsdefn}
\end{equation}

\subsection{Macdonald polynomials}

Let $g^\vee$, $T_i$ and $Y_i$ be as in \eqref{DAHAonCX} and \eqref{GLnCDopsdefn}
and define
\begin{equation}
\tau_\pi^\vee = g^\vee,
\quad\hbox{and}\quad
\tau_i^\vee = T_i + \frac{t^{-\frac12}(1-t)}{1-Y^{-1}_i Y_{i+1}}\quad\hbox{for $i\in \{1, \ldots, n-1\}$.}
\label{earlyintertwinerdefn}
\end{equation}
Using the action of $s_1, \ldots, s_n, \pi$ on $\ZZ^n$ given in \eqref{WXonZn},
the \emph{(nonsymmetric) Macdonald polynomials $E_\mu$}, for $\mu\in \ZZ^n$, 
are determined by $E_0 = 1$,
\begin{equation}
\begin{array}{l}
E_{\pi\mu} 
=  t^{\#\{ i\in \{1, \ldots, n\}\ |\  \mu_i>\mu_n\} -\frac12(n-1)}\, \tau_\pi^\vee E_\mu,
\quad\hbox{and} \\
\\
E_{s_i\mu} = t^{\frac12}\, \tau_i^\vee E_\mu
\quad \hbox{if $i\in \{1, \ldots, n-1\}$ and $\mu_{i+1}>\mu_i$.}
\end{array}
\label{recursion}
\end{equation}

\begin{remark}
The source of the strange coefficients in \eqref{recursion} is Proposition \ref{wmuexplicit} (e) and (f)
which gives that $-\frac12(\ell(v_{s_i\mu})-\ell(v_{\mu})) = \frac12$ and
$-\frac12(\ell(v_{\pi\mu})-\ell(v_{\mu})) = \frac12(n-1) - (v_\mu(n)-1) 
= \frac12(n-1) - \#\{i\in \{1, \ldots, n-1\}\ |\ \mu_i\le \mu_n\}$.  The role of these coefficients is
to force the coefficient of $x^\mu$ in $E_\mu$ to be $1$.
\end{remark}

The following theorem, the type $GL_n$ case of \cite[Theorem 4.1 and Proposition 4.2]{Che95},
shows that the $E_\mu$ are simultaneous eigenvectors of the Cherednik-Dunkl operators.
We provide a proof in Theorem \ref{KSGLnformularepeat} of this paper.

\begin{thm} \label{Macpolyconst}  Let $\mu\in \ZZ^n$ and let $v_\mu\in S_n$ be the permutation
that rearranges $\mu$ into (weakly) increasing order.  Then $E_\mu$ is the unique element of 
$\CC[x_1^{\pm1}, \ldots, x_n^{\pm1}]$ such that 
$$\hbox{if $i\in \{1, \ldots, n\}$\quad then\quad}
Y_i E_\mu = q^{-\mu_i}t^{-(v_\mu(i)-1)+\frac12(n-1)}E_\mu,
$$
and the coefficient of $x^\mu$ in $E_\mu$ is 1.
\end{thm}

Let $\lambda = (\lambda_1\ge \cdots \ge \lambda_n)\in \ZZ^n$.  
\begin{equation}
\hbox{The \emph{symmetric Macdonald polynomial} $P_\lambda$ is}\qquad
P_\lambda = \sum_{\nu\in S_n\lambda} t^{\frac12\ell(z_\nu)}T_{z_\nu} E_\lambda,
\label{Pdefn}
\end{equation}
where the sum is over rearrangements $\nu$ of $\lambda$ and
$z_\nu\in S_n$ is minimal length such that $\nu = z_\nu\lambda$. \hfil\break
Let $\mu = (\mu_1, \ldots, \mu_n)$ and let $z\in S_n$. 
\begin{equation}
\hbox{The \emph{relative Macdonald polynomial} $E_\mu^z$  is}\qquad
E^z_\mu = t^{-\frac12(\ell(zv_\mu^{-1})-\ell(v_\mu^{-1}))}T_zE_\mu.
\label{Ezmudefn}
\end{equation}
These definitions follow \cite[Remarks after (6.8)]{Mac95},
\cite[(5.7.6), (5.7.7)]{Mac03} , \cite[Definition 4.4.2]{Fe11}, \cite[Definition 5]{Al16}
and \cite[(2.8)]{FMO16}
(Ferreira references private communication with Haglund).
In \cite{Al16}, the $E^z_\mu$ are called \emph{permuted basement Macdonald polynomials}.

\begin{remark} \label{Emusymmremark}
The following properties of the $E_\mu$ are proved in Proposition \ref{Emusymmetries}:
$$E_{(\mu_n+1, \mu_1, \ldots, \mu_{n-1})}
=q^{\mu_n}x_1E_\mu(x_2, \ldots, x_n, q^{-1}x_1),
$$
$$
E_{(\mu_1+1, \ldots, \mu_n+1)} = x_1\cdots x_n E_{(\mu_1, \ldots, \mu_n)},
\qquad
E_{(-\mu_n, \ldots, -\mu_1)}(x_1, \ldots, x_n;q,t) 
= E_\mu(x^{-1}_n, \ldots, x^{-1}_1;q,t).
$$
\end{remark}

\begin{remark}  In generalization of \eqref{Ezmudefn}, one could,
for any $\mu\in \ZZ^n$ and any $n$-periodic permutation $z\in W$,
define $E_\mu^z = (const) T_zE_\mu$,
where $(const)$ is a constant determined by requiring 
the coefficient of $x^{z\mu}$ in $E_\mu^z$ to be 1.  A more useful alternative might
be to define $E^z_\mu = X^z E_\mu = X^z \tau_{\mu}\mathbf{1}$ in the notation of
\cite[(2.26) and Theorem 2.2]{RY08}.
\end{remark}

%

\subsection{Explicit $E_\mu$ with less than three boxes}

The following explicit formulas for $E_\mu$ with 1 and 2 boxes already provide enough data
that one might have a chance at guessing the nonattacking fillings formula.

\begin{prop} \label{Emu1box}  Let $\varepsilon_i = (0,\ldots, 0, 1, 0, \ldots, 0)$ be the sequence
with $1$ in the $i$th component.
\item[(a)] 
If $i\in \{1, \ldots, n\}$ then
$$E_{\varepsilon_i} 
= x_i + \frac{(1-t)}{(1-qt^{n-(i-1)})} (x_{i-1}+\cdots + x_1).
$$
\item[(b)]  If $i\in \{1, \ldots, n\}$ then
 \begin{align*}
E_{2\varepsilon_i} &= x_i^2+ \Big(\frac{1-t}{1-q^2t^{n-(i-1)}}\Big) \sum_{k\in \{1, \ldots, i-1\}}  x_k^2 
+ \Big(\frac{1-t}{1-qt}\Big) q \sum_{\ell\in \{i+1,\ldots, n\}}  x_ix_\ell 
\\
&\qquad
+ \Big(\frac{1-t}{1-qt}\Big) \Big(1+  \Big(\frac{1-t}{1-q^2t^{n-(i-1)} }\Big) q \Big) 
\sum_{k \in\{1, \ldots, i-1\} } x_k x_i 
\\
&\qquad 
+ \Big(\frac{1-t}{1-qt}\Big) \Big(\frac{1-t}{1-q^2t^{n-(i-1)} }\Big)(1 +q)
\sum_{\{k, \ell\}\subseteq \{1, \ldots, i-1\}}  x_k x_\ell  \\
&\qquad
+\Big( \frac{1-t}{1-q^2t^{n-(i-1)} }\Big) \Big(\frac{1-t }{ 1-qt }\Big)q
\sum_{k\in \{1,\ldots, i-1\}} \sum_{\ell\in \{i+1, \ldots, n\}} x_kx_\ell,
\end{align*}
\item[(c)] If $j_1,j_2\in \{1, \ldots, n\}$ with $j_1<j_2$ then
\begin{align*}
E_{\varepsilon_{j_1}+\varepsilon_{j_2} }
&= x_{j_1}x_{j_2} 
+ \Big( \frac{1-t}{1-qt^{n-j_1} } \Big)  \sum_{k=1}^{j_1-1} x_k x_{j_2} 
+ \Big( \frac{1-t}{1-qt^{n-(j_2-2)} } \Big) \sum_{\ell=j_1+1}^{j_2-1} x_{j_1}x_\ell  \\
&\qquad
+ \Big( \frac{1-t}{1-qt^{n-(j_2-2)} }\Big) \Big(\frac{1-t}{1-qt^{n-j_1} }+t\Big) 
\sum_{k = 1}^{j_1-1} x_k x_{j_1} \\
&\qquad + \Big( \frac{1-t}{1-qt^{n-(j_2-2)} } \Big) \Big( \frac{1-t}{1-qt^{n-j_1} } \Big)
\sum_{k=1}^{j_1-1} \sum_{\ell=j_1+1}^{j_2-1} x_k x_\ell \\
&\qquad + \Big(\frac{1-t}{1-qt^{n-(j_2-2)} } \Big)
\Big(  \frac{1-t}{1-qt^{n-j_1} } \Big)(1+t)
\sum_{\{k,\ell\}\subseteq  \{1,\ldots, j_1-1\}}   x_k x_\ell 
\end{align*}
\end{prop}
\begin{proof}  
Using the first identity in \eqref{DAHAonCX}, if $r\in \{1, \ldots, n\}$ then
\begin{equation}
t^{\frac12}T_i(x_r)
= \begin{cases}
x_{i+1}, &\hbox{if $r=i$,} \\
tx_i + (t-1)x_{i+1}, &\hbox{if $r=i+1$,}  \\
tx_r, &\hbox{otherwise.}
\end{cases}
\label{Tiondeg1}
\end{equation}
and, assuming $r,s\in \{1, \ldots, n\}$ with $r<s$ then
$$t^{\frac12}T_i(x_rx_s)
= \begin{cases}
x_rx_{i+1}, &\hbox{if $s=i$,} \\
tx_rx_i + (t-1)x_rx_{i+1}, &\hbox{if $s=i+1$ and $r<i$,} \\
tx_i x_{i+1}, &\hbox{if $r=i$ and $s=i+1$}, \\
x_{i+1}x_s, &\hbox{if $r=i$ and $s>i+1$,} \\
(t-1)x_{i+1}x_s + tx_ix_s, &\hbox{if $r=i+1$ and $s>i+1$,} \\
tx_rx_s, &\hbox{otherwise.}
\end{cases}
$$
and, if $r\in \{1,\ldots, n\}$ then
\begin{equation}
t^{\frac12}T_i(x_r^2)
= \begin{cases}
x_{i+1}^2+(1-t)x_ix_{i+1}, &\hbox{if $r=i$,} \\
tx_i^2+(t-1)x_{i+1}^2+(1-t)x_ix_{i+1}, &\hbox{if $r=i+1$,} \\
tx_r^2, &\hbox{otherwise.}
\end{cases}
\label{Tiondeg2}
\end{equation}
(a) The proof is by induction on $i$.  The base case $i=1$ is
$$E_{\varepsilon_1} 
= t^{-\frac12(n-1)}\tau_\pi^\vee \mathbf{1}
= t^{-\frac12(n-1)}X_1T_1 \cdots T_{n-1} \mathbf{1} = x_1.
$$
For the induction step (note that 
$Y_i^{-1}Y_{i+1}E_{\varepsilon_i} = q^{1-0}t^{n-i}E_{\varepsilon_i}$) and use
\begin{align*}
E_{\varepsilon_{i+1}} 
&= t^{\frac12}\tau_i^\vee E_{\varepsilon_i}
= \Big(t^{\frac12}T_i + \frac{1-t}{1-qt^{n-i}}\Big) E_{\varepsilon_i} 
\end{align*}

\noindent
(b) The proof is by induction on $i$.  Using part (a) and the first 
identity in Remark \ref{Emusymmremark} applied to $E_{\varepsilon_n}$,
$$
E_{2\varepsilon_1} 
= x_1^2 + \frac{1-t}{1-qt} q (x_1x_n+\cdots+x_1x_2).
$$
and this provides the base of the induction.  
Then use $Y_i^{-1}Y_{i+1}E_{2\varepsilon_i} = q^{2-0}t^{n-i}E_{2\varepsilon_i}$ and
\begin{align*}
&E_{2\varepsilon_{i+1}} 
= \Big(t^{\frac12}T_i + \frac{1-t}{1-Y_i^{-1}Y_{i+1}}  \Big) E_{2\varepsilon_i} 
= \Big(t^{\frac12}T_i + \frac{1-t}{1-q^2t^{n-i}} \Big) E_{2\varepsilon_i}.
\end{align*}

\smallskip\noindent
(c) The proof is by induction on $j_1$.  From part (a) and the first 
identity in Remark \ref{Emusymmremark} applied to $E_{\varepsilon_{j_2-1}}$,
\begin{align*}
E_{\varepsilon_1+\varepsilon_{j_2}} 
&= x_1x_{j_2} 
+ \Big( \frac{1-t}{1-qt^{n-(j_2-2)} }\Big)  (x_1x_{j_2-1} +\cdots +x_1x_3 + x_1x_2),
\end{align*}
and this provides the base of the induction.
Then use
\begin{align*}
E_{\varepsilon_{j_1}+\varepsilon_{j_2} }
&= 
\Big( t^{\frac12}T_{j_1-1} + \frac{1-t}{1-qt^{n-j_1} } \Big) 
E_{\varepsilon_{j_1-1}+\varepsilon_{j_2}}.
\end{align*}
\end{proof}


\section{Recursions for computing $E^z_\mu$}

In this section we derive the recursions which are used to produce expansions of
Macdonald polynomials in terms of monomials.  These computations are extensions of
the defining recursions given in \eqref{recursion}.
It will be helpful to record carefully the action of $t^{\frac12}\tau^\vee_i$ and $t^{\frac12}T_i$
on the Macdonald polynomials $E_\mu$ as follows.

Let $\mu\in \ZZ^n$ and, with notations as in Theorem \ref{Macpolyconst},  let
$$
\begin{array}{l}
a_\mu =q^{\mu_i-\mu_{i+1}}t^{v_\mu(i)-v_\mu(i+1)} , \\
a_{s_i\mu} =q^{\mu_{i+1}-\mu_i} t^{v_\mu(i+1)-v_\mu(i)} ,
\end{array}
\qquad\hbox{and}\qquad
D_\mu = \frac{(1-ta_\mu)(1-ta_{s_i\mu})}{(1-a_\mu)(1-a_{s_i\mu})}.
$$
Assume that $\mu_i>\mu_{i+1}$.
Using the identity $E_{s_i\nu} = t^{\frac12}\tau^\vee_iE_\nu$ if $\nu_{i+1}>\nu_i$ from
\eqref{recursion}, the eigenvalue from Theorem \ref{Macpolyconst},
and \eqref{tausquared} gives
\begin{equation}
\begin{array}{l}
Y_i^{-1}Y_{i+1} E_\mu = a_\mu E_\mu, \\
Y_i^{-1}Y_{i+1} E_{s_i\mu} = a_{s_i\mu}E_{s_i\mu},
\end{array}
\quad
\begin{array}{l}
t^{\frac12}\tau^\vee_i E_\mu = E_{s_i\mu}, \\
t^{\frac12}\tau^\vee_i E_{s_i\mu} = D_\mu E_\mu,
\end{array}
\quad\hbox{and}
\quad
\begin{array}{l}
t^{\frac12}T_i E_\mu = -\frac{1-t}{1-a_\mu}E_\mu + E_{s_i\mu}, \\
t^{\frac12}T_i E_{s_i\mu} = D_\mu E_\mu + \frac{1-t}{1-a_{s_i\mu}} E_{s_i\mu}.
\end{array}
\label{CXlambdaaction}
\end{equation}
Now assume that $\mu_i=\mu_{i+1}$.  Then
$v_\mu(i+1) = v_\mu(i)+1$ and $a_\mu = t^{-1}$ so that
\begin{equation}
Y_i^{-1}Y_{i+1} E_\mu = t^{-1} E_\mu, 
\qquad (t^{\frac12}\tau^\vee_i) E_\mu = 0,
\qquad\hbox{and}\qquad (t^{\frac12}T_i) E_\mu = t E_\mu.
\label{Tigivest}
\end{equation}

\subsection{Step by step recursion for computing $E^z_\mu$}

Proposition \ref{stepbystep}(a) is used to
reduce the number of boxes in $\mu$ and part (b)
is used to reduce the computation to decreasing $\mu$.  
Iterating these steps delivers a monomial expansion of $E^z_\mu$ as
a weighted sum of alcove walks $p$.
The permutation sequence $\vec z_p$ of the alcove walk which appears in 
\eqref{alcwalkpermseq} is the sequence $z_0,z_1, \ldots$ of permutations  which 
arise as superscripts of the  $E^{z_i}_\nu$ which occur in the intermediate applications 
of the step-by-step recursion to obtain the monomial expansion.

\begin{prop}  \label{stepbystep}
Let $\mu\in \ZZ^n$ and let $z\in S_n$. 
Let $v_\mu$ be the minimal length element of $S_n$ such that 
$v_\mu$ rearranges $\mu$ to be weakly increasing.
\item[(a)] If $\mu_1\ne 0$ then
$$E^z_\mu = x_{z(1)} E^{zc_n}_{(\mu_2, \ldots, \mu_n, \mu_1-1)},
\qquad\hbox{where $c_n = s_1\cdots s_{n-1}$ (an $n$-cycle in $S_n$).}
$$
\item[(b)]  Let $i\in \{1, \ldots, n-1\}$ such that $\mu_i<\mu_{i+1}$ and let
$$\beta^\vee = \varepsilon^\vee_{v_\mu(i+1)}-\varepsilon^\vee_{v_\mu(i)}+(\mu_{i+1}-\mu_i)K
\quad\hbox{so that}\quad
q^{\mathrm{sh}(-\beta^\vee)}t^{\mathrm{ht}(-\beta^\vee)} 
= q^{\mu_{i+1}-\mu_i }t^{v_\mu(i+1)-v_\mu(i)}.
$$
Let
$
\mathrm{norm}^z_\mu(i) = \hbox{$\frac12$}\big(\ell(zs_iv^{-1}_\mu)-\ell(zv^{-1}_\mu) - \ell(s_i) \big).
$
Then
$$E^z_\mu = \begin{cases}
 \displaystyle{ E^{zs_i}_{s_i\mu}
 +\Big(\frac{1-t}{1-q^{\mathrm{sh}(-\beta^\vee)}t^{\mathrm{ht}(-\beta^\vee)}  } \Big) 
 t^{\mathrm{norm}^z_\mu(i)} E^z_{s_i\mu}, }
&\hbox{if $z(i)<z(i+1)$,} \\
\\
 \displaystyle{ E^{zs_i}_{s_i\mu}
+\Big(\frac{ 1-t } {1- q^{\mathrm{sh}(-\beta^\vee)}t^{\mathrm{ht}(-\beta^\vee)}  }  \Big)
q^{\mathrm{sh}(-\beta^\vee)}t^{\mathrm{ht}(-\beta^\vee)} 
t^{\mathrm{norm}^z_\mu(i)}  E^z_{s_i\mu}, }
&\hbox{if $z(i)>z(i+1)$.}
\end{cases}
$$
\end{prop}
\begin{proof} 
(a)  By the second identity in Proposition \ref{TtoXconversion}, $T_z g^\vee = x_{z(1)} T_{zc_n}$ giving
\begin{align*}
T_z \tau_{u_\mu}\mathbf{1} = T_z \tau^\vee_\pi \tau^\vee_{u_{\pi^{-1}\mu}}\mathbf{1}
=T_z g^\vee \tau^\vee_{u_{\pi^{-1}\mu}}\mathbf{1}
=x_{z(1)} T_{zc_n} \tau^\vee_{u_{\pi^{-1}\mu}}\mathbf{1}.
\end{align*}
Then (a) follows by using
$E^z_\mu 
= t^{-\frac12 \ell(zv^{-1}_\mu) } T_z \tau^\vee_{u_\mu} \mathbf{1}
$
to rewrite each side, and computing
$$\ell(zv^{-1}_\mu) - \ell(zc_nv^{-1}_{\pi^{-1}\mu}) 
= \ell(\overline{zu_\mu}) - \ell(\overline{zc_n u_{\pi^{-1}\mu} })
= \ell(\overline{zu_\mu}) - \ell(\overline{zc_n \pi^{-1} u_\mu })
= \ell(\overline{zu_\mu}) - \ell(\overline{zc_n c^{-1}_n u_\mu }) =0,
$$
where $\overline{\phantom{T}}\colon W\to S_n$ is the homomorphism defined in 
\eqref{directionhom}.

\smallskip\noindent
(b) 
Let $\nu = s_i\mu$ and let $a_\nu = q^{\nu_i-\nu_{i+1}}t^{v_\nu(i)-v_\nu(i+1)}
=q^{\mathrm{sh}(-\beta^\vee)}t^{\mathrm{ht}(-\beta^\vee)}$.
Using \eqref{recursion}, \eqref{GLnintertwineralternative} and the eigenvalue
formula from \eqref{Macpolyconst}, then
\begin{align*}
T_z \tau^\vee_{u_\mu}\mathbf{1}
&= T_z \tau^\vee_i \tau^\vee_{u_{s_i\mu}} \mathbf{1}
=\begin{cases}
T_z \Big(T_i + \Big(\frac{t^{-\frac12}(1-t)}{1-a_\nu } \Big)\Big) \tau^\vee_{u_{s_i\mu}} \mathbf{1},
&\hbox{if $zs_i>z$,} \\
T_z \Big(T^{-1}_i + \Big(\frac{t^{-\frac12}(1-t)}{1-a_\nu } \Big)a_\nu \Big) \tau^\vee_{u_{s_i\mu}} \mathbf{1},
&\hbox{if $zs_i<z$,}
\end{cases}
\\
&= \begin{cases}
T_{zs_i} \tau^\vee_{u_{s_i\mu}} \mathbf{1}
+\Big(\frac{t^{-\frac12}(1-t)}{1-a_\nu } \Big) T_z \tau^\vee_{u_{s_i\mu}} \mathbf{1} ,
&\hbox{if $zs_i>z$,} \\
T_{zs_i} \tau^\vee_{u_{s_i\mu}} \mathbf{1}
+\Big(\frac{t^{-\frac12}(1-t)}{1-a_\nu } \Big)a_\nu T_z \tau^\vee_{u_{s_i\mu}} \mathbf{1}, 
&\hbox{if $zs_i<z$,}
\end{cases}
\end{align*}
Then, using
$E^z_\mu 
= t^{-\frac12 \ell(zv^{-1}_\mu) } T_z \tau^\vee_\mu \mathbf{1}
$
to obtain
$$
E^z_{\mu} = \begin{cases}
\displaystyle{
t^{\frac12(\ell(zs_iv^{-1}_{s_i\mu}) - \ell(zv^{-1}_\mu))} E^{zs_i}_{s_i\mu}  
+ \Big(\frac{1-t}{1-a_\nu } \Big) t^{\frac12 (-1+\ell(zs_iv^{-1}_\mu) - \ell(zv^{-1}_\mu) )} E^z_{s_i\mu},
}
&\hbox{if $zs_i>z$,}  \\
\displaystyle{
t^{\frac12(\ell(zs_iv^{-1}_{s_i\mu}) - \ell(zv^{-1}_\mu))} E^{zs_i}_{s_i\mu}  
+ \Big(\frac{1-t}{1-a_\nu } \Big) a_\nu t^{\frac12 (-1+\ell(zs_iv^{-1}_\mu) - \ell(zv^{-1}_\mu) )}
E^z_{s_i\mu},
}
&\hbox{if $zs_i<z$.}
\end{cases}
$$
By Proposition \ref{wmuexplicit}(e), $v^{-1}_{s_i\mu} = s_iv^{-1}_{\mu}$ and so 
$t^{\frac12(\ell(zs_iv^{-1}_{s_i\mu}) - \ell(zv^{-1}_\mu))}
=t^{\frac12(\ell(zs_is_iv^{-1}_{\mu}) - \ell(zv^{-1}_\mu))}=t^0=1$ and
the result follows.
\end{proof}

\subsection{Box by box recursion for computing $E^z_\mu$}

Proposition \ref{boxbybox} executes several steps of Proposition \ref{stepbystep} at once to
provide a recursion for computing $E^z_\mu$ which removes a box at each application of
the recursion.
Iterating this recursion delivers a monomial expansion of $E^z_\mu$ as
a weighted sum of nonattacking fillings $T$. 
The permutation sequence $\vec z_T$ of the nonattacking filling $T$ 
(see \eqref{permseqsemidefn})
is the sequence of permutations $z_0,z_1, \ldots$ which arise as superscripts of the 
$E^{z_i}_\nu$ which occur in the intermediate applications of the box-by-box recursion.

\begin{lemma} \label{compressionlemma} (Compressing $2^{j-1}$ terms to $j$ terms)
Let $j\in \{1, \ldots, n\}$ and let
$$\mu = (0, \ldots, 0, \mu_j, \ldots, \mu_n)\in \ZZ^n_{\ge 0}
\qquad\hbox{with $\mu_1=0, \ldots, \mu_{j-1} = 0$ and $\mu_j \ne 0$.}
$$
and let
$\gamma = \pi\nu = (\mu_j, 0, \ldots, 0, \mu_{j+1}, \ldots, \mu_n).$
Let $\tau^\vee_1, \ldots, \tau^\vee_{n-1}$ be the 
intertwiners of \eqref{earlyintertwinerdefn} acting on $\CC[x_1^{\pm1}, \ldots, x_n^{\pm1}]$
by \eqref{DAHAonCX} and \eqref{GLnCDopsdefn}.
\item[(a)] 
$\displaystyle{
\tau^\vee_{j-1}\cdots \tau^\vee_1E_\gamma
= T_{j-1}\cdots T_1E_\gamma
+ \frac{1-t}{1-q^{\mu_j}t^{v_\mu(j)-(j-1)} }
\sum_{a=1}^{j-1} T_{a-1} \cdots T_1 t^{-\frac12(j-a)} E_\gamma.
}$
\item[(b)] Let $i\in \{1, \ldots, j-1\}$.  Then
\begin{align*}
\tau^\vee_{j-1}\cdots \tau^\vee_1E_\gamma
&= T^{-1}_{j-1}T^{-1}_{j-2} \cdots T^{-1}_i T_{i-1} \cdots T_1E_\gamma \\
&\qquad\quad + \frac{1-t}{1-q^{\mu_j}t^{v_\mu(j)-(j-1)} } q^{\mu_j}t^{v_\mu(j)-(j-1)}
\sum_{a =i }^{j-1} T_{a-1} \cdots T_1 t^{-\frac12(j-a)} E_\gamma \\
&\qquad\qquad\qquad + \frac{1-t}{1-q^{\mu_j}t^{v_\mu(j)-(j-1)} }
\sum_{a=1}^{i-1} T_{a-1} \cdots T_1 t^{-\frac12(j-a)} E_\gamma.
\end{align*}
\end{lemma}
\begin{proof}
Let $\ev^\rho_\gamma\colon \CC(Y)\to \CC(q,t)$ be the homomorphism given by
$$\ev^\rho_\gamma(Y_i) = q^{-\gamma_i}t^{-(v_\gamma(i)-1)+\frac12(n-1)}
\qquad\hbox{so that}\qquad
fE_\gamma = \ev^\rho_\gamma(f) E_\mu,\quad\hbox{for $f\in \CC(Y)$.}
$$
(we shall only apply this to rational expressions in $Y_1, \ldots, Y_n$ where the denominator
does not evaluate to $0$.)
For $i\in \{1, \ldots, n-1\}$ set $\alpha^\vee_i = \varepsilon^\vee_i - \varepsilon^\vee_{i+1}$  and let
$$\beta^\vee_1 = \alpha^\vee_1=\varepsilon^\vee_1-\varepsilon^\vee_2, \qquad
\beta^\vee_2 = s_1\alpha^\vee_2=\varepsilon^\vee_1-\varepsilon^\vee_3, \qquad \ldots \qquad
\beta^\vee_{j-1} = s_1\cdots s_{j-2}\alpha^\vee_{j-1}=\varepsilon^\vee_1-\varepsilon^\vee_j.
$$
For $i\in \{1, \ldots, j-1\}$ let
$$
Y^{-\beta^\vee_i} = Y_1^{-1}Y_{i+1}, \qquad
F_{-\beta^\vee_i} = \frac{t^{-\frac12}(1-t)}{1-Y^{-\beta^\vee_i}},
\qquad
C_{-\beta^\vee_i} = T_i +f_{-\beta^\vee_i}
\qquad\hbox{and}\qquad
$$
$$
d_{-\beta^\vee_i} = \ev^\rho_\gamma(Y^{-\beta^\vee_i}),
\qquad
f_{-\beta^\vee_i} = \ev^\rho_\gamma(F_{-\beta^\vee_i}), \qquad
c_{-\beta^\vee_i} = \ev^\rho_\gamma(t^{\frac12} + F_{-\beta^\vee_i})
$$
If $i\in \{1, \ldots, j-1\}$ then
\begin{align}
&c_{-\beta^\vee_i} = \ev^\rho_\gamma\Big(
\frac{ t^{-\frac12}(1-tY^{-\beta_i^\vee})}{1-Y^{-\beta^\vee_i}}\Big)
=\ev^\rho_\gamma(t^{\frac12} + F_{-\beta^\vee_i}) = t^{\frac12}+f_{-\beta^\vee_i}, \\
&d_{-\beta^\vee_i} = q^{\gamma_1-\gamma_{i+1}}t^{v_\gamma(1)-v_\gamma(i+1)}
= q^{\mu_j}t^{v_{\gamma}(1)-i} = td_{-\beta^\vee_{i+1}},
\nonumber \\
&c_{\beta_{j-1}} \cdots c_{\beta_{a+1}} f_{\beta_{a}} 
=\Big( \frac{t^{-\frac12}\cancel{(1-td_{-\beta^\vee_{j-1}})} }{1-d_{-\beta^\vee_{j-1}}}\Big)
\cdots \Big( \frac{t^{-\frac12}\cancel{(1-td_{-\beta^\vee_{a+1}})}}
{\cancel{1-d_{-\beta^\vee_{a+1}}}}\Big)
\Big( \frac{t^{-\frac12}(1-t)}{\cancel{1-d_{-\beta^\vee_{a}} }}\Big) 
\nonumber \\
&\qquad\qquad\qquad = t^{-\frac12(j-1-a)}\Big( \frac{t^{-\frac12}(1-t)}{1-d_{-\beta^\vee_{j-1}} }\Big) 
= t^{-\frac12(j-1-a)}f_{\beta^\vee_{j-1}}
= t^{-\frac12(j-a)}t^{\frac12} f_{\beta^\vee_{j-1}},
\label{Csmash}
\\
&t-1+t^{\frac12}f_{-\beta^\vee_i} = d_{-\beta^\vee_i} t^{\frac12} f_{-\beta^\vee_i},
\label{savings}
\end{align}
where the last equality follows from
$$(t^{\frac12}-t^{-\frac12})+F_{-\beta^\vee_i} = 
(t^{\frac12}-t^{-\frac12})+ \frac{t^{-\frac12}(1-t)}{1-Y^{-\beta^\vee_i}}
=\frac{t^{-\frac12}(1-t)Y^{-\beta^\vee_i}}{1-Y^{-\beta^\vee_i}}
= Y^{-\beta^\vee_i} F_{-\beta^\vee_i}.$$

Since 
\begin{align*}
\tau^\vee_i \tau^\vee_{i-1}\cdots \tau^\vee_1E_\gamma
&=(T_i+F_{-\alpha_i})\tau^\vee_{i-1}\cdots \tau^\vee_1E_\gamma
=(T_i+\ev^\rho_\gamma(F_{-s_1\cdots s_{i-1}\alpha^\vee_i})
\tau^\vee_{i-1}\cdots \tau^\vee_1E_\gamma \\
&=(T_i+f_{-\beta^\vee_i})\tau^\vee_{i-1}\cdots \tau^\vee_1E_\gamma
=C_{-\beta^\vee_i} \tau^\vee_{i-1}\cdots \tau^\vee_1E_\gamma
\end{align*}
then
$\tau^\vee_{j-1}\cdots \tau^\vee_1 E_\gamma
=C_{-\beta^\vee_{j-1}}\cdots C_{-\beta^\vee_1} E_\gamma.$
$$\hbox{For $i\in \{2, \ldots, j-1\}$,}\qquad
C_{-\beta^\vee_i}E_\gamma = (T_i+f_{-\beta^\vee_i})E_\gamma
=(t^{\frac12}+f_{-\beta^\vee_i})E_\gamma = c_{-\beta^\vee_i}E_\gamma.
$$
Thus,
\begin{align}
\tau^\vee_{j-1}\cdots \tau^\vee_1E_\gamma
&=C_{-\beta^\vee_{j-1}}\cdots C_{-\beta^\vee_1} E_\gamma
= C_{-\beta^\vee_{j-1}}\cdots C_{-\beta^\vee_2} (T_1+ f_{-\beta^\vee_1})E_\gamma 
\nonumber \\
&= C_{-\beta^\vee_{j-1}}\cdots C_{-\beta^\vee_2} T_1E_\gamma
+ f_{-\beta^\vee_1}C_{-\beta^\vee_{j-1}}\cdots C_{-\beta^\vee_2}  E_\gamma 
\nonumber \\
&= C_{-\beta^\vee_{j-1}}\cdots C_{-\beta^\vee_2}  T_1E_\gamma
+ c_{-\beta^\vee_{j-1}}\cdots c_{-\beta^\vee_2} f_{-\beta^\vee_1}E_\gamma 
\label{compressionlineA}
\\
&= C_{-\beta^\vee_{j-1}}\cdots C_{-\beta^\vee_3}  (T_2+ f_{-\beta^\vee_2}) T_1E_\gamma
+ c_{-\beta^\vee_{j-1}}\cdots c_{-\beta^\vee_2} f_{-\beta^\vee_1}E_\gamma 
\nonumber \\
&= C_{-\beta^\vee_{j-1}}\cdots C_{-\beta^\vee_3}  T_2 T_1 E_\gamma
+ c_{-\beta^\vee_{j-1}}\cdots c_{-\beta^\vee_3} f_{-\beta^\vee_2} T_1 E_\gamma
+ c_{-\beta^\vee_{j-1}}\cdots c_{-\beta^\vee_2} f_{-\beta^\vee_1}E_\gamma,
\nonumber 
\end{align}
and continuing this process and using \eqref{Csmash} gives (a).

Let $R_i$ be the right hand side of the expression in statement of (b), so that the identity
in (a) can be considered as $\tau^\vee_{j-1}\cdots \tau^\vee_1E_\gamma = R_j$.
Then, canceling the common terms in $R_{i+1}$ and $R_i$ gives 
\begin{align*}
R_{i+1}-R_i 
&= 
T^{-1}_{j-1} \cdots T^{-1}_{i+1} T_i \cdots T_1E_\gamma
+T_{i-1} \cdots T_1 t^{-\frac12(j-i)} t^{\frac12} f_{-\beta^\vee_{j-1}}E_\gamma \\
&\qquad\quad -T^{-1}_{j-1}T^{-1}_{j-2} \cdots T^{-1}_i T_{i-1} \cdots T_1E_\gamma 
- t^{\frac12} f_{-\beta^\vee_{j-1}} d_{-\beta^\vee_j}
t^{-\frac12(j-i)} T_{i-1} \cdots T_1   E_\gamma \\
&=T^{-1}_{j-1} \cdots T^{-1}_{i+1} (T_i^{-1}+t^{-\frac12}(t - 1))T_{i-1} \cdots T_1E_\gamma 
+T_{i-1} \cdots T_1 t^{-\frac12(j-i)} t^{\frac12} f_{-\beta^\vee_{j-1}}E_\gamma \\
&\qquad -T^{-1}_{j-1}T^{-1}_{j-2} \cdots T^{-1}_i T_{i-1} \cdots T_1E_\gamma 
- t^{\frac12} f_{-\beta^\vee_{j-1}} d_{-\beta^\vee_j}
t^{-\frac12(j-i)} T_{i-1} \cdots T_1   E_\gamma \\
&= t^{-\frac12}(t - 1)T_{i-1} \cdots T_1   T^{-1}_{j-1} \cdots T^{-1}_{i+1}E_\gamma
+ t^{-\frac12(j-i)} t^{\frac12} f_{-\beta^\vee_{j-1}}
(1 -  d_{-\beta^\vee_j}) T_{i-1} \cdots T_1   E_\gamma \\
&= t^{-\frac12(j-i)}( (t - 1) + t^{\frac12} f_{-\beta^\vee_{j-1}}
-  t^{\frac12} f_{-\beta^\vee_{j-1}}d_{-\beta^\vee_j}) T_{i-1} \cdots T_1   E_\gamma = 0,
\end{align*}
by \eqref{savings}. Thus $\tau^\vee_{j-1}\cdots \tau^\vee_1E_\gamma =R_j = R_{j-1}=\ldots =R_1$ and this establishes (b).
\end{proof}

Using Propostion \ref{stepbystep}(a) and adjusting for the normalization in the
definition of $E^z_\nu$ in 
\eqref{Ezmudefn} produces the following box by box recursion for the relative
Macdonald polynomials $E^z_\mu$.

\begin{prop}  \label{boxbybox}
Let $z\in S_n$.
Let 
$$\mu = (0, \ldots, 0, \mu_j, \ldots, \mu_n)\in \ZZ^n_{\ge 0}
\qquad\hbox{with $\mu_1=0, \ldots, \mu_{j-1} = 0$ and $\mu_j \ne 0$.}
$$
and let $\nu = (0, \ldots, 0, \mu_{j+1}, \ldots, \mu_n, \mu_j-1)$.
For $m\in \{0, \ldots, n\}$ let $c^{-1}_m = s_{m-1}\cdots s_2s_1$ (which is an $m$-cycle in $S_n$).
Let $y = (y(1), \ldots, y(n))$ be the permutation which has
$y(k) = z(k)$ for $k\in \{j, \ldots, n\}$
$$
\hbox{and}\quad
\{y(1), \ldots, y(j-1)\} = \{z(1), \ldots, z(j-1)\}\quad\hbox{and}\quad
y(1)<\cdots<y(j-1).
$$
Then
\begin{align*}
E^z_\mu
&=  x_{y(j)} E^{yc^{-1}_jc_n}_\nu
+ \frac{(1-t)}{1-q^{\mu_j}t^{v_\mu(j)-(j-1) }}
\sum_{a=1}^{j-1} 
q^{\mathrm{maj}^y_\mu(a)}t^{\mathrm{covid}^y_\mu(a)} x_{y(a) }  E^{yc^{-1}_a c_n}_\nu,
\end{align*}
where, for $a\in \{1, \ldots, j-1\}$,
$$
\mathrm{maj}^y_\mu(a) = \begin{cases}
0, &\hbox{if $y(j)>y(a)$,}  \\
\mu_j, &\hbox{if $y(j)<y(a)$,} 
\end{cases}
\qquad\qquad\hbox{and}
$$
\begin{align*}
\mathrm{covid}^y_\mu(a) 
&= \begin{cases}
\hbox{$\frac12$}(\ell(yc^{-1}_ac_jv^{-1}_\mu) - \ell(yv^{-1}_\mu)-\ell(c^{-1}_ac_j),  
&\hbox{if $y(j)>y(a)$,} \\
v_\mu(j)-(j-1) +\hbox{$\frac12$}(\ell(yc^{-1}_ac_jv^{-1}_\mu) - \ell(yv^{-1}_\mu)-\ell(c^{-1}_ac_j)),
&\hbox{if $y(j)< y(a)$.} 
\end{cases}
\end{align*}
\end{prop}
\begin{proof}
Write $z=y\sigma$ with $\sigma\in S_{j-1}$ and $y$ minimal length in the coset $zS_{j-1}$.
Then
$y(j) = z(j)$ and, by the last identity in \eqref{Tigivest},
$$T_z E_\mu = T_yT_\sigma E_\mu = T_yt^{\frac12\ell(\sigma)}E_\mu,
\quad\hbox{so that}\quad
E^z_\mu = E^y_\mu.
$$
To control the spacing let $c_a = s_1\cdots s_{a-1}$ (which is an $a$-cycle in $S_n$) and let
$$d_{-\beta^\vee_{j-1}} = q^{\mu_j}t^{v_\gamma(1)-(j-1)}
= q^{\mu_j}t^{v_\mu(j)-(j-1)}
\quad\hbox{and}\qquad
t^{\frac12}f_{-\beta^\vee_{j-1}} = \frac{1-t}{1- q^{\mu_j}t^{v_\mu(j)-(j-1)}}.
$$
Let $\gamma = \pi\nu = (\mu_j, 0, \ldots, 0, \mu_{j+1}, \ldots, \mu_n)$ as in Lemma
\ref{compressionlemma} and note that $v_\gamma = v_\mu s_{j-1}\cdots s_1 = v_\mu c_j^{-1}$.
\hfil\break
If $y(j)>y(j-1)$ then $T_{yc_j^{-1}} = T_y T_{j-1}\cdots T_1$ and
using  Lemma \ref{compressionlemma}(a) gives
\begin{align*}
T_y \tau^\vee_{j-1} &\cdots \tau^\vee_1E_\gamma
= T_{yc^{-1}_j}E_\gamma 
+ t^{\frac12}f_{-\beta^\vee_{j-1}} \sum_{a=1}^{j-1} t^{-\frac12(j-a)} T_{yc^{-1}_a}E_\gamma,
\end{align*}
If $y(j)<y(j-1)$ then $T_{yc_j^{-1}} = T_y T^{-1}_{j-1}\cdots T^{-1}_i T_{i-1} \cdots T_1$ 
with $i=\min\{r\in \{1, \ldots, j-1\}\ |\ y(r)>y(j)\}$ gives
and using  Lemma \ref{compressionlemma}(b) gives
\begin{align*}
T_y \tau^\vee_{j-1} \cdots \tau^\vee_1E_\gamma
&= T_{yc^{-1}_j}E_\gamma
+ t^{\frac12}f_{-\beta^\vee_{j-1}} 
\sum_{a=i}^{j-1} t^{-\frac12(j-a)}d_{-\beta^\vee_{j-1}}  T_{yc^{-1}_a}E_\gamma \\
&\qquad\qquad + t^{\frac12}f_{-\beta^\vee_{j-1}}
\sum_{a=1}^{i-1} t^{-\frac12(j-a)}  T_{yc^{-1}_a}E_\gamma.
\end{align*}
For $a\in \{1, \ldots, j\}$, let
\begin{align*}
\mathrm{norm}^y_\mu(a)
&= 
(\ell(yc^{-1}_a c_j v^{-1}_\mu)-\ell(c_j v^{-1}_\mu)) - (\ell(yv^{-1}_\mu)-\ell(v^{-1}_\mu)-(j-1))  \\
&= 
(\ell(yc^{-1}_ac_jv^{-1}_\mu)-(\ell(v^{-1}_\mu)+j-1) - \ell(yv^{-1}_\mu) + (\ell(v^{-1}_\mu)+(j-1)) \\
&= 
\ell(yc^{-1}_ac_jv^{-1}_\mu) - \ell(yv^{-1}_\mu).
\end{align*}
With this notation, the identities
\begin{align*}
E^y_\mu 
&= t^{-\frac12(\ell(yv^{-1}_\mu)-\ell(v^{-1}_\mu)) } T_y E_\mu 
 = t^{-\frac12(\ell(yv^{-1}_\mu)-\ell(v^{-1}_\mu))}  
T_y (t^{\frac12}\tau^\vee_{j-1})\cdots (t^{\frac12}\tau^\vee_1)E_\gamma \\
&= t^{-\frac12(\ell(yv^{-1}_\mu)-\ell(v^{-1}_\mu)-(j-1))} 
T_y \tau^\vee_{j-1}\cdots \tau^\vee_1E_\gamma, \qquad\hbox{and} \\
E^{yc^{-1}_a}_\gamma 
&= t^{-\frac12(\ell(yc^{-1}_av^{-1}_\gamma)-\ell(v^{-1}_\gamma))} T_{yc^{-1}_a}E_\gamma
= t^{-\frac12(\ell(yc^{-1}_av^{-1}_\gamma)-\ell(v^{-1}_\gamma))} T_yT_{a-1}\cdots T_1 E_\gamma \\
&
= t^{-\frac12(\ell(yc^{-1}_a c_j v^{-1}_\mu)-\ell(c_j v^{-1}_\mu))} T_yT_{a-1}\cdots T_1 E_\gamma,
\qquad\hbox{for $a\in \{1, \ldots, j\}$,}
\end{align*}
then give
\begin{align*}
E^y_\mu 
&= t^{\frac12\mathrm{norm}^y_\mu(j)}
E^{yc^{-1}_j}_\gamma 
+ t^{\frac12}f_{-\beta^\vee_{j-1}} \sum_{a=1}^{j-1} t^{-\frac12(j-a)} 
t^{\frac12\mathrm{norm}^y_\mu(a)} E^{yc^{-1}_a}_\gamma, \qquad\hbox{if $y(j)>y(j-1)$, and}  \\
E^y_\mu 
&=  t^{\frac12\mathrm{norm}^y_\mu(a)} E^{yc^{-1}_j}_\gamma 
+ t^{\frac12} f_{-\beta^\vee_{j-1}} \sum_{a=i}^{j-1} 
d_{-\beta^\vee_{j-1}}  
t^{-\frac12(j-a)}  
t^{\frac12\mathrm{norm}^y_\mu(a)} E^{yc^{-1}_a}_\gamma  \\
&\qquad\qquad\qquad + t^{\frac12} f_{-\beta^\vee_{j-1}}  \sum_{a=1}^{i-1} t^{-\frac12(j-a)}
t^{\frac12\mathrm{norm}^y_\mu(a)} E^{yc^{-1}_a}_\gamma, \qquad\hbox{if $y(j)<y(j-1)$ and}
\end{align*}
$i=\min\{r\in \{1, \ldots, j-1\}\ |\ y(r)>y(j)\}$.

If $a=j$ then $\mathrm{norm}^y_\mu(j)=0$.  Since
$\ell(c^{-1}_ac_j) = \ell(s_{a-1}\cdots s_1 s_1\cdots s_{j-1}) = \ell(s_a\cdots s_{j-1}) = (j-a)$ 
then
$$\mathrm{norm}^y_\mu(a) - (j-a) = 
\ell(yc^{-1}_ac_jv^{-1}_\mu) - \ell(yv^{-1}_\mu) - \ell(c^{-1}_a c_j).$$
Applying Proposition \ref{stepbystep}(a) to the right hand side of the expressions that
have been obtained for $E^y_\mu$ and substtuting
$d_{-\beta^\vee_{j-1}} = q^{\mu_j}t^{v_\gamma(1)-(j-1)}
= q^{\mu_j}t^{v_\mu(j)-(j-1)}
= q^{\mu_j}t^{\frac12\cdot 2(v_\mu(j)-(j-1))}$ gives
\begin{align*}
E^z_\mu = E^y_\mu
&=  x_{y(j)} E^{yc^{-1}_j c_n}_\nu
+ \frac{(1-t)}{1-q^{\mu_j}t^{v_\mu(j)-(j-1) }}
\sum_{a=1}^{j-1} 
q^{\mathrm{maj}^y_\mu(a)}t^{\mathrm{covid}^y_\mu(a)} x_{y(a) }  E^{yc^{-1}_a c_n}_\nu,
\end{align*}
where, if $i= \min\{ r\in \{1, \ldots, j\}\ |\ y(r)\ge y(j)\}$ then
\begin{align*}
\mathrm{covid}^y_\mu(a) 
&= \begin{cases}
\hbox{$\frac12$}\mathrm{norm}^y_\mu(a)-\hbox{$\frac12$}(j-a), &\hbox{if $a<i$,} \\
v_\mu(j)-(j-1)+\hbox{$\frac12$}\mathrm{norm}^y_\mu(a)-\hbox{$\frac12$}(j-a), &\hbox{if $a\ge i$,} 
\end{cases}
\\
&= \begin{cases}
\hbox{$\frac12$}(\ell(yc^{-1}_ac_jv^{-1}_\mu) - \ell(yv^{-1}_\mu)-\ell(c^{-1}_ac_j)),  
&\hbox{if $y(j) > y(a)$,} \\
v_\mu(j)-(j-1)+\hbox{$\frac12$}(\ell(yc^{-1}_ac_jv^{-1}_\mu) - \ell(yv^{-1}_\mu)-\ell(c^{-1}_ac_j)),
&\hbox{if $y(j) < y(a)$.} 
\end{cases}
\end{align*}
\end{proof}

\begin{remark} \textbf{Relating $\mathrm{covid}^y_\mu(a)$ to coinversion triples.} \label{covidremark}
To give an alternate point of view on the statistic
\begin{align*}
\mathrm{covid}^y_\mu(a) 
&= \begin{cases}
\hbox{$\frac12$}(\ell(yc^{-1}_ac_jv^{-1}_\mu) - \ell(yv^{-1}_\mu)-\ell(c^{-1}_ac_j),  
&\hbox{if $y(j)>y(a)$,} \\
v_\mu(j)-(j-1) +\hbox{$\frac12$}(\ell(yc^{-1}_ac_jv^{-1}_\mu) - \ell(yv^{-1}_\mu)-\ell(c^{-1}_ac_j)),
&\hbox{if $y(j)< y(a)$,} 
\end{cases}
\end{align*}
which fell out of the computation in the proof of Proposition \ref{boxbybox},
let us analyze how the inversions of $yv^{-1}_\mu$ change when the factor $c^{-1}_ac_j$ is
inserted to form $yc^{-1}_ac_jv^{-1}_\mu$.  To do this note that
$$yc^{-1}_ac_jv^{-1}_\mu = yv^{-1}_\mu (s_{v_\mu(j)-1}s_{v_\mu(j)-2}\cdots s_j)
(s_a\cdots s_{j-2})s_{j-1}(s_j\cdots s_{v_\mu(j)-1}).$$
and analyze the effect of each of the factors on the right hand side.
\begin{enumerate}
\item[(a)] 
Since $y(a)<y(a+1)<\cdots < y(j-1)$ then $(s_a\cdots s_{j-2})$ creates $(j-1-a)$ inversions in 
$yc^{-1}_ac_jv^{-1}_\mu$ which do not occur in $xv^{-1}_\mu$.
\item[(b)] 
The factor $s_{j-1}$ creates an inversion if $y(j)>y(a)$ and removes an inversion if
$y(j)<y(a)$.
\item[(c)] The factor $(s_j\cdots s_{v_\mu(j)-1})$
\item[] \ \qquad undoes inversions $yv^{-1}_\mu(k) < yv^{-1}_\mu(a)$ for $k\in \{j, \ldots, v_\mu(j)-1\}$,
\item[] \ \qquad adds inversions $yv^{-1}_\mu(k) > yv^{-1}_\mu(a)$ for  $k\in \{j, \ldots, v_\mu(j)-1\}$,
\item[(d)] The factor $(s_{v_\mu(j)-1}s_{v_\mu(j)-2}\cdots s_j)$
\item[] \ \qquad undoes inversions $yv^{-1}_\mu(k) > yv^{-1}_\mu(v_\mu(j))$ for $k\in \{j, \ldots, v_\mu(j)-1\}$,
\item[] \ \qquad adds inversions $yv^{-1}_\mu(k) < yv^{-1}_\mu(v_\mu(j))$ for  $k\in \{j, \ldots, v_\mu(j)-1\}$,
\end{enumerate}
Thus, if $yv^{-1}_\mu(v_\mu(j)) > yv^{-1}_\mu(a)$ (so that $y(j)>y(a)$) then 
\begin{align*}
\ell(yc^{-1}_ac_jv^{-1}_\mu) &- \ell(yv^{-1}_\mu) - \ell(c^{-1}_ac_j)
=\ell(yc^{-1}_ac_jv^{-1}_\mu) - \ell(yv^{-1}_\mu) - (j-a) \\
&= (j-1-a)+1 \\
&\qquad
- (\#\{k\in \{j, \ldots, v_\mu(j)-1\}\ |\ yv^{-1}_\mu(k) < yv^{-1}_\mu(a) \}) \\
&\qquad 
+ (\#\{k\in \{j, \ldots, v_\mu(j)-1\}\ |\ yv^{-1}_\mu(k)>yv^{-1}_\mu(a)  \}) \\
&\qquad
- (\#\{k\in \{j, \ldots, v_\mu(j)-1\}\ |\ yv^{-1}_\mu(v_\mu(j)) < yv^{-1}_\mu(k)\}) \\
&\qquad
+ (\#\{k\in \{j, \ldots, v_\mu(j)-1\}\ |\ yv^{-1}_\mu(v_\mu(j)) > yv^{-1}_\mu(k) \}) \\
&\qquad - (j-a) \\
&= 
-(\#\{k\in \{j, \ldots, v_\mu(j)-1\}\ |\ yv^{-1}_\mu(k) < yv^{-1}_\mu(a)< yv^{-1}_\mu(v_\mu(j))\}) \\
&\qquad 
+(\#\{k\in \{j, \ldots, v_\mu(j)-1\}\ |\ yv^{-1}_\mu(a) < yv^{-1}_\mu(k) < yv^{-1}_\mu(v_\mu(j))\}) \\
&\qquad
+ (\#\{k\in \{j, \ldots, v_\mu(j)-1\}\ |\ yv^{-1}_\mu(a)< yv^{-1}_\mu(v_\mu(j)) < yv^{-1}_\mu(k)\}) \\
&\qquad
- (\#\{k\in \{j, \ldots, v_\mu(j)-1\}\ |\ yv^{-1}_\mu(a)< yv^{-1}_\mu(v_\mu(j)) < yv^{-1}_\mu(k)\}) \\
&\qquad 
+ (\#\{k\in \{j, \ldots, v_\mu(j)-1\}\ |\ yv^{-1}_\mu(a) < yv^{-1}_\mu(k) < yv^{-1}_\mu(v_\mu(j))\}) \\
&\qquad 
+ (\#\{k\in \{j, \ldots, v_\mu(j)-1\}\ |\ yv^{-1}_\mu(k) < yv^{-1}_\mu(a)< yv^{-1}_\mu(v_\mu(j))\}) \\
&= 
2\cdot (\#\{k\in \{j, \ldots, v_\mu(j)-1\}\ |\ yv^{-1}_\mu(a)<yv^{-1}_\mu(k) < yv^{-1}_\mu(v_\mu(j))\}).
\end{align*}
Then, if $yv^{-1}_\mu(v_\mu(j)) < yv^{-1}_\mu(a)$ (so that $y(j)<y(a)$) then
\begin{align*}
\ell(yc^{-1}_ac_jv^{-1}_\mu) &- \ell(yv^{-1}_\mu) - \ell(c^{-1}_ac_j)
=\ell(yc^{-1}_ac_jv^{-1}_\mu) - \ell(yv^{-1}_\mu) - (j-a) \\
&= (j-1-a)-1 \\
&\qquad
- (\#\{k\in \{j, \ldots, v_\mu(j)-1\}\ |\ yv^{-1}_\mu(k) < yv^{-1}_\mu(a) \}) \\
&\qquad 
+ (\#\{k\in \{j, \ldots, v_\mu(j)-1\}\ |\ yv^{-1}_\mu(a)<yv^{-1}_\mu(k)  \}) \\
&\qquad
- (\#\{k\in \{j, \ldots, v_\mu(j)-1\}\ |\ yv^{-1}_\mu(v_\mu(j)) < yv^{-1}_\mu(k)\}) \\
&\qquad
+ (\#\{k\in \{j, \ldots, v_\mu(j)-1\}\ |\ yv^{-1}_\mu(k) < yv^{-1}_\mu(v_\mu(j)) \}) \\
&\qquad - (j-a) \\
&= -1
-(\#\{k\in \{j, \ldots, v_\mu(j)-1\}\ |\ yv^{-1}_\mu(k) < yv^{-1}_\mu(v_\mu(j)) < yv^{-1}_\mu(a) \}) \\
&\qquad 
- (\#\{k\in \{j, \ldots, v_\mu(j)-1\}\ |\ yv^{-1}_\mu(v_\mu(j)) < yv^{-1}_\mu(k) < yv^{-1}_\mu(a) \}) \\
&\qquad
+ (\#\{k\in \{j, \ldots, v_\mu(j)-1\}\ |\ yv^{-1}_\mu(v_\mu(j)) < yv^{-1}_\mu(a)< yv^{-1}_\mu(k)\}) \\
&\qquad
- (\#\{k\in \{j, \ldots, v_\mu(j)-1\}\ |\ yv^{-1}_\mu(v_\mu(j)) < yv^{-1}_\mu(k) < yv^{-1}_\mu(a) \}) \\
&\qquad 
- (\#\{k\in \{j, \ldots, v_\mu(j)-1\}\ |\ yv^{-1}_\mu(v_\mu(j)) < yv^{-1}_\mu(a) < yv^{-1}_\mu(k) \}) \\
&\qquad 
+ (\#\{k\in \{j, \ldots, v_\mu(j)-1\}\ |\ yv^{-1}_\mu(k) < yv^{-1}_\mu(v_\mu(j)) < yv^{-1}_\mu(a) \}) \\
&= -2
- 2\cdot (\#\{k\in \{j, \ldots, v_\mu(j)-1\}\ |\ yv^{-1}_\mu(a)<yv^{-1}_\mu(k) < yv^{-1}_\mu(v_\mu(j))\}).
\end{align*}
If $y(j) > y(a)$ then
\begin{align*}
\hbox{$\frac12$}(\ell(yc^{-1}_ac_jv^{-1}_\mu) &- \ell(yv^{-1}_\mu) - \ell(c^{-1}_ac_j)) \\
&= \#\{k\in \{j, \ldots, v_\mu(j)-1\}\ |\ yv^{-1}_\mu(a)<yv^{-1}_\mu(k) < yv^{-1}_\mu(v_\mu(j))\} \\
&= 
\#\{k\in \{j, \ldots, v_\mu(j)-1\}\ |\ y(a) < yv^{-1}_\mu(k) < y(j) \}  \\
&= 
\#\{v_\mu(b)\in \{j, \ldots, v_\mu(j)-1 \}\ |\ y(a) < yv^{-1}_\mu(v_\mu(b)) < y(j) \}  \\
&= 
\#\{ b\in \{ v^{-1}_\mu(j), \ldots, v^{-1}_\mu(v_\mu(j)-1) \} \ |\ y(a) < y(b) < y(j) \}  
\end{align*}
and, 
if $y(j)<y(a)$ then
\begin{align*}
(v_\mu(j)&-(j-1)) + \hbox{$\frac12$} 
\big(\ell(yc^{-1}_ac_jv^{-1}_\mu) - \ell(yv^{-1}_\mu) - \ell(c^{-1}_ac_j)\big) \\
&= 
(v_\mu(j)-j+1) -1
- \#\{k\in \{j, \ldots, v_\mu(j)-1\}\ |\ yv^{-1}_\mu(a)<yv^{-1}_\mu(k) < yv^{-1}_\mu(v_\mu(j))\} \\
&= 
((v_\mu(j)-1)-(j-1))
- \#\{k\in \{j, \ldots, v_\mu(j)-1\}\ |\ yv^{-1}_\mu(a)<yv^{-1}_\mu(k) < yv^{-1}_\mu(v_\mu(j))\} \\
&= 
\#\Big\{k\in \{j, \ldots, v_\mu(j)-1\}\ \Big\vert \ 
\begin{array}{l}
\hbox{$yv^{-1}_\mu(k) < yv^{-1}_\mu(a) < yv^{-1}_\mu(v_\mu(j))$} \\
\hbox{ or $yv^{-1}_\mu(a) < yv^{-1}_\mu(v_\mu(j))  < yv^{-1}_\mu(k)$}
\end{array}
\Big\} \\
&= 
\#\Big\{b\in \{ v^{-1}_\mu(j), \ldots, v^{-1}_\mu(v_\mu(j)-1) \} \ \Big\vert \ 
\begin{array}{l}
\hbox{$y(b) < y(a) < y(j)$} \\
\hbox{ or $y(a) < y(j)  < y(b)$}
\end{array}
\Big\}. 
\end{align*}
These last two expressions are exactly the numbers of coninversion triples that appear
in \cite[Lemma 3.6.3]{HHL06} for the box $(j,1)$ filled with $y(a)$ in a filling of shape 
$\mu$ with basement $(y(1), \ldots, y(n))$.
\end{remark}


\section{Type $GL_n$ DAArt, DAHA and the polynomial representation}\label{GLnDAHAsection}

 The power tools that enable us to
construct and manipulate Macdonald polynomials with ease
are the polynomial generators $X_1, \ldots, X_n$, 
the Cherednik-Dunkl operators $Y_1, \ldots, Y_n$ and the intertwiners
$\tau^\vee_1, \ldots, \tau^\vee_{n-1}, \tau^\vee_\pi$ which all live inside the 
double affine Hecke algebra
$\tilde H_{GL_n}$.  In this section we will build the Macdonald polynomials $E_\mu$ by
first constructing the double affine Artin group $\tilde \cB_{GL_n}$, 
then the elements $X_1, \ldots, X_n$ and  $Y_1, \ldots, Y_n$,
then the DAHA $\tilde H_{GL_n}$ and 
the intertwiners $\tau^\vee_1, \ldots, \tau^\vee_{n-1}, \tau^\vee_\pi$.
Let us begin by defining the DAArt $\tilde \cB_{GL_n}$ and establishing its primary dualities.
The definition is by generators and relations and the dualities are automorphisms of 
$\tilde \cB_{GL_n}$.  The double affine Hecke algebra $\tilde H_{GL_n}$
is constructed as a quotient of the 
group algebra of $\tilde \cB_{GL_n}$ by the Hecke relations $T_i^2 = (t^{\frac12}-t^{-\frac12})T_i +1$. 

Use Coxeter diagram shorthand for relations so that

\begin{equation*}
\begin{tikzpicture}
	\draw (0,0)--(1,0);
	\draw[fill=white] (0,0) circle (2.5pt) node[above=1pt] {\small $a$};
	\draw[fill=white] (1,0) circle (2.5pt) node[above=1pt] {\small $b$}; 
\end{tikzpicture}
\quad  \hbox{indicates $aba=bab$,\quad and\qquad}
\begin{tikzpicture}
	\draw[fill=white] (0,0) circle (2.5pt) node[above=1pt] {\small $a$};
	\draw[fill=white] (1,0) circle (2.5pt) node[above=1pt] {\small $b$}; 
\end{tikzpicture}
\quad\hbox{indicates $ab=ba$,} 
\end{equation*}

\subsection{The type $GL_n$ double affine Artin group (DAArt)}


The element $q$ will be a parameter in the Macdonald polynomials.  In the definition
of the DAArt by generators and relations the elememt $q$ 
appears as a central element of the group, but in Section \ref{GLnpolyderiv} 
the element $q$ will get specialized to be a complex parameter.

The \emph{type $GL_n$ double affine Artin group (DAArt)} $\tilde \cB_{GL_n}$
is generated by $q, g^\vee, g, S_0^\vee, S_0, T_1,\ldots, T_{n-1}$
with the relations 
\begin{equation}
\begin{matrix}
\begin{matrix}\begin{tikzpicture}[every node/.style={inner sep=1}, scale=1.0]
	\foreach \x/\y in {1/1,2/2,4/{n\!-\!2},5/{n\!-\!1}}{
		\node[wV, label=below:{$T_{\y}$}] (\x) at (\x,0) {};}
	\node[wV, label=above:{$S_0$}] (0) at (3,1) {};
	\draw (1) to (2) (4) to (5);
	\draw [bend left=15] (1) to (0) (0) to (5);
	\draw[dashed] (2) to (4);
	\end{tikzpicture}\end{matrix}
%
\quad
&\begin{array}{c}
gS_0g^{-1} = T_1, \\ \ \\
gT_ig^{-1} = T_{i+1},
\qquad
g T_{n-1} g^{-1} = S_0,
\end{array}
\\
\begin{matrix}\begin{tikzpicture}[every node/.style={inner sep=1}, scale=1.0]
	\foreach \x/\y in {1/1,2/2,4/{n\!-\!2},5/{n\!-\!1}}{
		\node[wV, label=below:{$T_{\y}$}] (\x) at (\x,0) {};}
	\node[wV, label=above:{$S^\vee_0$}] (0) at (3,1) {};
	\draw (1) to (2) (4) to (5);
	\draw [bend left=15] (1) to (0) (0) to (5);
	\draw[dashed] (2) to (4);
	\end{tikzpicture}\end{matrix}
\quad
&\begin{array}{c}
g^\vee S_0^\vee (g^\vee)^{-1} = T_1, \\ \ \\
g^\vee T_i (g^\vee)^{-1} = T_{i+1},
\qquad
g^\vee T_{n-1} (g^\vee)^{-1} = S^\vee_0,
\end{array}
\end{matrix}
\label{ggveecyclicDAArtrel}
\end{equation}
$$q\in Z(\widetilde{\cB}_{GL_n}) \qquad\hbox{and}$$
\begin{equation}
T_1 g^\vee g = g g^\vee T_{n-1}^{-1}
\qquad\hbox{and}\qquad
T_{n-1}^{-1}\cdots T_1^{-1}g(g^\vee)^{-1} 
= q(g^\vee)^{-1} gT_{n-1}\cdots T_1.
\label{ggveecommutationDAArtrel}
\end{equation}
for $i\in \{1,\ldots, n-2\}$.

The two visible symmetries in this defnition, switching the Coxeter diagram containing $S_0$ and
the Coxeter diagram containing $S_0^\vee$, and flipping the Coxeter diagrams about the middle,
form two important dualities.  These dualities are expressed as involutive automorphisms
of the DAArt $\tilde \cB_{GL_n}$.

\begin{thm} \label{dualitythm} 
\item[(a)] ($\vee$Duality)  
There is an involution $\iota\colon \tilde\cB_{GL_n}\to \tilde \cB_{GL_n}$ with
$$\iota(q) = q^{-1}, \qquad \iota(T_i)= T_i^{-1}, \qquad \iota(S_0^\vee) = S_0^{-1},\qquad
\iota(g) = g^\vee.
$$
\item[(b)] ( $\vdash\dashv$Duality) There is an involution $\eta\colon \tilde\cB_{GL_n}\to \tilde B_{GL_n}$ with
$$\eta(q) = q, \qquad \eta(T_i) = T_{n-i}, \qquad \eta(g) = g^{-1}, \qquad \eta(g^\vee) = (g^\vee)^{-1}.
$$
\end{thm}
\begin{proof}  (a) Applying $\iota$ to the relations in  \eqref{ggveecyclicDAArtrel}
switches the upper (nonchecked) relations with the lower (checked) relations.
Applying $\iota$ to the relations in \eqref{ggveecommutationDAArtrel}
produces the relations $q^{-1}\in Z(\widetilde{\cB}_{GL_n})$,
$$T_1^{-1}gg^\vee = g^\vee g T_{n-1}
\quad\hbox{and}\quad
T_{n-1}\cdots T_1g^\vee g^{-1} = q^{-1}g^{-1}g^\vee T_{n-1}^{-1}\cdots T_1^{-1},$$
respectively.
Thus the relations  in \eqref{ggveecommutationDAArtrel} are preserved under $\iota$.

\smallskip\noindent
(b) The involution $\eta$ preserves the relations in  \eqref{ggveecyclicDAArtrel}.
Applying $\eta$ to the relations in \eqref{ggveecommutationDAArtrel}
produces the relations $q\in Z(\widetilde{\cB}_{GL_n})$,
$$T_{n-1}(g^\vee)^{-1}g^{-1} = g^{-1}(g^\vee)^{-1}T_1^{-1}
\quad\hbox{and}\quad
T_1^{-1}\cdots T_{n-1}^{-1} g^{-1}g^\vee
= q g^\vee g^{-1}T_1\cdots T_{n-1}
$$
which are equivalent to the original relations in \eqref{ggveecommutationDAArtrel}
by taking inverses.
\end{proof}

\subsection{The elements $X^{\varepsilon_1}, \ldots, X^{\varepsilon_n}$
and $Y^{\varepsilon_1^\vee}, \ldots, Y^{\varepsilon_n^\vee}$}

The elements $X^{\varepsilon_1}, \ldots, X^{\varepsilon_n}$ will be used as the 
generators for a polynomial ring (inside the group algebra of $\tilde \cB_{GL_n}$),
and the Macdonald polynomials are polynomials in these variables.  Inside the DAArt,
these elements form a large commuttative subgroups and, because of duality,
there is \emph{another} large commutative subgroup generated by elements
$Y^{\varepsilon^\vee_1}, \ldots, Y^{\varepsilon^\vee_n}$.
%
In this section we define these elements and give alternate presentations of
$\tilde \cB_n$ in terms of these elements.

Define $Y^{\varepsilon_1^\vee}, \ldots, Y^{\varepsilon_n^\vee}$ 
and $X^{\varepsilon_1}, \ldots, X^{\varepsilon_n}$ in $\tilde \cB_{GL_n}$ 
by
\begin{equation}\label{GLXYdefn}
\begin{array}{lcl}
Y^{\varepsilon_1^\vee} = gT_{n-1}\cdots T_1
&\qquad\hbox{and}\qquad 
&Y^{\varepsilon_{j+1}^\vee} = T_j^{-1}Y^{\varepsilon_j^\vee}T_j^{-1},
\\ 
X^{\varepsilon_1} = g^\vee T_{n-1}^{-1}\cdots T_1^{-1},
&\quad\hbox{and}\quad
&X^{\varepsilon_{j+1}} = T_j X^{\varepsilon_j}T_j,
\end{array}
\end{equation}
for $j\in \{1, \ldots, n-1\}$.  If $\iota\colon \tilde \cB_{GL_n}\to \tilde \cB_{GL_n}$ 
and $\eta\colon \tilde \cB_{GL_n}\to \tilde \cB_{GL_n}$ are the 
involutions in Theorem \ref{dualitythm} then
\begin{equation}
\iota(X^{\varepsilon_i}) = Y^{\varepsilon_i}
\qquad\hbox{and}\qquad
\eta(X^{\varepsilon_i}) = X^{-\varepsilon_{n-i+1}}
\quad\hbox{and}\quad
\eta(Y^{\varepsilon_i}) = Y^{-\varepsilon_{n-i+1}},
\label{dualitiesonXY}
\end{equation}
for $i\in \{1, \ldots, n\}$.

The subgroup generated by $g^\vee, T_1, \ldots, T_{n-1}$ has a pictorial representation given by
\begin{equation*}
{ 
\def\TOP{2}\def\K{6}
g^\vee =
\TikZ{[scale=.5]
		\Pole[\K+.85][0,.5]
		\Over[6,0][\K+1.3,0.5]
		\Under[\K+1.3,0.5][1,2]
		\Pole[\K+.85][.5,2]
		 \foreach \x in {1,...,5} {\draw[thin, style=over] (0,\x) -- (2,\x+1);}
		\Caps[\K+.85][0,\TOP][\K]
}}
\qquad\hbox{and}\qquad
{\def\TOP{2} \def\K{6}
T_i =
\TikZ{[scale=.5]
	\Pole[\K+.85][0,2][\K]
	\Under[4,0][3,2]
	\Over[3,0][4,2]
	 \foreach \x in {1,2,5,\K} {
		 \draw[thin] (0,\x) -- (\TOP,\x);
		 }
	\Caps[\K+.85][0,\TOP][\K]
	\Label[0,\TOP][3][\footnotesize $i+1$]
	\Label[0,\TOP][4][\footnotesize$i$]
}
\qquad \text{for $i=1, \dots, n-1$,}
}
\end{equation*}
\begin{equation}
\hbox{so that}\qquad
{ 
\def\TOP{2}\def\K{6}
X^{\varepsilon_i}  = T_{i-1} \cdots T_1 g^\vee  T_{n-1}^{-1}\cdots T_i^{-1}
=
\TikZ{[scale=.5]
		\Pole[\K+.85][0,1]
		 \foreach \x in {1,2} {\draw[thin] (0,\x) -- (1,\x);}
		\foreach \x in {4,...,\K} {\draw[thin] (0,\x) -- (1,\x);}
		\Over[3,0][\K+1.3,1]
		\Under[\K+1.3,1][3,2]
		\Pole[\K+.85][1,2]
		\foreach \x in {1,2} { \draw[thin, style=over] (1,\x) -- (2,\x); }
		\foreach \x in {4,...,\K} {\draw[thin, style=over] (1,\x) -- (\TOP,\x);}
		\Caps[\K+.85][0,\TOP][\K]
		\Label[0,\TOP][3][{\footnotesize $i$}]
}\quad\hbox{for $i\in \{1, \ldots,n\}$.}
}\end{equation}
Use the notation
$X_i = X^{\varepsilon_i}$,
$$\hbox{and let}\qquad
X^\mu = X_1^{\mu_1}\cdots X_n^{\mu_n} = X^{\mu_1\varepsilon_1+\cdots+\mu_n\varepsilon_n},
\qquad
\hbox{for $\mu = (\mu_1, \ldots, \mu_n)\in \ZZ^n$.}
$$
Similarly use the notation
$Y_i = Y^{\varepsilon^\vee_i}$,
$$\hbox{and let}\qquad
Y^{\lambda^\vee} = Y_1^{\lambda_1}\cdots X_n^{\lambda_n} 
= Y^{\lambda_1\varepsilon^\vee_1+\cdots+\lambda_n\varepsilon^\vee_n},
\qquad
\hbox{for $\lambda^\vee = (\lambda_1, \ldots, \lambda_n)\in \ZZ^n$.}
$$
The pictorial representation provides an easy check of the relations
\begin{align}
&(g^\vee)^n = X^{\varepsilon_1+\cdots+\varepsilon_n}
\quad\hbox{and}\quad
X^{\varepsilon_i}X^{\varepsilon_j} = X^{\varepsilon_j}X^{\varepsilon_i},
\label{Xommutation}
\\
\hbox{and similarly,}\qquad
&g^n = Y^{\varepsilon_1^\vee+\cdots+\varepsilon_n^\vee}
\quad\hbox{and}\quad
Y^{\varepsilon_i^\vee}Y^{\varepsilon_j^\vee}
= Y^{\varepsilon_j^\vee}Y^{\varepsilon_i^\vee},
\label{Ycommutation}
\end{align}
for $i,j\in \{1,\ldots, n\}$.  The pictorial perspective also verifies the relations
\begin{equation}
S_0T_{s_\varphi} = Y^{\varepsilon_1^\vee-\varepsilon_n^\vee}
\quad\hbox{and}\quad
(S_0^\vee)^{-1}T^{-1}_{s_\varphi} = X^{\varepsilon_1-\varepsilon_n},
\quad\hbox{where}\quad
T_{s_\varphi} = T_{n-1}\cdots T_1\cdots T_{n-1}.
\label{S0sphi}
\end{equation}

\begin{thm} \label{BGLaltpres}  
\item[(a)]The group $\tilde \cB_{GL_n}$ is presented by generators $q, g^\vee, S_0^\vee, T_1,\ldots, T_{n-1}$ and
$Y^{\varepsilon_1^\vee}, \ldots, Y^{\varepsilon_n^\vee}$ \hfil\break
and relations
$$
\begin{matrix}\begin{tikzpicture}[every node/.style={inner sep=1}, scale=1.0]
	\foreach \x/\y in {1/1,2/2,4/{n\!-\!2},5/{n\!-\!1}}{
		\node[wV, label=below:{$T_{\y}$}] (\x) at (\x,0) {};}
	\node[wV, label=above:{$S_0^\vee$}] (0) at (3,1) {};
	\draw (1) to (2) (4) to (5);
	\draw [bend left=15] (1) to (0) (0) to (5);
	\draw[dashed] (2) to (4);
	\end{tikzpicture}\end{matrix}
\quad
g^\vee T_i (g^\vee)^{-1} = T_{i+1},
\quad
g^\vee T_{n-1} (g^\vee)^{-1} = S_0^\vee,
$$
\begin{equation}\label{TY0GL}
q \in Z(\tilde \cB_{GL_n}),
\qquad
Y^{\varepsilon_k^\vee}Y^{\varepsilon_j^\vee}
= Y^{\varepsilon_j^\vee}Y^{\varepsilon_k^\vee}
\ \hbox{ for $k,j\in \{1, \ldots, n\}$,}
\end{equation}
\begin{equation}\label{TY1GL}
Y^{\varepsilon_{i+1}^\vee} = T_i^{-1}Y^{\varepsilon_i^\vee}T_i^{-1},
\quad\hbox{and}\quad
T_iY^{\varepsilon_j^\vee} = Y^{\varepsilon_j^\vee}T_i
\ \hbox{for $i\in\{1,\ldots, n-1\}$ and $j\ne i, i+1$,}
\end{equation}
\begin{equation}\label{TY2GL}
g^\vee Y^{\varepsilon_i^\vee} (g^\vee)^{-1} = Y^{\varepsilon_{i+1}^\vee}\  
\hbox{for $i\in \{1,\ldots, n-1\}$,}
\quad\hbox{and}\quad
g^\vee Y^{\varepsilon_n^\vee} (g^\vee)^{-1}= q Y^{\varepsilon_1^\vee},
\end{equation}
\item[(b)]The group $\tilde \cB_{GL_n}$ is presented by generators 
$q, g, T_0, T_1,\ldots, T_{n-1}$ and $X^{\varepsilon_1}, \ldots, X^{\varepsilon_n}$
\hfil\break
and relations
$$
\begin{matrix}\begin{tikzpicture}[every node/.style={inner sep=1}, scale=1.0]
	\foreach \x/\y in {1/1,2/2,4/{n\!-\!2},5/{n\!-\!1}}{
		\node[wV, label=below:{$T_{\y}$}] (\x) at (\x,0) {};}
	\node[wV, label=above:{$S_0$}] (0) at (3,1) {};
	\draw (1) to (2) (4) to (5);
	\draw [bend left=15] (1) to (0) (0) to (5);
	\draw[dashed] (2) to (4);
	\end{tikzpicture}\end{matrix}
\quad
gT_ig^{-1} = T_{i+1},
\quad g T_{n-1}g^{-1} = S_0,
$$
\begin{equation}\label{BGLXcomm}
q \in Z(\tilde \cB_{GL_n}),
\qquad
X^{\varepsilon_k}X^{\varepsilon_j}
= X^{\varepsilon_j}X^{\varepsilon_k}
\ \hbox{for $k,j\in \{1, \ldots, n\}$,}
\end{equation}
\begin{equation}\label{TX2}
X^{\varepsilon_{i+1}} = T_iX^{\varepsilon_i}T_i,
\quad\hbox{and}\quad
T_iX^{\varepsilon_j} = X^{\varepsilon_j}T_i,
\ \hbox{for $i\in \{1,\ldots, n-1\}$ and $j\ne i, i+1$,}
\end{equation}
\begin{equation}
gX^{\varepsilon_i}g^{-1} = X^{\varepsilon_{i+1}},
\ \ \hbox{for $i\in \{1,2,\ldots, n-1\}$\quad and}
\qquad
gX^{\varepsilon_n}g^{-1} = q^{-1}X^{\varepsilon_1}.
\end{equation}
\end{thm}
\begin{proof}  The proof is by showing that the relations in \eqref{TY0GL}, \eqref{TY1GL}
and \eqref{TY2GL} follow from the defining relations of $\tilde \cB_{GL_n}$, and vice versa.

\smallskip\noindent
\eqref{ggveecyclicDAArtrel}\&\eqref{ggveecommutationDAArtrel} 
$\Longrightarrow$ (a):  Using \eqref{GLXYdefn} to 
define $Y^{\varepsilon_i^\vee}$, the pictorial perspective
establishes the relations in 
\eqref{TY0GL} and \eqref{TY1GL}.  The proof of the relations in \eqref{TY2GL} is completed by
\begin{align*}
g^\vee Y^{\varepsilon_1^\vee}
&= g^\vee g T_{n-1}\cdots T_1
= T_1^{-1} g g^\vee T_{n-2}\cdots T_1
= T_1^{-1}g T_{n-1}\cdots T_2g^\vee \\
&= T_1^{-1}g T_{n-1}\cdots T_2T_1T_1^{-1}g^\vee
= T_1^{-1}Y^{\varepsilon_1^\vee}T_1^{-1}g^\vee
=Y^{\varepsilon_2^\vee}g^\vee, \\
\\
g^\vee Y^{\varepsilon_i^\vee} (g^\vee)^{-1}
&= g^\vee T_{i-1}^{-1}\cdots T_1^{-1}Y^{\varepsilon_1^\vee}
T_1^{-1}\cdots T_{i-1}^{-1} (g^\vee)^{-1} \\
&= T_i^{-1}\cdots T_2^{-1}Y^{\varepsilon_2^\vee} T_2^{-1}\cdots T_i^{-1} 
= Y^{\varepsilon_{i+1}^\vee}, 
\qquad\hbox{and} \\
\\
g^\vee Y^{\varepsilon_n^\vee} (g^\vee)^{-1}
&=g^\vee T_{n-1}^{-1}\cdots T_1^{-1} Y^{\varepsilon_1^\vee}
T_1^{-1}\cdots T_{n-1}^{-1} (g^\vee)^{-1}  \\
&= g^\vee T_{n-1}^{-1}\cdots T_1^{-1}
g T_{n-1}\cdots T_1 T_1^{-1}\cdots T_{n-1}^{-1} (g^\vee)^{-1} 
= g^\vee T_{n-1}^{-1}\cdots T_1^{-1} g  (g^\vee)^{-1} \\
& = qg^\vee (g^\vee)^{-1} g  T_{n-1}\cdots T_1 
=qY^{\varepsilon_1^\vee}.
\end{align*}
(a) $\Rightarrow$ \eqref{ggveecyclicDAArtrel}\&\eqref{ggveecommutationDAArtrel}:  
Use \eqref{S0sphi} to define $g$ and $S_0$ in terms of $T_i$ and $Y^{\varepsilon_i^\vee}$s.
If $i\in \{1, \ldots, n-2\}$ then
\begin{align*}
gT_i g^{-1} 
&= Y^{\varepsilon_1^\vee} T_1^{-1} \cdots T_{n-1}^{-1} T_i
T_{n-1}\cdots T_1 Y^{-\varepsilon^\vee_1} 
= Y^{\varepsilon_1^\vee} T_1^{-1} \cdots T_{i+1}^{-1} T_i
T_{i+1}\cdots T_1 Y^{-\varepsilon^\vee_1} \\
&= Y^{\varepsilon_1^\vee} T_1^{-1} \cdots T^{-1}_i T_i T_{i+1} T_i^{-1} T_i 
\cdots T_1 Y^{-\varepsilon^\vee_1} 
= Y^{\varepsilon_1^\vee} T_1^{-1} \cdots T^{-1}_{i-1} T_{i+1} T_{i-1}^{-1} 
\cdots T_1 Y^{-\varepsilon^\vee_1} \\
&= Y^{\varepsilon_1^\vee} T_{i+1} Y^{-\varepsilon^\vee_1} = T_{i+1},
\end{align*}
and
\begin{align*}
gT_{n-1}g^{-1}
&= g T_{n-1}T_{n-2}\cdots T_1T_1^{-1}\cdots T_{n-2}^{-1} g^{-1} 
= Y^{\varepsilon_1^\vee} T_1^{-1}\cdots T_{n-2}^{-1} g^{-1} \\
&= Y^{\varepsilon_1^\vee} g^{-1}T_2^{-1}\cdots T_{n-1}^{-1} 
= Y^{\varepsilon_1^\vee} g^{-1} T_1\cdots T_{n-1} T_{n-1}^{-1}\cdots T_1^{-1}
T_2^{-1}\cdots T_{n-1}^{-1} \\
&= Y^{\varepsilon_1^\vee - \varepsilon_n^\vee} T_{n-1}^{-1}\cdots T_1^{-1}\cdots T_{n-1}^{-1} 
= S_0,
\end{align*}
and
\begin{align*}
gS_0g^{-1}
&= Y^{\varepsilon_1^\vee}T_1^{-1}\cdots T_{n-1}^{-1}Y^{\varepsilon_1^\vee-\varepsilon_n^\vee}
T_{n-1}^{-1}\cdots T_1^{-1}\cdots T_{n-1}^{-1} T_{n-1}\cdots T_1Y^{-\varepsilon_1^\vee} \\
&=Y^{\varepsilon_1^\vee}T_1^{-1}\cdots T_n^{-1}Y^{-\varepsilon_n^\vee}
Y^{\varepsilon_1^\vee}T_{n-1}^{-1}\cdots T_2^{-1}Y^{-\varepsilon_1^\vee} \\
&=Y^{\varepsilon_1^\vee}T_1^{-1}\cdots T_{n-1}^{-1}T_{n-1}\cdots T_1
Y^{-\varepsilon_1^\vee}T_1\cdots T_{n-1}Y^{\varepsilon_1^\vee}Y^{-\varepsilon_1^\vee}
T_{n-1}^{-1}\cdots T_2^{-1} = T_1,
\end{align*}
which establishes the relations in \eqref{ggveecyclicDAArtrel}.

To prove the first relation in \eqref{ggveecommutationDAArtrel}:
\begin{align*}
T_1g^\vee g 
&= T_1 g^\vee Y^{\varepsilon_1^\vee}T_1^{-1}\cdots T_{n-1}^{-1} 
= T_1Y^{\varepsilon_2^\vee} g^\vee T_1^{-1}\cdots T_{n-1}^{-1} \\
&=T_1T_1^{-1} Y^{\varepsilon_1^\vee}T_1^{-1}g^\vee T_1^{-1}\cdots T_{n-1}^{-1} 
=g T_{n-1}\cdots T_2T_1 T_1^{-1}g^\vee T_1^{-1}\cdots T_{n-1}^{-1} \\
&= gT_{n-1}\cdots T_2 g^\vee T_1^{-1}\cdots T_{n-1}^{-1} 
= g g^\vee T_{n-2}\cdots T_1 T_1^{-1}\cdots T_{n-1}^{-1} 
= g g^\vee T_{n-1}^{-1},
\end{align*}
and to prove the second relation in \eqref{ggveecommutationDAArtrel}:
\begin{align*}
T_{n-1}^{-1}\cdots T_1^{-1} &g (g^\vee)^{-1}
= T_{n-1}^{-1}\cdots T_1^{-1} g T_{n-1}\cdots T_1T_1^{-1}\cdots T_{n-1}^{-1} (g^\vee)^{-1}  \\
&= T_{n-1}^{-1}\cdots T_1^{-1} Y^{\varepsilon_1^\vee} T_1^{-1}\cdots T_{n-1}^{-1}(g^\vee)^{-1} 
= Y^{\varepsilon_n^\vee}(g^\vee)^{-1}  \\
&= (g^\vee)^{-1}g^\vee Y^{\varepsilon_n^\vee} (g^\vee)^{-1} 
= (g^\vee)^{-1} q Y^{\varepsilon_1^\vee} 
= q(g^\vee)^{-1} g T_{n-1}\cdots T_1.
\end{align*}
Part (b) follows from part (a) by applying the duality involution $\iota$.
\end{proof}

\subsection{The elements $X^{t_\mu w}$}


Use the notation
$X_i = X^{\varepsilon_i}$,
$$\hbox{and let}\qquad
X^\mu = X_1^{\mu_1}\cdots X_n^{\mu_n} = X^{\mu_1\varepsilon_1+\cdots+\mu_n\varepsilon_n},
\qquad
\hbox{for $\mu = (\mu_1, \ldots, \mu_n)\in \ZZ^n$.}
$$
Using the notation of the affine Weyl group $W= \{ t_\mu w\ |\ \mu\in \ZZ^n, w\in S_n\}$ from 
Section \ref{affWsection},
for $\mu\in \ZZ^n$ and $w\in S_n$ define
\begin{equation}
X^{t_\mu w} = X^\mu T_w,
\qquad\hbox{where $T_w = T_{i_1}\cdots T_{i_\ell}$
if $w= s_{i_1}\cdots s_{i_\ell}$ is a reduced word.}
\label{Xvdefn}
\end{equation}
The following Proposition establishes how these elements are affected by right multiplication
by the generators $T_1, \ldots, T_n$ and $g^\vee$ of $\tilde \cB_{GL_n}$.

\begin{prop}  \label{TtoXconversion}
Let $\mu\in \ZZ^n$ and $w\in S_n$.  Then
\begin{equation}
X^\mu T_{ws_i} = \begin{cases}
X^\mu T_w T_i, &\hbox{if $\ell(ws_i)>\ell(w)$,} \\
X^\mu T_w T_i^{-1}, &\hbox{if $\ell(ws_i)<\ell(w)$,}
\end{cases}
\quad \hbox{and}\quad
X^\mu T_w g^\vee = X^\mu X_{w(1)} T_{ws_1\cdots s_{n-1}}.
\end{equation}
\end{prop}
\begin{proof}
The first equality follows from the fact that if $z\in S_n$ and $\ell(zs_i) > \ell(z)$ 
then $T_zT_i = T_{zs_i}$.  For the second equality: 
Let $k = w(1)$ and write $w = s_{k-1}\cdots s_1z$ with $z$ in the subgroup of $S_n$ 
that is generated by $s_2, \ldots, s_{n-1}$.
Letting $c_n = s_1\cdots s_{n-1}$ and using 
$g^\vee T_i (g^\vee)^{-1} = T_{i+1}$ then 
$(g^\vee)^{-1}T_zg^\vee = T_{c_n^{-1}zc_n}$ and
\begin{align*}
T_w g^\vee &= T_{k-1}\cdots T_1 T_z g^\vee 
= T_{k-1}\cdots T_1 g^\vee ((g^\vee)^{-1}T_z g^\vee) \\
&= T_{k-1}\cdots T_1 g^\vee T_{c_n^{-1}zc_n}
= (T_{k-1}\cdots T_1 g^\vee T_{n-1}^{-1}\cdots T_k^{-1}) T_k\cdots T_{n-1} T_{c_n^{-1}zc_n} \\
&= X_k T_k\cdots T_{n-1} T_{s_k\cdots s_{n-1} c_n^{-1}zc_n} 
= X_kT_{s_{k-1}\cdots s_1 zc_n}
= X_kT_{wc_n}.
\end{align*}
\end{proof}


\subsection{The type $GL_n$ double affine Hecke algebra (DAHA)}

The \emph{type $GL_n$ double affine Hecke algebra} $\tilde H_{GL_n}$ is the quotient of the
group algebra of $\tilde \cB_{GL_n}$ by the relations
\begin{equation}
(T_i-t^{\frac12})(T_i+t^{-\frac12}) = 0,
\qquad\hbox{for $i\in \{1, \ldots, n-1\}$.}
\label{Heckerelation}
\end{equation}
The involutions 
$\iota\colon \tilde \cB_{GL_n}\to \tilde \cB_{GL_n}$ 
and $\eta\colon \tilde \cB_{GL_n}\to \tilde \cB_{GL_n}$
from Theorem \ref{dualitythm}  preserve the relations in 
\eqref{Heckerelation} to provide involutions
\begin{equation}
\iota\colon \tilde H_{GL_n}\to \tilde H_{GL_n}
\qquad\hbox{and}\qquad
\eta\colon \tilde H_{GL_n}\to \tilde H_{GL_n}.
\label{Heckedualities}
\end{equation}

The following proposition explains how $g, T_1, \ldots, T_n$ move past the $X_1,\ldots, X_n$
inside the affine Hecke algebra.

\begin{prop} \label{GLnpolyaction} 
Let $\mu = (\mu_1, \ldots, \mu_n)\in \ZZ^n$ and 
define $X^\mu = X^{\mu_1\varepsilon_1+\cdots+\mu_n\varepsilon_n}$.  
The symmetric group $S_n$ acts on $\ZZ^n$ by permuting coordinates.
Let $s_1, \ldots, s_{n-1}$ be the simple reflections in $S_n$.
Then, as elements of  $\tilde H_{GL_n}$,
$$gX^\mu 
= q^{-\mu_n}X^{s_1s_2\cdots s_{n-1}\mu}g
= q^{-\mu_n}X^{(\mu_n,\mu_1, \ldots, \mu_{n-1})} g
,
$$
and
$$T_iX^\mu = (s_iX^{\mu})T_i + \frac{t^{\frac12}-t^{-\frac12}}{1-X_iX_{i+1}^{-1}}(1-s_i)X^\mu,
\quad\hbox{for $i\in \{1, \ldots, n-1\}$}.$$
\end{prop}
\begin{proof}
Start with $X_{i+1}=T_iX_iT_i$ and use $T_i^{-1} = T_i-(t^{\frac12}-t^{-\frac12})$ to get
\begin{align*}
T_iX_i &= X_{i+1}T_i^{-1} = X_{i+1}(T_i - (t^{\frac12}-t^{-\frac12}))
=X^{s_i\varepsilon_i}T_i + (t^{\frac12}-t^{-\frac12})\frac{X_i-X_{i+1}}{1-X_iX_{i+1}^{-1}} \\
&=X^{s_i\varepsilon_i}T_i + \frac{(t^{\frac12}-t^{-\frac12})}{1-X_iX_{i+1}^{-1}}(1-s_i)X_i.
\end{align*}
and
\begin{align*}
T_iX_{i+1} &= T_i^2X_iT_i = ((t^{\frac12}-t^{-\frac12})T_i+1)X_i T_i
=X_iT_i + (t^{\frac12}-t^{-\frac12})X_{i+1} \\
&=X^{s_i\varepsilon_{i+1}}T_i + (t^{\frac12}-t^{-\frac12})\frac{X_{i+1}-X_i}{1-X_iX_{i+1}^{-1}} \\
&=(s_iX_{i+1})T_i + \frac{t^{\frac12}-t^{-\frac12}}{1-X_iX_{i+1}^{-1}} (1-s_i)X_{i+1}
\end{align*}
and, for $j\not\in \{i, i+1\}$,
$\displaystyle{
T_iX_j = X_jT_i = (s_iX_j)T_i+0 = (s_iX_j)T_i + \frac{t^{\frac12}-t^{-\frac12}}{1-X_iX_{i+1}^{-1}} (1-s_i)X_j.
}
$
\hfil\break
If
$$
T_iX^\mu = (s_iX^{\mu})T_i + \frac{t^{\frac12}-t^{-\frac12}}{1-X_iX_{i+1}^{-1}}(1-s_i)X^\mu
\quad\hbox{and}\quad
T_iX^\nu = (s_iX^{\nu})T_i + \frac{t^{\frac12}-t^{-\frac12}}{1-X_iX_{i+1}^{-1}}(1-s_i)X^\nu
$$
then
\begin{align*}
T_iX^{\mu+\nu} &= T_iX^\mu X^\nu = \left(
(s_iX^{\mu})T_i + \frac{t^{\frac12}-t^{-\frac12}}{1-X_iX_{i+1}^{-1}}(1-s_i)X^\mu\right) X^\nu
\\
&= (s_iX^\mu) \left( (s_iX^{\nu})T_i + \frac{t^{\frac12}-t^{-\frac12}}{1-X_iX_{i+1}^{-1}}(1-s_i)X^\nu \right)
+ \frac{t^{\frac12}-t^{-\frac12}}{1-X_iX_{i+1}^{-1}}(X^\mu-(s_iX^\mu)) X^\nu
\\
&= (s_iX^{\mu+\nu}) T_i + \frac{t^{\frac12}-t^{-\frac12}}{1-X_iX_{i+1}^{-1}}
(X^{s_i\mu+\nu} - X^{s_i(\mu+\nu)} + X^{\mu+\nu} - X^{s_i\mu+\nu})
\\
&= (s_iX^{\mu+\nu}) T_i + \frac{t^{\frac12}-t^{-\frac12}}{1-X_iX_{i+1}^{-1}}
(1-s_i)X^{\mu+\nu}.
\end{align*}
Since $gX_n g^{-1} = q^{-1}X_1$ and $gX_ig^{-1} = X_{i+1}$ then
$$gX^{(\mu_1, \ldots, \mu_n)} = gX_1^{\mu_1}\cdots X_n^{\mu_n}
=X_2^{\mu_1}\cdots X_n^{\mu_{n-1}}gX_n^{\mu_n}
=q^{-\mu_n}X_2^{\mu_1}\cdots X_n^{\mu_{n-1}}X_1^{\mu_n}g.
$$
\end{proof}

\subsection{Intertwiners}

Structurally, the elements $X^{\varepsilon_1}, \ldots, X^{\varepsilon_n}$ are playing the
role of generators of a polynomial ring inside of the double affine Hecke algebra $\tilde H_{GL_n}$.
The next key point is that we can produce elements  $\tau^\vee_1, \ldots, \tau^\vee_{n-1}$
and $\tau^\vee_\pi$
which are ``replacements'' for the generators $T_1, \ldots, T_n$ and $g^\vee$, 
and which move past the elements $Y^{\varepsilon^\vee_1},
\ldots, Y^{\varepsilon^\vee_n}$ in the best possible way, by permuting the $Y_i$, as seen in
\eqref{Ypasttau}.

Define $Y_i$ for $i\in \ZZ$ by setting
\begin{equation}
Y_i = Y^{\varepsilon_i^\vee}\ \ \hbox{for $i\in \{1, \ldots, n\}$}
\qquad\hbox{and}\qquad Y_{j+n} = qY_j\ \ \hbox{for $j\in \ZZ$.}
\label{YiforiinZ}
\end{equation}
Letting
\begin{equation}
Y^{K}=q^{-1}
\quad\hbox{and}\quad \varepsilon^\vee_0 = \varepsilon^\vee_n+K
\quad\hbox{then}\quad Y_0 = Y^{\varepsilon^\vee_0} = Y^{\varepsilon^\vee_n+K} = 
Y^K Y^{\varepsilon^\vee_n} = q^{-1}Y_n.
\label{Ksource}
\end{equation}

Let 
\begin{equation}
\tau_\pi^\vee = g^\vee
\qquad\hbox{and}\qquad 
\tau_i^\vee = T_i + \frac{t^{-\frac12}(1-t)}{1-Y_i^{-1}Y_{i+1}}\quad 
\hbox{for $i\in \{1, \ldots, n-1\}$.}
\label{GLnintertwiners}
\end{equation}
For $w\in W$ define
$$\tau_w^\vee = \tau_{i_1}^\vee \cdots \tau_{i_\ell}^\vee
\qquad\hbox{for a reduced word $w= s_{i_1}\cdots s_{i_\ell}$.}
$$
The following proposition establishes that the $\tau^\vee_i$ satisfy the braid relations
so that the element $\tau^\vee_w$ does not depend on the choice of reduced word for
$w$.  The $\tau^\vee_i$ do not quite generate a symmetric group, because $(\tau^\vee_i)^2$
is not the identity.

\begin{prop} \label{intertwiners} 
For $i\in \{1, \ldots, n-2\}$ and $j,k\in \{1, \ldots, n-1\}$ with $k\not\in\{ j+1, j-1\}$,
\begin{equation}
\tau^\vee_\pi \tau_i^\vee = \tau_{i+1}^\vee \tau^\vee_\pi,
\qquad
\tau_i^\vee\tau^\vee_{i+1}\tau_i^\vee = 
\tau_{i+1}^\vee \tau^\vee_i \tau_{i+1}^\vee
\qquad\hbox{and}\qquad
\tau_k^\vee \tau_j^\vee = \tau_j^\vee \tau_k^\vee ;
\label{taubraidrels}
\end{equation}
\begin{equation}
(t^{\frac12}\tau_i^\vee)^2 =
\frac{(1-tY_i^{-1}Y_{i+1} )(1-tY_iY_{i+1}^{-1}) }
{(1-Y_i^{-1}Y_{i+1} )(1-Y_iY_{i+1}^{-1}) },
\qquad\hbox{for $i\in \{1, \ldots, n-1\}$;}
\label{tausquared}
\end{equation}
\begin{equation}
Y_i \tau_w^\vee = \tau_w^\vee Y_{w^{-1}(i)}
\qquad\hbox{for $w\in W$ and $i\in \ZZ$.}
\label{Ypasttau}
\end{equation}
\end{prop}

\begin{proof}
Using $T_i = T_i^{-1}+(t^{\frac12}-t^{-\frac12})$,
\begin{align}
\tau_i^\vee 
&= T_i + \frac{t^{-\frac12}(1-t)}{1-Y_i^{-1}Y_{i+1}}
= (T_i^{-1}+(t^{\frac12}-t^{-\frac12}) + \frac{t^{-\frac12}(1-t)}{1-Y_i^{-1}Y_{i+1}}
\nonumber
\\
&
= T_i^{-1} + \frac{ (Y_i^{-1}Y_{i+1}-1+1)t^{-\frac12}(1-t)}{1-Y_i^{-1}Y_{i+1}} 
= T_i^{-1} + \frac{t^{-\frac12}(1-t)Y_i^{-1}Y_{i+1}}{1-Y_i^{-1}Y_{i+1}}.
\label{GLnintertwineralternative}
\end{align}
To prove \eqref{Ypasttau}, prove that
\begin{equation}
Y_1\tau_\pi^\vee = q^{-1}\tau_\pi^\vee Y_n
\quad\hbox{and}\quad
Y_i\tau_\pi^\vee = \tau_\pi^\vee Y_{i-1} \quad\hbox{for $i\in \{2, \ldots n\},\quad and$}
\label{rootaction}
\end{equation}
$$
Y_i\tau_i^\vee = \tau_i^\vee Y_{i+1}, \qquad
Y_{i+1}\tau_i^\vee = \tau_i^\vee Y_i \qquad \hbox{and}\qquad Y_k\tau_i^\vee =\tau_i^\vee Y_k,$$
for $i\in \{1, \ldots, n-1\}$ and $k\in \{1, \ldots, n\}$ with $k\not\in \{i, i+1\}$.
By \eqref{TY2GL} and \eqref{pidefn},
$$\tau^\vee_\pi Y^{\varepsilon^\vee_n} = qY^{\varepsilon_1}\tau^\vee_\pi 
\quad\hbox{gives}\quad
Y_1\tau^\vee_\pi = \tau^\vee_\pi  q^{-1}Y_n  = \tau^\vee_\pi Y_0 = \tau^\vee_\pi Y_{\pi^{-1}(1)}.
$$
and $Y^{\varepsilon^\vee_{i+1}}g^\vee = g^\vee Y^{\varepsilon^\vee_i}$ for $i\in \{1, \ldots, n-1\}$.
Using $Y_{i+1} = T_i^{-1}Y_iT_i^{-1}$,
\begin{align*}
\tau_i^\vee Y_i 
&= \left( T_i^{-1} + \frac{t^{-\frac12}(1-t)Y_i^{-1}Y_{i+1}}{1-Y_i^{-1}Y_{i+1}}\right)Y_i 
= \left( Y_{i+1}T_i + \frac{t^{-\frac12}(1-t)Y_{i+1}}{1-Y_{i+1}}\right) 
= Y_{i+1}\tau^\vee _i,
\quad\hbox{and} \\
\tau_i^\vee Y_{i+1}
&= \left( T_iY_{i+1} + \frac{t^{-\frac12}(1-t)Y_{i+1}}{1-Y_{i+1}}\right)  
= Y_i\left( T_i^{-1} + \frac{t^{-\frac12}(1-t)Y_i^{-1}Y_{i+1}}{1-Y_i^{-1}Y_{i+1}}\right) 
=Y_i \tau_i^\vee.
\end{align*}
If $k\not\in \{i, i+1\}$ then $T_iY_k = Y_kT_i$ and $Y_iY_k = Y_kY_i$ and $Y_{i+1}Y_k=Y_kY_{i+1}$
and so
$$\tau_i^\vee Y_k
= \left( T_i + \frac{t^{-\frac12}(1-t)}{1-Y_i^{-1}Y_{i+1}}\right)Y_k
= Y_k \left( T_i + \frac{t^{-\frac12}(1-t)}{1-Y_i^{-1}Y_{i+1}}\right)
=Y_k \tau_i^\vee.
$$
Using \eqref{GLnintertwineralternative},
\begin{align*}
(\tau_i^\vee)^2
&= \left( T_i + \frac{t^{-\frac12}(1-t)}{1-Y_i^{-1}Y_{i+1}}\right)\tau_i^\vee 
=  T_i\tau_i^\vee 
+ \tau_i^\vee \frac{t^{-\frac12}(1-t)}{1-Y_{i+1}^{-1}Y_i} \\
&= T_i\left( T_i^{-1} + \frac{t^{-\frac12}(1-t)Y_i^{-1}Y_{i+1}}{1-Y_i^{-1}Y_{i+1}}\right)
+\left( T_i + \frac{t^{-\frac12}(1-t)}{1-Y_i^{-1}Y_{i+1}}\right)
\frac{t^{-\frac12}(1-t)}{1-Y_{i+1}^{-1}Y_i} \\
&= 1+ T_i \frac{t^{-\frac12}(1-t)Y_i^{-1}Y_{i+1}}{1-Y_i^{-1}Y_{i+1}}
+ T_i  \frac{t^{-\frac12}(1-t)}{1-Y_{i+1}^{-1}Y_i} 
+ \left(\frac{t^{-\frac12}(1-t)}{1-Y_i^{-1}Y_{i+1}}\right)
\frac{t^{-\frac12}(1-t)}{1-Y_{i+1}^{-1}Y_i}  \\
&=
\frac{(1-Y_i^{-1}Y_{i+1})(1-Y_{i+1}^{-1}Y_i) + t^{-1}-2+t}
{(1-Y_i^{-1}Y_{i+1} )(1-Y_{i+1}^{-1}Y_i) } 
=
\frac{(t^{-\frac12}-t^{\frac12}Y_i^{-1}Y_{i+1} )(t^{-\frac12}-t^{\frac12}Y_iY_{i+1}^{-1}) }
{(1-Y_i^{-1}Y_{i+1} )(1-Y_iY_{i+1}^{-1}) }.
\end{align*}
The proof of the relations in \eqref{taubraidrels} can be done by comparing
the brute force expansion of each side using the relations in \eqref{TY1GL}
and \eqref{TY2GL}.  An alternative, often used, argument is to note that the action of each
side on the polynomial representation (which is a faithful representation of $\widetilde{H}$) 
produces the same output (see Proposition 2.14(e) in \cite{Ra03}).
\end{proof}

\subsection{The polynomial representation}
\label{GLnpolyderiv}

In this section we build the action of the double affine Hecke algebra on Laurent
polynomials in $X_1, \ldots, X_n$.  The elements $Y_1, \ldots, Y_n$ are then
a large family of commuting elements acting on the polynomial ring 
$\CC[X_1^{\pm 1},\ldots, X_n^{\pm1}]$.  The Macdonald polynomials are the 
simultaneous eigenvectors for the family of commuting elements $Y_1, \ldots, Y_n$.

Let $q,t\in \CC^\times$ such that $1\not\in \{q^at^b\ |\ \hbox{$a,b\in \ZZ$ and 
$a$ and $b$ not both $0$}\}$.
The polynomial representation is
$$\CC[X] = \CC[X_1^{\pm 1},\ldots, X_n^{\pm1}]
= \hbox{$\CC$-span}\{ X^\mu\mathbf{1}\ |\ \mu\in \ZZ^n\}$$
with the action of DAHA determined by
$T_i\mathbf{1} = t^{\frac12}\mathbf{1}$ and
$g\mathbf{1} = \mathbf{1}$
so that, by \eqref{GLXYdefn},
\begin{equation}
Y^{\varepsilon_1^\vee}\mathbf{1} =
t^{\frac12(n-1)}\mathbf{1}
\qquad\hbox{and}\qquad
Y^{\varepsilon_i}\mathbf{1} = t^{\frac12(n-1)-2(i-1)\frac12}\mathbf{1}
= t^{-(i-1)+\frac12(n-1)}\mathbf{1}.
\label{Yion1}
\end{equation}
Following the notation of \cite[Ch.\ VI (3.1)]{Mac},
let $T_{q^{-1},X_n}$ be the operator on $\CC[X_1^{\pm1}, \ldots, X_n^{\pm1}]$ given by
$$(T_{q^{-1},X_n} h)(X_1, \ldots, X_n) = h(X_1, \ldots, X_{n-1}, q^{-1}X_n).$$

\begin{prop} \label{GLnpolyrepops}
As operators on the polynomial representation
$$g = s_1s_2\cdots s_{n-1} T_{q^{-1}, X_n}
\qquad\hbox{and}\qquad
T_i =t^{-\frac12} \big( t -\frac{tX_i - X_{i+1}}{X_i-X_{i+1}}(1-s_i) \big)
\qquad\hbox{for $i\in \{1, \ldots, n-1\}$.}
$$
\end{prop}
\begin{proof}  
The first statement in Proposition \ref{GLnpolyaction} gives
$gX^\mu \mathbf{1} = (s_1s_2\cdots s_{n-1} T_{q^{-1}X_n} X^\mu)\mathbf{1},$
since $T_{q^{-1},X_n} X^\mu = q^{-\mu_n}X^\mu$.
Using the second statement in Proposition \ref{GLnpolyaction},
\begin{align}
T_iX^\mu \mathbf{1}
&= \Big( (s_iX^\mu)t^{\frac12} + \frac{t^{\frac12}-t^{-\frac12}}{1-X_iX_{i+1}^{-1}} (1-s_i)X^\mu \Big) \mathbf{1} 
\label{origpolyop} \\
&=\left(t^{\frac12}-t^{\frac12}(1-s_i) + \frac{t^{\frac12}-t^{-\frac12}}{1-X_iX_{i+1}^{-1}} (1-s_i)\right) X^\mu \mathbf{1} 
\nonumber \\
&=\left( t^{\frac12} +\frac{1}{1-X_iX_{i+1}^{-1}}
(X_iX_{i+1}^{-1}t^{\frac12} - t^{-\frac12})(1-s_i) \right) X^\mu \mathbf{1}
\nonumber \\
&=\left( t^{\frac12} +\frac{1}{X_i-X_{i+1}}
(-X_it^{\frac12} + X_{i+1}t^{-\frac12})(1-s_i) \right) X^\mu \mathbf{1}
\nonumber \\
&=t^{-\frac12} \left( t -\frac{tX_i - X_{i+1}}{X_i-X_{i+1}}(1-s_i) \right) X^\mu \mathbf{1}.
\nonumber 
\end{align}
\end{proof}

\subsection{Constructing the nonsymmetric Macdonald polynomials $E_\mu$}

The Macdonald polynomials are the 
simultaneous eigenvectors for the action of $Y_1, \ldots, Y_n$
on the polynomial ring $\CC[X_1^{\pm1}, \ldots, X_n^{\pm1}]$.  Because the 
intertwiners $\tau^\vee_1, \ldots, \tau^\vee_{n-1}, \tau^\vee_\pi$ 
move past $Y_1, \ldots, Y_n$ in the best possible way, they are the
perfect tools for physically constructing the Macdonald polynomials $E_\mu$.

\begin{prop} \label{KSGLnformularepeat}
Let $\mu\in \ZZ_{\ge 0}^n$ and let $u_\mu$ and $v_\mu$ be as
in \eqref{umudefn} and let
$\ell(v_\mu)$ be the number of inversions of $v_\mu$. 
Choose a reduced word $\vec u_\mu = s_{i_1}\cdots s_{i_\ell}$ (where 
$i_1, \ldots, i_\ell\in \{\pi, 1, \ldots, n-1\}$ and let
$\tau_{u_\mu}^\vee = \tau_{i_1}^\vee \cdots \tau_{i_\ell}^\vee$. 
Define
$$E_\mu = t^{-\frac12\ell(v^{-1}_\mu)}\tau_{u_\mu}^\vee \mathbf{1}.
\qquad\hbox{Then}\quad
Y_i E_\mu = q^{-\mu_i}t^{-(v_\mu(i)-1)+\frac12(n-1)}E_\mu,$$
for $i\in \{1, \ldots, n\}$, and  the coefficient of $x^\mu$ in $E_\mu$ is 1.
\end{prop}
\begin{proof} 
Compute the eigenvalue as follows:
\begin{align*}
Y_i E_\mu
&= Y_i t^{-\frac12\ell(v^{-1}_\mu)}\tau_{u_\mu}^\vee \mathbf{1}
\qquad \hbox{(by definition of $E_\mu$)} \\
&= t^{-\frac12\ell(v^{-1}_\mu)} \tau_{u_\mu}^\vee Y_{u_\mu^{-1}(i)} \mathbf{1}  
\qquad \hbox{(by \eqref{Ypasttau})}  \\
&= t^{-\frac12\ell(v^{-1}_\mu)} \tau_{u_\mu}^\vee Y_{v_\mu t_\mu^{-1}(i)} \mathbf{1}  
\qquad \hbox{(by \eqref{umudefn})}  \\
&= t^{-\frac12\ell(v^{-1}_\mu)} \tau_{u_\mu}^\vee Y_{v_\mu (i-n\mu_i)} \mathbf{1} 
\qquad \hbox{(by \eqref{tmudefn})} \\
&= t^{-\frac12\ell(v^{-1}_\mu)} \tau_{u_\mu}^\vee Y_{v_\mu (i)-n\mu_i} \mathbf{1} 
\qquad \hbox{(by \eqref{nperiodicdefn})} \\
&= t^{-\frac12\ell(v^{-1}_\mu)} \tau_{u_\mu}^\vee q^{-\mu_i} Y_{v_\mu (i)} \mathbf{1} 
\qquad \hbox{(by \eqref{YiforiinZ})} \\
&= q^{-\mu_i}t^{-(v_\mu(i)-1)+\frac12(n-1)} (t^{-\frac12\ell(v^{-1}_\mu)}\tau_{u_\mu}^\vee  \mathbf{1})
\qquad \hbox{(by \eqref{Yion1})} \\
&= q^{-\mu_i}t^{-(v_\mu(i)-1)+\frac12(n-1)}E_\mu
\qquad \hbox{(by definition of $E_\mu$).}
\end{align*}
Using \eqref{Xvdefn}
and \eqref{TtoXconversion},
the top term of the expansion of in  $t^{-\frac12\ell(v_\mu)}\tau^\vee_{u_\mu} \mathbf{1}$ 
is
$$t^{-\frac12\ell(v^{-1}_\mu)}X^{u_\mu}\mathbf{1} 
= t^{-\frac12\ell(v^{-1}_\mu)}X^{t_\mu v^{-1}_\mu} \mathbf{1}
= t^{-\frac12\ell(v^{-1}_\mu)}X^\mu T_{v^{-1}_\mu} \mathbf{1}
= t^{-\frac12\ell(v^{-1}_\mu)}X^\mu t^{\frac12\ell(v^{-1}_\mu)} \mathbf{1}
=X^\mu\mathbf{1} = x^\mu.
$$
\end{proof}

\subsection{Steps and symmetries of $E_\mu$}

The following Proposition establishes the inductive construction of the $E_\mu$
and the symmetries in \eqref{Emusymmremark}.
For examples of the $E_\mu$ see Proposition \ref{Emu1box}, which provides explicit formulas 
for the cases when $\mu$ has 1 or 2 boxes.

\begin{prop} \label{Emusymmetries}
Let $\mu = (\mu_1, \ldots, \mu_n)\in \ZZ^n$.
\item[(a)] If $i\in \{1, \ldots, n-1\}$ and $\mu_i>\mu_{i+1}$ then 
$E_{s_i\mu} = t^{\frac12}\tau_i^\vee E_\mu$.
\item[(b)] $E_{(\mu_n+1, \mu_1, \ldots, \mu_{n-1})}
=  t^{\#\{ i\in \{1, \ldots, n\}\ |\  \mu_i>\mu_n\} -\frac12(n-1)} \tau_\pi^\vee E_{(\mu_1, \ldots, \mu_n)}$.
\item[(c)] $E_{(\mu_n+1, \mu_1, \ldots, \mu_{n-1})}
=q^{\mu_n}x_1E_\mu(x_2, \ldots, x_n, q^{-1}x_1)$.
\item[(d)] $E_{(\mu_1+1, \ldots, \mu_n+1)} = x_1\cdots x_n E_{(\mu_1, \ldots, \mu_n)}$.
\item[(e)] $E_{(-\mu_n, \ldots, -\mu_1)}(x_1, \ldots, x_n;q,t) 
= E_\mu(x^{-1}_n, \ldots, x^{-1}_1;q,t)$.
\end{prop}
\begin{proof}
(a) Let $\mu = (\mu_1, \ldots, \mu_n)$ and let $i$ be such that $\mu_i>\mu_{i+1}$. \hfil\break
By Proposition \ref{wmuexplicit}(e), $\ell(v_{s_i\mu}) - \ell(v_\mu) = -1$, giving
$$E_{s_i\mu} 
= t^{-\frac12\ell(v^{-1}_{s_i\mu})} \tau^\vee_{u_{s_i\mu}} \mathbf{1}
= t^{-\frac12\ell(v^{-1}_{s_i\mu})} \tau^\vee_i \tau^\vee_{u_\mu} \mathbf{1}
= t^{-\frac12(\ell(v^{-1}_{s_i\mu})-\ell(v^{-1}_{\mu}))} \tau^\vee_i E_\mu
=  (t^{\frac12}\tau_i^\vee)E_\mu.$$
(b) The left hand side is $E_{\pi \mu}$ and
\begin{align*}
E_{\pi\mu} 
&= t^{-\frac12\ell(v^{-1}_{\pi\mu})} \tau^\vee_{u_{\pi\mu}} \mathbf{1}
= t^{-\frac12\ell(v^{-1}_{\pi\mu})} \tau^\vee_\pi \tau^\vee_{u_\mu} \mathbf{1}
= t^{-\frac12(\ell(v^{-1}_{\pi\mu})-\ell(v^{-1}_\mu))} \tau^\vee_\pi 
t^{-\frac12\ell(v^{-1}_\mu)}\tau^\vee_{u_\mu} \mathbf{1}
= t^{ -(v_\mu(n)-1) + \frac12(n-1)} \tau_\pi^\vee E_\mu.
\end{align*}
The result then follows from Proposition \ref{wmuexplicit}(b).

\smallskip\noindent
(c)  
The second relation in  \eqref{ggveecommutationDAArtrel}
and the second relation in \eqref{TY2GL} give $X_1g = g^\vee Y_n$.
Beginning with the right hand side of (b) and using $g^\vee Y_n = X_1g$ gives
\begin{align*}
&t^{ -(v_\mu(n)-1 + \frac12(n-1)} \tau_\pi^\vee E_\mu
= t^{ -(v_\mu(n)-1) + \frac12(n-1)} g^\vee E_\mu
= g^\vee q^{\mu_n} Y_n E_\mu \\
&\qquad = q^{\mu_n} x_1gE_\mu
= q^{\mu_n} x_1 s_1\cdots s_{n-1}T_{q^{-1}, x_n}E_\mu
=q^{\mu_n}x_1E_\mu(x_2, \ldots, x_n, q^{-1}x_1).
\end{align*}
(d) By \eqref{Xommutation}, $(\tau^\vee_\pi)^n = (g^\vee)^n = X_1\cdots X_n$ and so
$$E_{(\mu_1+1, \ldots, \mu_n+1)} = (\tau^\vee_\pi)^n E_{(\mu_1, \ldots, \mu_n)} 
=x_1\cdots x_n E_{(\mu_1, \ldots, \mu_n)}.
$$
(e) Let $\eta\colon \tilde H_{GL_n} \to \tilde H_{GL_n}$ be the involution in 
\eqref{Heckedualities}.  
Let $w_0$ be the longest element of
$S_n$ so that $w_0(i) = n-i+1$ for $i\in \{1, \ldots, n\}$.
Using the last relation in \eqref{dualitiesonXY},
\begin{align*}
Y_i \eta(E_\mu(X_1, \ldots, X_n)) \mathbf{1} 
&= \eta(Y^{-1}_{n-i+1} E_\mu(X_1,\dots, X_n))\mathbf{1}   \\
&= q^{\mu_{n-i+1}} t^{(v_\mu(n-i+1)-1)-\frac12(n-1)} \eta(E_\mu(X_1, \ldots, X_n))\mathbf{1} \\
&= q^{\mu_{n-i+1}} t^{(w_0v_{-w_0\mu} w_0 (n-i+1)-1)-\frac12(n-1)} \eta(E_\mu(X_1, \ldots, X_n))\mathbf{1}, \\
&= q^{-(-w_0\mu)_i } t^{(w_0v_{-w_0\mu} (i)-1)-\frac12(n-1)} \eta(E_\mu(X_1, \ldots, X_n))\mathbf{1}, \\
&= q^{-(-w_0\mu)_i } t^{(n-v_{-w_0\mu} (i)+1)-1-\frac12(n-1)} \eta(E_\mu(X_1, \ldots, X_n))\mathbf{1}, \\
&= q^{-(-w_0\mu)_i } t^{-(v_{-w_0\mu} (i)-1)+\frac12(n-1)} \eta(E_\mu(X_1, \ldots, X_n))\mathbf{1},
\end{align*}
so that $\eta(E_\mu(X_1, \ldots, X_n)) \mathbf{1} = E_\mu(x_n^{-1}, \ldots, x_1^{-1})$
satisfies the conditions (from Theorem \ref{KSGLnformularepeat}) determining $E_{-w_0\mu}(x_1, \ldots, x_n)$.
\end{proof}


\end{document}